%% file: main.tex
\documentclass[11pt,reqno]{amsart}
\input{packages}

\input{environments}

\usepackage{xcolor}
\usepackage{booktabs}

\DeclareMathOperator{\Aut}{Aut}
\DeclareMathOperator{\trop}{trop}

	\newcommand{\T}{\mathcal T}

\title{A refined twist on Hurwitz numbers}

\author[R.~Fesler]{Raphaël Fesler}
\address{R.~Fesler: Guangdong Technion-Israel Institute of Technology, Daxue lu 241, 515063, Shantou, Guangdong, P.R. of China}
\email{raphael.fesler@gtiit.edu.cn}

\author[M.~A.~Hahn]{Marvin Anas Hahn}
\address{M.~A.~Hahn: School of Mathematics 17, Westland Row, Trinity College Dublin, Dublin 2, Ireland}
\email{hahnma@maths.tcd.ie}

\author[M.~Karev]{Maksim Karev}
\address{M. Karev: Guangdong Technion-Israel Institute of Technology, Daxue lu 241, 515063, Shantou, Guangdong, P.R. of China}
\email{maksim.karev@gtiit.edu.cn}

\author[H.~Markwig]{Hannah Markwig}
\address{H.~Markwig: Universität Tübingen, Fachbereich Mathematik, Auf der Morgenstelle 10, 72076 Tübingen, Germany}
\email{hannah@math.uni-tuebingen.de}

\keywords{Hurwitz numbers, symmetric functions, tropical geometry}
\thanks{\emph{2020 Mathematics Subject Classification:}  05E10, 14T15, 05A15, 57M12.}

\begin{document}
\begin{abstract}
   We introduce a two-parameter refinement of the Jucys-Murphy theory, that we call the \textit{CJT}--refinement, unifying Schur, zonal, and, conjecturally, Jack actions of the ring of symmetric functions on the Fock space. Applications of this formalism include a partial resolution of a recent conjecture of Coulter--Do, as well as cut--and--join recursion for $b$--Hurwitz numbers. The cut--and--join equations enable the derivation of the tropicalization of $b$--Hurwitz numbers. We also provide a first application of this tropical interpretation by answering an open problem of Chapuy--Do\l{\k e}ga on the polynomial structure of $b$--Hurwitz numbers.
\end{abstract}
\maketitle

\tableofcontents
\section{Introduction}

In this introduction, we illuminate the background of the numerous research themes in Hurwitz theory that are relevant to our research. Readers with expertise in one or more of these areas may as well skip parts of the introduction, and move on e.g.\ to Section \ref{subsec:aims} where we discuss the aims and contribution of this work.

\subsection{Historical background}
 \emph{Hurwitz numbers} is an umbrella term encompassing a huge variety of invariants related to the enumeration of branched morphisms between Riemann surfaces. Historically the first variant of Hurwitz numbers was introduced by A. Hurwitz in the late 19th century \cite{hurwitz1891riemann} in an attempt to study the geometry of the moduli spaces of complex curves. These invariants have found great interest since the 1990s due to intimate connections to Gromov--Witten theory (\cite{ekedahl1999hurwitz,zbMATH01777217,okounkov2006gromov}), mirror symmetry (\cite{dijkgraaf1995mirror,bohm2017tropical,goujard2019counting}), random matrix theory (\cite{goulden2014monotone}) and many more. Despite its longevity, the Hurwitz numbers still remain a topic of intensive research,  lying in between algebraic geometry, representation theory, integrable systems theory, combinatorics, complex analysis and other subjects.

\subsection{Variants of Hurwitz numbers}

The interest to Hurwitz numbers in the late XX century was mainly related to so-called \emph{complex simple Hurwitz numbers,} that correspond to counting covers of the Riemann sphere with a certain restriction on the possible ramification profiles. 
\emph{Single} complex simple Hurwitz numbers allow one arbitrary ramification profile and simple ramification else. \emph{Double} complex simple Hurwitz numbers are obtained by allowing a second arbitrary ramification profile. We will give the precise definition of these invariants in Section~\ref{sec:repthr}.

Among the most striking aspects of complex simple Hurwitz numbers, taking in account their purely geometric origin, is their intimate relation to the representation theory of the symmetric group. Indeed, already in Hurwitz' original work simple Hurwitz numbers were expressed as a an enumeration of factorisations in the symmetric group. For single and double simple Hurwitz numbers, this yields an interpretation of Hurwitz numbers as a counting problem of factorising permutations into transpositions allowing to explicitly produce the recursive identities they obey (see~\cite{mulase2009polynomial} for the single case, and \cite{Zhu12} for the double). An alternative way to carry out the count of factorisations is given by the Frobenius formula (see Appendix of~\cite{Lando-Zvonkin}), and can be performed in terms of the irreducible characters of the symmetric group. A major development of the past decade was the introduction of several \textit{variants} of single and double Hurwitz numbers, such as monotone \cite{goulden2014monotone} or strictly monotone \cite{ALS16} Hurwitz numbers and many others \cite{do2018pruned,bouchard2024topological,hahn2015pruned}. Most of these variants are obtained by imposing additional restrictions on the factorisations we count, and also admit a representation-theoretic interpretation. They appear in various contexts, e.g. monotone Hurwitz numbers arise as the coefficients as the asymptotic expansion of the HCIZ integral in random matrix theory, while strictly monotone Hurwitz numbers are an enumeration of certain Grothendieck dessins d'enfants.

A huge part of the variants we mentioned falls into the family of so-called \emph{weighted} Hurwitz numbers. To introduce this family, we begin by setting the stage. It has long been known that appropriate generating series of Hurwitz numbers provide solutions (called $\tau$--functions) of integrable hierarchies \cite{zbMATH01539310,goulden2008kp}. These solutions are of so--called \textit{hypergeometric type} \cite{orlov2001hypergeometric}. Working backwards, i.e. starting with hypergeometric $\tau$--functions and asking for an enumerative interpretation of its coefficients, we get the definition of the weighted complex Hurwitz numbers of \cite{Guay-Paquet-Harnad}. Weighted complex Hurwitz numbers are induced by a \textit{weight generating function} $G(z)$ with $G(0) = 1,$ that in some sense controls the ramification data (see Subsection \ref{subsec:symm} for an elaboration of the details). Different choices of this weight generating function recover most known variants of Hurwitz numbers, e.g. single and double Hurwitz numbers arise for $G(z)=e^z$, while $G(z)=\frac{1}{1-z}$ produces monotone and $G(z)=1+z$ strictly monotone Hurwitz numbers. Thus weighted complex Hurwitz numbers provide a unifying formalism for Hurwitz--type enumerations.

\subsection{Real Hurwitz numbers and Jack functions}
Motivated by the original Hurwitz problem, real analogues of Hurwitz numbers have been studied extensively \cite{itenberg2018hurwitz,Guay-Paquet-Markwig-Rau,natanzon2017bkp}.  Real Hurwitz numbers have various incarnations, but essentially they enumerate branched covers of Riemann surfaces with real involutions, such that the involutions are respected by the cover. The principal example is the study of covers of $(\mathbb CP^1,\mathcal J)$ with $\mathcal J$ being the complex conjugation map.

In this work, we focus on a version of real Hurwitz numbers introduced in \cite{Chapuy-Dolega} in the context of Jack functions. We refer to the discussion starting right before \cref{th:realHN} for precise definitions and details.

To explain the motivation, we note that the aforementioned hypergeometric $\tau$--functions  responsible for complex Hurwitz numbers may be expanded in Schur functions. It is well known that Jack functions are a one parameter deformation of Schur functions (we denote this parameter by $b$). The idea of \cite{Chapuy-Dolega} --- motivated by the enumeration of non-necessarily orientable \textit{dessins d'enfants} that appeared in \cite{goulden1996connection} --- was to replace Schur funtions by Jack functions in this expansion to obtain a $b$--deformation of hypergeometric $\tau$--functions. For $b=0$, the coefficients specialise to complex weighted Hurwitz numbers, while $b=1$ recovers their real analogue. This real analogue enumerates branched coverings with real critical values and empty real locus, see \cite{burman2024real} for a more detailed discussion. We refer to these covers as \textit{purely real}. Moreover, an enumerative interpretation for arbitrary $b$ in terms of so-called \emph{generalized covers} was proved in \cite{Chapuy-Dolega} as well --- these numbers are related to the count the quotients of purely real covers modulo the action of the real structure with a $b$--weight, such that the exponent of $b$ in a precise sense "measures" how non--orientable the quotient cover is. These $b$--deformed Hurwitz numbers are now called \textit{$b$--Hurwitz numbers}. 

\subsection{Symmetric functions and representation theory}
\label{subsec:symm} As mentioned before, complex Hurwitz numbers are intimately related to the representation theory of the symmetric group. The principal representation--theoretical instrument of the study of complex weighted Hurwitz numbers is the \emph{Jucys-Murphy formalism,} distilled in~\cite{Okounkov-Vershik}. The importance of Jucys-Murphy formalism to the Hurwitz theory was highlighted in~\cite{ALS16}.

This formalism allows to control the spectrum of the operators that arise in the Hurwitz theory. Its key feature is the special class of \emph{Jucys-Murphy elements} $X_k$ introduced in~\cite{Jucys,Murphy}. These are pairwise commuting elements of the group algebra of the symmetric group $\mathbb C[S_n]$, that act diagonally on the elements of the \emph{Gelfand-Zetlin} basis, and can be used to generate the centre of the group algebra of the symmetric group. Namely, for any symmetric polynomial $\mathrm F,$ the evaluation $\mathrm F(X_1,\ldots,X_n)\in\mathbb{C}[S_n]$ belongs to the centre of the group algebra $\mathbb C[S_n].$ Thus, the usual group algebra product defines the action of the ring of symmetric functions on the centre of the symmetric group algebra $S_n$: given a symmetric function $\mathrm F,$ and an element $\mathcal C \in Z\mathbb C[S_n]$ one defines the action $F.\mathcal{C}$ as $F(X_1,\ldots,X_n)\mathcal C\in Z\mathbb C[S_n]$. Not limiting ourselves to a fixed value of $n,$ one can define an action of a symmetric polynomial on a direct sum $\mathcal Z = \bigoplus_n Z\mathbb C[S_n].$ The space $\mathcal Z$ has a basis of conjugacy class indicators $C_\mu$ labelled by partitions $\mu,$ and a certain inner product. Thus it makes an instance of a \emph{Fock space} which we define as a graded vector space with basis labelled by all possible partitions, and provided with an inner product that makes these basis vectors orthogonal. We refer to the action of the ring of symmetric functions $\Lambda$ on $\mathcal Z$ as the~\emph{Schur action.} 

 Different complex weighted Hurwitz problems correspond to choices of families of symmetric functions $(\mathrm F_k)_{k = 0}^\infty$. As will be explained in more detail in \cref{sec:schur}, choosing $\mathrm F_k=\mathrm p_1^k$, where $\mathrm p_1$ is the first Newton power sum, yields  complex simple Hurwitz numbers, while $\mathrm F_k=\mathrm h_k$ the complete homogeneous symmetric polynomial and $\mathrm F_k=\mathrm e_k$ the elementary  symmetric polynomial recover monotone and strictly monotone Hurwitz numbers respectively. In the weighted Hurwitz numbers formalism, different choices of weight generating functions $G$ essentially correspond to a base change in the ring of symmetric functions $\Lambda,$ selecting different families of symmetric functions. Namely, a weight generating function $G$ gives rise to the following generating series of symmetric functions:
\begin{equation}\label{equ:weight}
    \mathcal G(\hbar; x_1,x_2,\ldots) = \sum_{k=0}^\infty \hbar^k \mathcal G_k(x_1,x_2,\ldots) = \prod_{j=1}^\infty G(\hbar x_j)\in \Lambda \otimes \mathbb C[[\hbar]].
\end{equation}

The symmetric functions $(\mathcal{G}_k)$ determine the Hurwitz problem corresponding to $G$. For partitions $\nu,\mu\vdash n$, and $k\in \mathbb N\cup \{0\}$ one defines the corresponding 
\emph{weighted possibly disconnected Hurwitz number} $H^{G,\bullet}_k\left(\begin{smallmatrix} \nu \\ \mu \end{smallmatrix}\right)$ as
\begin{equation}\label{equ:weighthurw}
H^{G,\bullet}_k\left(\begin{smallmatrix} \nu \\ \mu \end{smallmatrix}\right) = \frac 1{n!}\sum_{\sigma, C(\sigma) = \mu} [\sigma].\mathcal G_k(X_1,\ldots,X_n)\mathcal C_\nu,
\end{equation}

where $[\sigma].$ indicates the coefficient of $\sigma$ in $\mathcal G_k(X_1,\ldots,X_n)\mathcal C_\nu$, and $C(\sigma)$ refers to the cycle type of $\sigma$, i.e. the partition that  corresponds to its conjugacy class. As shown in \cite{Guay-Paquet-Harnad}, these numbers may be seen --- via monodromy representations --- to be a weighted enumeration of branched covers of $\mathbb{CP}^1$ with restricted ramification described by $G$.

This relation between weight generating functions and  enumeration of branched covers can work the both ways: either, provided with a family of symmetric functions $\{\mathcal G_k\}$, one can study the properties of the corresponding enumerative problem by working with the operators with the known spectrum (e.g. \cite{borot2023double,Alexandrov-Chapuy-Eynard-Harnad,bychkov2024topological}), or, given a weight generating function, one can explicitly work out the \emph{cut-and-join} recursions, allowing to compute the corresponding Hurwitz numbers (e.g. \cite{Zhu12,goulden2014monotone,Karev-Do}).\vspace{\baselineskip}

It is known that similarly to the complex case, the real covers count can be related to the combinatorics of the symmetric group (\cite{Guay-Paquet-Markwig-Rau,Lozhkin}). 
This connection was used in~\cite{burman2021ribbon,burman2024real}, to study the generating functions of the number of purely real covers. 
The purely real covers enumeration can be translated to the language of combinatorics of the symmetric group. 
Namely, using an equivalent language, we can state the correspondence of~\cite{burman2024real} between the purely real covers enumeration and factorisations count  in the following way. Assume that the group $S_{2n}$ is supported in the set $\{1,\bar1,\ldots,n,\bar n\}.$ Fix a partition $\mu = (1^{m_1}2^{m_2}3^{m_3}\ldots) \vdash d$. The weighted number  of non-necessary irreducible degree $2n$ real covers of $(\mathbb CP^1,\mathcal J)$ with one point of arbitrary ramification of profile $\mu\cup\mu = (1^{2m_1}2^{2m_2}3^{2m_3}\ldots)$ and $k$ critical values of profile $(1^{2n - 4}2^2)$ elsewhere, and  divided by $k!$\footnote{We impose this division to align with the definition of weighted Hurwitz numbers for $G = e^z$.}, can be computed as
\[
\mathcal H^{e^z,\bullet}_k\left( \begin{smallmatrix} 1^{n}\\ \mu \end{smallmatrix}\right)= \frac {1}{2^{n+k}n!k!}\# \left\lbrace (\sigma_1,\ldots,\sigma_k)\in S_{2n}^k\, |\, \begin{smallmatrix} C(\sigma_i) = (1^{2n-2}2^1),\sigma_i \ne  (j~\bar j)\, \forall j = 1,\ldots,d \\ C(\sigma_k\cdots\sigma_1 \tau \sigma_1\cdots \sigma_k)= \mu\cup \mu \end{smallmatrix} \right\rbrace,
\]
where $\tau = (1~\bar1)\cdots (n~\bar n)$ is a special fixed point-free involution in $S_{2n}$.  The explanation of the pieces of notation for $\mathcal H^{e^z,\bullet}_k\left( \begin{smallmatrix} 1^{n}\\ \mu \end{smallmatrix}\right)$ will be given in section~\ref{sec:repthr}.

The papers~\cite{goulden1996connection,hanlon1992some} reveal a connection of the enumeration of non-necessarily orientable dessins d'enfant to the representation theory of the Gelfand pair $(S_{2n},H_n),$ where $H_n$ is the hyperoctahedral group, which can be extended to the purely real covers enumeration question. Namely,~\cite{Matsumoto} introduces a notion of \emph{odd Jucys-Murphy} elements $\mathcal X_k$ that form nothing but "a half" of the collection of usual Jucys-Murphy elements in $S_{2n}$ (see section~\ref{sec:repthr} for the details). The symmetric group $S_{2n}$ (and the corresponding symmetric group algebra $\mathbb C[S_{2n}]$) act by the conjugation on the subspace of $\mathbb C[S_{2n}]$ spanned by all the fixed point-free involutions that we refer to as $M_n$. Using the odd Jucys-Murphy elements, one can restate the count of $\mathcal H^{e^z,\bullet}_k\left( \begin{smallmatrix} 1^{n}\\ \mu \end{smallmatrix}\right)$ in a way similar to Equation~\ref{equ:weighthurw}:
\begin{equation}\label{equ:weighthurwreal}
\mathcal H^{e^z,\bullet}_k\left( \begin{smallmatrix} 1^{n}\\ \mu \end{smallmatrix}\right) = \frac {1}{2^n n!k!}\sum_{\begin{smallmatrix}\rho\in S_{2n} \mbox{\small  a fixed point-free involution}\\ C(\rho\tau) = \mu\cup \mu\end{smallmatrix} }[\rho].\mathrm{p^k_1}(\mathcal X_1,\ldots,\mathcal X_n).\tau,
\end{equation}
where $\mathrm{p^k_1}(\mathcal X_1,\ldots,\mathcal X_n).\tau$ denotes the $k'$th power of the first Newton power sum evaluated on odd Jucys-Murphy elements acting on the special involution $\tau$ by the conjugation. It turns out, that this formula can be treated as an application of an another action of $\Lambda$ on an instance of a Fock space $I \le \bigoplus_n M_n$, which we will refer to as the~\emph{zonal action.} We will elaborate the details in Section~\ref{sec:repthr}. Here, we only notice, that the formula above can be used to produce a cut-and-join recursion, that can be used to compute the number of purely real covers inductively.


 The similarity of Equations~(\ref{equ:weighthurw}) and~(\ref{equ:weighthurwreal}) naturally sets the question of the existence of a~\emph{"Jack action"} of $\Lambda$ on a Fock space, that would set a base for counting of $b-$Hurwitz numbers in a representation-theoretical setup. First steps were taken in a recent preprint~\cite{Coulter-Do} of X.~Coulter and N.~Do, where they adapt the odd Jucys-Murphy elements introduced in~\cite{Matsumoto} to the~\emph{Weingarten calculus on real Grassmanians}, not mentioning, however, any relation to real covers enumeration. In particular, they provide computational evidence for conjecture, that formulates a precise Jucys--Murphy formalism for $b$--Hurwitz numbers (thus specialising to the purely real case for $b=1$) via a deformation of odd Jucys-Murphy elements of~\cite{Matsumoto}.  Namely, they invoke a certain graded vector space $\mathcal V_k$ defined over the field $\mathbb  C(b)$ and provided with a special element $\mathfrak e_k$  such that: 

\begin{conjecture}[Coulter-Do]
\label{conj-coulterdo}
For a fixed $n\in \mathbb N$ there exists a collection of operators  $\mathcal J_1,\ldots,\mathcal J_n \colon \mathcal V_k \to \mathcal V_k$ such that, being restricted to $\Xi_n = \left< \mathcal J_1,\ldots,\mathcal J_n\right>\cdot \mathfrak e_k$ the following properties hold:
\begin{itemize}
\item the restrictions of $\mathcal J_k$ for $k = 1,\ldots,n $ on $\Xi_n$ commute pairwise;
\item there exists a basis of\, $\Xi_n$ that diagonalizes all the $\mathcal J_k$ for $k = 1,\ldots,n$ simultaneously;
\item the vectors of the above-mentioned basis are labelled by standard Young tableau of size $n$ and satisfy an explicit recursive relation, allowing to construct them inductively by $n$.
\end{itemize}
\end{conjecture}

For a precise statement of the Coulter-Do conjecture, we refer to Conjecture 5.4 of their paper.

 Implicitly the Coulter-Do conjecture claims that there is an action of the ring of symmetric functions $\Lambda$ on a Fock space $\mathcal F$ with $\mathcal F_n \le \Xi_n.$, that refines both the Schur and the zonal actions, and allows to compute the weighted $b-$Hurwitz numbers of~\cite{Chapuy-Dolega} as outputs of cut-and-join recursions, arising from the combinatorics of the symmetric group.

\subsection{Centrality in Hurwitz theory}
We wish to highlight a key property for the use of Jucys-Murphy formalism for Hurwitz numbers enumeration. For this, we revisit \cref{equ:weighthurw}. Recall that $Z\mathbb{C}[S_n]$ has a vector space basis given by the conjugacy classes $C_\mu$ of $S_n$. Since by construction $\mathcal{G}_k$ is a symmetric function and thus by the Jucys--Murphy theorem $\mathcal{G}_k(X_1,\dots,X_n)$ lives in the center, there is an expression of $\mathcal{G}_k(X_1,\dots,X_n)$ as a linear combination of conjugacy classes. In particular, it follows that the coefficient of $\sigma$ in $\mathcal{G}_k(X_1,\dots,X_n)\mathcal{C}_\nu$ in \cref{equ:weighthurw} only depends on the conjugacy class of $\sigma$ and not on the permutation itself. We call this property \textit{centrality}. We note that the centrality property essentially follows from the fact that the Schur action is indeed an action of $\Lambda$ on $\mathcal Z$. Centrality gives better control on the cut-and-join analysis of the corresponding Hurwitz problem. For example, it was critically used in~\cite{Karev-Do} to produce the cut-and-join recursion for the complex orbifold monotone Hurwitz numbers.

It is natural to ask whether centrality extends to the purely real or $b$--deformed case. The zonal action naturally provides the centrality result for the former. 
We explain the full picture in \cref{Sc:Centrality}.
For $b$--monotone Hurwitz numbers the issue of the missing centrality was already highlighted in \cite{Bonzom-Chapuy-Dolega} in the cut--and--join analysis of these invariants. In some sense, \cref{conj-coulterdo} proposes to extend the centrality for zonal action to weighted $b$--Hurwitz numbers by producing a refinement. We will give a more detailed account of this problem in \cref{sec:b-hurwitz}.

\subsection{Tropical Hurwitz theory}
An exciting development in the past two decades was the introduction of tropical Hurwitz theory. Tropical geometry is a relatively new field of mathematics that at its core employs a process called \textit{tropicalization}. Tropicalization can be viewed as a degeneration process turning algebraic varieties into certain polyhedral complexes, which allows an exchange of methods between algebraic geometry and polyhedral geometry resp.\ combinatorics. Likewise, covers of algebraic curves (a.k.a. Riemann surfaces) can be degenerated to covers of abstract tropical curves, which are certain metric graphs. The resulting covers of metric graphs are called tropical covers, and their combinatorics reflects many properties of the original covers. In particular, in spite of the degeneration, enough structure is preserved to obtain results for counting problems also on the tropical side.  The basic version of tropical Hurwitz numbers counts certain covers of metric graphs with fixed properties \cite{CJM10,BBM10}.

There are many ways to obtain complex simple Hurwitz numbers as a weighted enumeration of tropical covers. E.g. while \cite{CJM10} employs the interpretation of (complex) double Hurwitz numbers via factorisations, \cite{BBM10} uses topological pairs--of--pants decompositions, in \cite{cavalieri2016tropicalizing}, a moduli theoretic framework is used and \cite{cavalieri2018graphical} is based on a Fock space formalism very much in the same flavour as the discussion above. Quite astoundingly, all these seemingly unrelated approaches produce the same tropical interpretation. The transition to the tropical world has proved particularly fruitful with regards to deriving structural properties of complex Hurwitz numbers with a particular focus on polynomiality results \cite{cjm:wcfdhn} and their relation to mirror symmetry.

Ever since the introduction of tropical Hurwitz theory, these techniques have been expanded to study variants of Hurwitz numbers, see e.g. \cite{Karev-Do,hahn2019monodromy,hahn2022tropical,hahn2020wall,hahn2022triply} a selection of works on the tropicalization of monotone and strictly monotone Hurwitz numbers. The full weighted complex Hurwitz problem will be tropicalised in forthcoming work \cite{HOW}. The applications of tropical techniques, i.e. the derivation of polynomiality results or quasimodularity in the context of mirror symmetry, extend to all weighted complex Hurwitz numbers.

Under close examination, all representation--theoretic approaches to tropicalising Hurwitz numbers "secretly" rely on centrality. Motivated by the above discussion, it is natural to ask for a tropical interpretation in the purely real or $b$--deformed case. Building on \cite{burman2021ribbon}, the second and fourth author derived tropicalizations of the purely real enumeration in \cite{HMtwisted1,HMtwistelliptic} (called twisted Hurwitz numbers in these works) corresponding to $G(z)=e^z$. Maybe unsurprisingly, one obtains tropical covers with involution. Again the tropical interpretation enabled the study of structural behaviour that indeed closely mirror the complex picture. Due to a lack of centrality, the general $b$--deformed case, as well as variants (e.g. $b$--monotone Hurwitz numbers) had not been studied before.

\subsection{Aims and contribution}\label{subsec:aims}
We begin by explaining on our (quite general) point of view on weighted Hurwitz theory.

For us, its ingredients are
\begin{itemize}
    \item an action of the ring of symmetric function $\Lambda$ on the Fock space $\mathcal F$, which is a space equipped with basis elements $v_\lambda$ for $\lambda$ a partition, and an inner product $\left(\cdot,\cdot\right)$, such that all the action operators are self-adjoint and the $v_\lambda$ are pairwise orthogonal;
    \item a weight generating function $G$.
\end{itemize}

We define the associated $G$--weighted Hurwitz number as
\begin{equation}
   H_k^{G,\bullet}\left(\begin{smallmatrix}\nu\\ \mu\end{smallmatrix}\right)=\left( v_\mu, \mathcal{G}_k.v_{\nu}\right),
\end{equation}
where the symmetric function $\mathcal G_k$ is obtained from $G$ by Equation~(\ref{equ:weight}). The corresponding connected Hurwitz numbers is defined via the inclusion-exclusion (see, e.g.~\cite{LandoICM,BDBKS20}, or any other paper that discusses disconnected Hurwitz numbers):
\begin{equation}\label{eq:dis-to-conn}
H^G_k\left(\begin{smallmatrix}\nu\\ \mu\end{smallmatrix}\right) =  \sum_{I = 1}^\infty \sum_{\begin{smallmatrix}
    \nu = \nu_1\cup\ldots\cup\nu_I\\ \mu = \mu_1\cup\ldots\cup\mu_I
\end{smallmatrix}} \frac{(-1)^{I-1}}{I}\sum_{\begin{smallmatrix}k = k_1 + \ldots + k_I\\ k_i \ge 0 \end{smallmatrix}} \prod_{i = 1}^I H^{G,\bullet}_{k_i}\left(\begin{smallmatrix}\nu_i\\ \mu_i\end{smallmatrix}\right),
\end{equation}
where the second sum ranges over the set of all possible representations of $\mu,\nu$ as a union of $I$ ordered subpartitions and it is assumed that for empty $\mu,\nu$ we have $H_k^{G,\bullet}\left(\begin{smallmatrix}\emptyset\\ \emptyset\end{smallmatrix}\right) =0 $ for all $k$. Similarly
\begin{equation}\label{eq:conn-to-dis}
H^{G,\bullet}_k\left(\begin{smallmatrix}\nu\\ \mu\end{smallmatrix}\right) =  \sum_{I = 1}^\infty \sum_{\begin{smallmatrix}
    \nu = \nu_1\cup\ldots\cup\nu_I\\ \mu = \mu_1\cup\ldots\cup\mu_I
\end{smallmatrix}}\frac{1}{|\mathrm{Aut(\nu_1,\ldots,\nu_I)\times \mathrm{Aut}(\mu_1,\ldots,\mu_I)|}} \sum_{\begin{smallmatrix}k = k_1 + \ldots + k_I\\ k_i \ge 0 \end{smallmatrix}} \prod_{i = 1}^I H^{G}_{k_i}\left(\begin{smallmatrix}\nu_i\\ \mu_i\end{smallmatrix}\right),
\end{equation}
where $\mathrm{Aut}(\mu_1,\ldots,\mu_I)$ stands for the automorphism group of collection of partitions.
Self-adjointness of the action operators implies that for all $\mu,\nu$ and for all $k$ we have
\[
H^{G,\bullet}_k\left(\begin{smallmatrix}\nu\\ \mu\end{smallmatrix}\right) = H^{G,\bullet}_k\left(\begin{smallmatrix}\mu\\ \nu\end{smallmatrix}\right),\quad H^{G}_k\left(\begin{smallmatrix}\nu\\ \mu\end{smallmatrix}\right) = H^{G}_k\left(\begin{smallmatrix}\mu\\ \nu\end{smallmatrix}\right).
\]

Classical complex weighted Hurwitz numbers then arise from the Schur action, while purely real weighted Hurwitz numbers correspond to the zonal action.

The main contribution of this work is the introduction of a refined Jucys--Murphy formalism that interpolates between the Schur and zonal actions, as well as conjecturally specialises to the weighted $b$--Hurwitz theory. More precisely, in \cref{sc:cjt-prep} we introduce a novel two parameter deformation of the zonal action  (that we call the CJT--refinement; the name will become clear from the construction) over $\mathbb{C}(C,J,T)$ that for $CJ=\frac{1}{2}$ and $T=0$ recovers the Schur action, while for $CJ=1$ and $T=1$ yields the zonal action (see below). We note that the construction itself appears as a three parameter deformation in $C$, $J$ and $T$, however the resulting Hurwitz numbers only depend on the values of $CJ$ and $T$, thus ultimately producing a two parameter deformation.

Our first main result is that this refined Jucys--Murphy formalism admits a generalisation of the centrality result in \cite{Matsumoto}, that can be restated as a definition of an action of $\Lambda$ on $I$. We prove the following --- for a precise statement, we refer to \cref{th:action}:

\begin{theorem*}
    The CJT--refinement of the zonal action defines an action of $\Lambda$ on $I\otimes\mathbb{C}(C,J,T)$.
\end{theorem*}

The proof of this result is obtained by an intricate combinatorial analysis of the commutation behaviour of odd Jucys--Murphy elements taking the refinement into account. In particular, we give a partial resolution of the aforementioned Conjecture 5.4 in \cite{Coulter-Do}.

In order to obtain a weighted Hurwitz theory as defined above, we need to show that the resulting operators are self--adjoint, i.e. for $\mathrm F.v_\nu$, $\mathrm F\in \Lambda$, denoting the refined action, we have $\langle v_\mu,\mathrm F.v_\nu\rangle=\langle \mathrm F.v_\mu,v_\nu\rangle$ for a certain inner product. For this, we introduce a refinement of the inner product corresponding to the zonal action and prove the following in \cref{th:sa}.

\begin{theorem*}
    Let $\mathrm F\in\Lambda$ be a symmetric function. Then, the refined action operator
    \begin{equation}
        \label{equ-action}
        v_\lambda\mapsto \mathrm F.v_\lambda
    \end{equation}
    is self--adjoint with respect to the refined inner product.
\end{theorem*}

The proof itself relies on a close analysis of the case $\mathrm F=\mathrm h_k$, where $\mathrm h_k$ is the complete homogeneous symmetric polynomial. This results in an analysis of refined monotone Hurwitz numbers. The monotone case is of particular importance since the $\mathrm h_k$ provide a multiplicative generating set of $\Lambda$ (for example, $\mathrm p_1^k$ that are responsible for the simple Hurwitz numbers case, do not form a multiplicative generating family).

Focusing on the monotone case for the moment, we note that we also conduct a cut--and--join analysis of refined monotone Hurwitz numbers, i.e. the refined weighted Hurwitz numbers corresponding to $G=\frac{1}{1-z}$. This analysis allows us to derive the following result in \cref{thm:structcoeff}.

\begin{theorem*}
    For $CJ=T=1$ refined weighted Hurwitz numbers specialise to purely real weighted Hurwitz numbers, while for $CJ=\frac{1}{2}$, $T=0$, we obtain complex weighted Hurwitz numbers.
\end{theorem*}

Based on what we discussed earlier about $b$--Hurwitz numbers, it is natural to ask whether for specific choices of $CJ$ and $T$, we recover a Jack action. Indeed, as we discuss in \cref{rem-eigenvec,rem-spec}, for $CJ=\frac{1+b}{2}$ and $T=b$ the eigenvectors of the action in \cref{equ-action} are (rescaled) Jack functions. However, in order to prove that our Hurwitz theory provides exactly the weighted $b$--Hurwitz numbers of \cite{Chapuy-Dolega}, we also need to describe the spectrum of these operators. This will be the topic of a separate publication.

We provide explicit evidence that our Hurwitz theory does indeed specialise to weighted $b$--Hurwitz numbers by working out the cases of double $b$--Hurwitz numbers ($G(z)=e^z$, i.e. $\mathrm F_k=\mathrm p_1^k/k!$) and single monotone $b$--Hurwitz numbers ($G(z)=\frac{1}{1-z}$, i.e. $\mathrm F_k=\mathrm h_k$) "by hand" in \cref{sec:simple}. This is achieved by explicitly deriving the cut--and--join recursion and comparing to the cut--and--join recursions derived (using different methods) in \cite{Chapuy-Dolega} for double $b$--Hurwitz numbers and single monotone $b$--Hurwitz numbers in \cite{Bonzom-Chapuy-Dolega}.

Based on this cut--and--join analysis, we derive tropical interpretations of double $b$--Hurwitz numbers and single monotone $b$--Hurwitz numbers in \cref{thm-corresdoubleHNb} and \cref{thm-tropmon} respectively. This generalises previous work of the second and fourth author in the case of $b=1$ in \cite{HMtwisted1}. The tropical covers involved are the same as those in the purely real case. This is not terribly surprising since the the deformed Hurwitz numbers are obtained as deformations of the $b=1$ case. We refer for details to \cref{sec:simple} but wish to highlight one application.

A famous result by Goulden, Jackson and Vakil in \cite{goulden2005towards} states for \emph{connected} complex double Hurwitz numbers $H^{e^z}_k\left(\begin{smallmatrix}\nu\\ \mu\end{smallmatrix}\right)$ (i.e. for the coefficients of the expancion of the logarithm of the $\tau$-function for $G(z)=e^z$, that in turn are responsible for the enumeration of \emph{connected} coverings) the following:

For $m,n$ positive integers, consider the parameter space $\mathcal{P}_{m,n}=\{(\mu,\nu)\in\mathbb{N}^m\times\mathbb{N}^n\mid \sum\mu_i=\nu_j\}$. Moreover, we define a hyperplane arrangement $\mathcal{R}_{m,n}$ in $\mathcal{P}_{m,n}$, the so-called \textbf{resonance arrangement}. For each $I\subset[m],J\subset$, we define a hyperplane via
\begin{equation}
    \{\textbf{x}\in\mathbb{Z}^n\mid \sum_{i\in I}\mu_i=\sum_{j\in J}\nu_j\}
\end{equation}
and define $\mathcal{R}_{m,n}$ as consisting of all these hyperplanes. We call each connected component of $\mathcal{P}_n\backslash\mathcal{R}_n$ a \textbf{chamber}. Notice, that the natural points of the resonance arrangement precisely correspond to pairs of partitions $\mu,\nu$ on which connected and not necessarily connected Hurwitz numbers can differ.

In \cite{goulden2005towards}, the following theorem was proved. (By an abuse of notation, an integer vector $\mu = (\mu_1,\ldots,\mu_m)$ is identified with the corresponding partition.)

\begin{theorem}
\label{thm-GJV}
    Let $n$ be a positive integer and $g$ a non-negative integer. Then, the map
\[
        h_g^0\colon\left( \begin{matrix} \mathcal{P}_n\to\mathbb{Q} \\ 
        (\mu,\nu)\mapsto H^{e^z}_{k(g,\mu,\nu)}\left(\begin{smallmatrix}
            \nu\\ \mu
        \end{smallmatrix}\right)\end{matrix}\right)
\]
    with $k(g,\mu,\nu) = 2g - 2 + m+ n$ is piecewise polynomial in the entries of $\mu$ and $\nu$. More precisely, the map $h_g^0$ restricted to each chamber is polynomial in the entries of $\mu$ and $\nu$.
\end{theorem}

In \cite{Chapuy-Dolega}, the piecewise polynomiality of double $b$--Hurwitz numbers was studied. In particular, for the analogous function
\[
    h_g^{(b)}\colon \left( \begin{matrix}
    \mathcal{P}_n\to\mathbb{Q}(b)\\
    (\mu,\nu) \mapsto H^{e^z}_{k(g,\mu,\nu)}\left(\begin{smallmatrix}
            \nu\\ \mu
        \end{smallmatrix};b\right)
        \end{matrix}\right),
\]
where $H^{e^z}_{k(g,\mu,\nu)}\left(\begin{smallmatrix}
            \nu\\ \mu
        \end{smallmatrix};b\right)$ stands for the corresponding simple $b$-Hurwitz number, 
it was proved that $(1+b)dh_g^{(b)}$ is a polynomial in $b$ whose coefficients are piecewise polynomial in the entries of $\mu$ and $\nu$ with respect to the resonance arrangement, where $d=|\mu|=|\nu|$. For $b=0$, this result is slightly weaker than the one obtained in \cref{thm-GJV}, since there is an extra factor of $d$. It was asked in \cite{Chapuy-Dolega}, whether the stronger result without this factor holds for arbitrary $b$. Using the tropical interpretation, this was settled affirmatively in \cite{HMtwisted1} for $b=1$. In \cref{thm-piecepoly}, we give the full answer for arbitrary $b$.

\begin{theorem*}
    The map $(1+b)h_g^{(b)}$ is a polynomial in $b$ whose coefficients are piecewise polynomials in the entries of $\mu$ and $\nu$.
\end{theorem*}

Continuing in \cref{sc:trop-mono} we study refined monotone Hurwitz numbers for arbitrary $C,J,T$ and derive a tropicalization in \cref{thm-refinedtropcorr}. We derive a piecewise polynomiality result for these numbers in \cref{prop-refinpoly} based on techniques previously employed in \cite{Karev-Do,hahn2019monodromy}. This result is weaker than the one we obtain for $b$--Hurwitz numbers, as we currently need a refinement of the resonance arrangement. Even in the classical setting, the analogous "naive" tropicalization of monotone Hurwitz numbers failed to detect the coarsest piecewise polynomial structure. Nonetheless, we view our result as evidence that refined monotone Hurwitz numbers exhibit the same underlying polynomiality behaviour as their classical counterparts.

Some computational data on the refined action may be found in \cref{sec:pdata}.

Finally, in \cref{sc:simple-mon} we explore the \emph{topological recursion} for single refined monotone Hurwitz numbers. The case of weighted $b$--Hurwitz numbers was previously shown to be solved by \textit{refined topological recursion} (for rational $G$--weight) in \cite{KO22,Ch-Dolega-Osuga2}. We show in the instance of $G(z)=\frac{1}{1-z}$ that this formalism generalises to the CJT--refinement in \cref{th:toprec}.

\subsection{Acknowledgements}
We thank S. Galkin, P. Nikitin, Ye. Makedonskyi, A. Liashyk, I. Sechin, E. Smirnov, M. Do\l{\k e}ga for many discussions and valuable comments. H.M. acknowledges support by the Deutsche Forschungsgemeinschaft (DFG, German Research Foundation), Project-ID 286237555, TRR 195.  

\section{Hurwitz numbers, representation theory, and tropical theory}\label{sec:repthr}
In this preliminary section, we introduce the basic notions surrounding Hurwitz numbers. We begin with the definition of single and double Hurwitz numbers in \cref{sc:sing-doub}. We continue by developing the representation theory of complex and purely real Hurwitz numbers in \cref{sec:symm,sec:schur,Sc:Centrality} and  by discussing the weighted $b$--Hurwitz numbers formalism in \cref{sec:b-hurwitz}. 
In Section \ref{sec:tropprelim} we introduce basic notions of tropical Hurwitz theory.

\subsection{Hurwitz numbers}\label{sc:sing-doub}
In this subsection, we introduce the classical notion of single and double complex simple Hurwitz numbers. 

\begin{definition}\label{def-Hurwitzcover}
Let $g$ be a non-negative integer, $d$ a positive integer and $\mu$ a partition of $d$. Moreover, we fix $k$ pairwise distinct points $p_1,\dots,p_k\in \mathbb C$ for $k=2g-2+\ell(\mu)+\ell(\nu)$. We call a holomorphic map $f\colon S\to\mathbb{P}^1$ a Hurwitz cover of type $(g,\mu,\nu)$ if
\begin{itemize}
	\item $S$ is a connected Riemann surface of genus $g$,
	\item $f$ is of degree $d$,
	\item the ramification profile of $0$ is $\mu$, of $\infty$ is $\nu$,
	\item the ramification profile of $p_i$ is $(2,1,\dots,1)$.
\end{itemize}

Further, we define an isomorphism of two Hurwitz covers $f_1\colon S_1\to\mathbb{P}^1,f_2\colon S_2\to\mathbb{P}^1$ of type $(g,\mu,\nu)$ to be an isomorphism $g\colon S_1\to S_2$, such that $f_1=f_2\circ g$.

We define \textit{double simple Hurwitz numbers} as
\begin{equation}
H_g(\begin{smallmatrix}\nu\\ \mu \end{smallmatrix};0)=\sum_{[f]}\frac{1}{|\mathrm{Aut}(f)|},
\end{equation}
where the sum runs over all isomorphism classes $[f]$ of Hurwitz covers of type $(g,\mu,\nu)$. Finally, we define \textit{single simple Hurwitz numbers} as
\begin{equation}
	H_g(\mu;0)\coloneqq H_g(\begin{smallmatrix}(1^d)\\ \mu \end{smallmatrix};0).
\end{equation}
\end{definition}

As already observed in Hurwitz' original work, these invariants are closely related to the representation theory of the symmetric group. In fact, Hurwitz numbers can be computed by the character theory of $S_d$.

The following theorem is essentially due to Hurwitz \cite{hurwitz1891riemann}. For this, we denote by $\mathcal{C}(\sigma)$ the \emph{cycle type} of a permutation $\sigma$, i.e. the partition corresponding to its conjugacy class.

\begin{theorem}
Let $g$ be a non-negative integer, $d$ a positive integer, $\mu,\nu$ partitions of $d$ and $k=2g-2+\ell(\mu)+\ell(\nu)$. We have
\begin{equation}
H_g(\begin{smallmatrix}\nu\\ \mu \end{smallmatrix};0)=\frac{1}{d!}\left\lbrace\begin{smallmatrix}(\sigma,\eta_1,\dots,\eta_k)\in S_d^{r+1}\mid \mathcal{C}(\sigma_1)=\nu,\mathcal{C}(\eta_i)=(2,1,\dots,1),\mathcal{C}(\eta_k\cdots\eta_1\sigma_1)=\mu \\ \mbox{the subgroup } \left< \sigma,\eta_1,\ldots,\eta_k\right> \mbox{ is transitive}\end{smallmatrix}\right\rbrace.
\end{equation}
\end{theorem}

The transitivity condition here is responsible for the connectedness of the covers we are enumerating. Lifting this condition lead to the definition of \emph{non-necessarily connected complex double simple Hurwitz numbers}, which we denote by $H^\bullet_g(\begin{smallmatrix}\nu\\ \mu \end{smallmatrix};0).$

\subsection{The ring of symmetric functions and the Fock space}
\label{sec:symm}
Let $x_1,\ldots,x_d$ be independent commuting variables. The $k'th$ homogeneous symmetric polynomial $\mathrm h_k$ evaluates on commuting variables $x_1,\ldots,x_d$ as follows:
\[
\mathrm h_k(x_1,\ldots,x_d) = 
    \sum_{1\le i_1\le \ldots \le i_k\le d} x_{i_1}\cdots x_{i_d}
\]
The Fundamental Theorem of symmetric polynomials (see, e.g.~\cite{Macdonald}) asserts that the ring of symmetric polynomials in $d$ variables $\mathbb C[x_1,\ldots,x_d]^{S_d}$ is freely generated by the homogeneous symmetric polynomials $\mathrm h_1,\ldots, \mathrm h_d$. In particular, it means that for $k> d,$ the value $\mathrm h_k(x_1,\ldots,x_d)$ can be expressed as a polynomial in $\mathrm h_1,\ldots,\mathrm h_d.$ Provide the ring $\mathbb C[\mathrm h_1,\ldots \mathrm h_d]$ with a grading $\mathrm{deg} (\mathrm h_k) = k.$  

\begin{definition}
    The ring of symmetric functions $\Lambda$ is the projective limit of the graded rings $\Lambda_d = \mathbb C[\mathrm h_1,\ldots,\mathrm h_d]$ under the projections induced by the specialization maps $x_d \mapsto 0.$ 
\end{definition}

The definition implies, that every symmetric function can be specialized on a finite number of variables. In the following, given a symmetric function $\mathrm F\in \Lambda$ and a finite number of commuting variables $x_1,\ldots,x_d$, by $\mathrm F(x_1,\ldots,x_d)$ we mean the result of the specialization.

It is well-known, that the ring of symmetric function also admits the bases $\{\mathrm e_\lambda\}$ of \emph{elementary symmetric functions,} $\{\mathrm p_\lambda\}$ of \emph{power-sum functions,} $\{\mathrm s_\lambda\}$ of \emph{Schur functions,} $\{ Z_\lambda\}$ of \emph{zonal functions,} and $\{J^{(b)}_\lambda\}$ of \emph{Jack functions}. These bases will appear in various places of our paper.

As a $\mathbb C$-vector space, the ring of symmetric functions $\Lambda$ has a basis $\mathrm h_\lambda = \mathrm h_{\lambda_1}\cdots \mathrm h_{\lambda_l}$ indexed by the set of partitions and aligned with the grading. As it is stated in the introduction, a lot of questions studied by the Hurwitz theory reduce to the study of the action of the ring of symmetric functions on a \emph{Fock space.}

\begin{definition}
    A Fock space $\mathcal F = \bigoplus_{n\ge 0}\mathcal F_n$  is a graded vector space  with homogeneous components $\mathcal F_n$, provided with a homogeneous basis indexed by the set of partitions $\lambda$ with $\lambda\vdash n,$ and an inner product that makes the basis vectors orthogonal.
\end{definition}

Any two Fock spaces as pairwise isomorphic as graded vector spaces. In particular, as a vector space, the ring of symmetric functions is isomorphic to any Fock space. However, it worth keeping in mind, that we distinguish them in the following way: the \emph{ring} $\Lambda$ \emph{acts} on the \emph{vector space} $\mathcal F$.


\subsection{Schur action and the  weighted complex Hurwitz numbers}
\label{sec:schur}
Let $n\in \mathbb N.$
The centre of the group algebra $\mathcal Z_n = Z\mathbb C[S_n]$ is an algebra of functions $f\colon S_n \to \mathbb C,$ which are constant on the conjugacy classes of $S_n$, so it can be also referred to as the \emph{algebra of conjugacy classes}. The product of $\mathcal Z_n$  is the convolution product. 

 Recall, that for a permutation $\sigma$, the symbol $C(\sigma)$ denotes its cycle type. The elements 
 \[
 \mathcal C_\mu\colon \left( \begin{matrix} S_n \to \mathbb C \\ \sigma \mapsto \begin{cases}
     1, \mbox{ if } C(\sigma) = \mu, \\ 0, \mbox{ otherwise.}
 \end{cases}\end{matrix}\right)
 \]
 for $\mu\vdash n$ chosen as a basis, span a realization of $n'$th homogeneous component of a Fock space with an inner product to be specified later. The distinguished basis elements are referred to as the \emph{class indicators.}


The \emph{Schur action} on the Fock space $\mathcal Z = \bigoplus_n \mathcal Z_n$ is defined as follows. Introduce the \emph{Jucys-Murphy elements} $X_i\in \mathbb C[S_n]$, $i = 1,\ldots,n$:
$$X_i = \sum_{k = 1}^{i-1} (k\, i),$$
which commute pairwise (notice that $X_1 = 0$). The following statement holds:
\begin{theorem}[see \cite{Jucys} and also \cite{Okounkov-Vershik}]
Let $\mathrm F \in \Lambda$ be a symmetric function. Then the result of evaluation $\mathrm F(X_1,\ldots X_n)$ is a central element of $\mathbb C[S_n].$ The map
\[
(\mathrm F,\mathcal C_\lambda) \mapsto \mathrm F(X_1,\ldots,X_{|\lambda|}) \mathcal C_\lambda,
\]
where on the right-hand part we have a convolution product of two elements of $\mathcal Z_{|\lambda|}$, defines an action of $\Lambda$ on $\mathcal F$.
\end{theorem}

It will be explained below, that for the fixed $n\in \mathbb N$, the action operators $\mathrm F(X_1,\ldots,X_n)$ for all $\mathrm F\in \Lambda$ are simultaneously diagonalizable. Their eigenvectors are proportional to the \emph{primitive central idempotents}.  
The primitive central idempotents are indexed by partitions, we denote the idempotent corresponding to a partition $\lambda$ by $ \check s_\lambda.$\footnote{If we identify the Fock space with the space of symmetric functions spanned by the power-sums, the elements $\check s_\lambda$ correspond to Schur functions.} Before describing the spectrum of the Schur action,  we require the following notion:
\begin{definition}\label{def:cont}
    Let $\lambda$ be a Young diagram. The Schur content $c_S(\square)$ of a box $\square \in \lambda$ is defined as follows. Draw $\lambda$ in the first quadrant of the plane with the horizontal coordinate axis $OX$ and the vertical coordinate axis $OY$ such that every box is a square with sides of length 1 aligned with the coordinate axes, and so that the center of the left-most bottom-most square is located in $(1,1)$. If the center of the box $\square$ has its center in the points of coordinates $(x,y),$ then we set the Schur content to be
    \[
    c_S(\square) = (x-1) - (y-1).
    \]
   
\end{definition}
 See Figure~\ref{fig:Schurcont} for an example.

\begin{figure}
\begin{tikzpicture}
    \def\cellSize{0.5cm}

    \foreach \y [count=\row] in {4, 2, 2, 1} {
        \foreach \x in {1, ..., \y} {
            \draw[thick] (\x*\cellSize, \row*\cellSize) rectangle ++(\cellSize, \cellSize);
            \pgfmathtruncatemacro{\value}{\x - \row} 
            \node at (\x*\cellSize + 0.5*\cellSize, \row*\cellSize + 0.5*\cellSize) {\value}; 
        }
    }
\end{tikzpicture}
\caption{A Young diagram and contents of its boxes.}\label{fig:Schurcont}
\end{figure}

\begin{theorem}[\cite{Jucys,Okounkov-Vershik}]\label{th:schur-spectrum}
For a symmetric function $\mathrm F\in \Lambda$ we have
\[
\mathrm F(X_1,X_2,\ldots)  \check s_\lambda = \mathrm F(\{c_S(\square)\}_{\square \in \lambda})  \check s_\lambda.
\]
\end{theorem}

As it was mentioned in Subsection~\ref{sc:sing-doub}, the algebra of conjugacy classes can be used for counting the complex covers of $\mathbb CP^1$ with a prescribed ramification data. 

Generalizing Definition \ref{def-Hurwitzcover}, by a \emph{ramified cover} of $\mathbb CP^1$ by a Riemann surface $\Sigma$ we mean a holomorphic map $f\colon \Sigma \to \mathbb CP^1$. Given such a map of degree $n$, every critical value $a\in \mathbb CP^1$ of $f$ gets endowed with the corresponding \emph{ramification profile} $r(a) = \mu \vdash n.$ We say, that two ramified covers $(\Sigma_1,f_1)$ and $(\Sigma_2,f_2)$ are equivalent, if there is a holomorphic homeomorphism $D\colon \Sigma_1 \to \Sigma_2,$ satisfying the \emph{equivariance property:} $f_2 \circ D = f_1.$ In this case, we say, that $D$ is an \emph{equivalence} of the ramified covers. The notion of equivalence gives rise to the group of automorphisms $\mathrm {Aut}(f)$ of the ramified cover $(\Sigma,f),$ that consists of self-equivalences.

For a ramified cover $(\Sigma,f)$ its \emph{weight} is the number $\frac 1{\mathrm{Aut}(f)}.$

A classical result (see e.g.~\cite{hurwitz1891riemann,Lando-Zvonkin} or also~\cite{ALS16,KarevFeynman}), states:

\begin{theorem}
	Given the ramification profiles $\mu_1,\ldots,\mu_k\vdash n$ over the points $p_1,\ldots,p_k \in \mathbb CP^1,$ the corresponding \emph{weighted} number of (not necessarily connected) ramified covers of $\mathbb CP^1$ ramified over $p_1,\ldots,p_k$ with the ramification profiles $\mu_1,\ldots,\mu_k$, can be computed as
\begin{equation}\label{Eq:Fro} \sum_{\begin{smallmatrix} [f] \\f \colon \Sigma\to \mathbb CP^1\\r(p_i) = \mu_i 
, i= 1,\ldots,n\\ f \mbox{ \small is unramified over } \mathbb CP^1\setminus \{p_1,\ldots,p_k
\}\end{smallmatrix}} \frac 1{|\mathrm{Aut}(f)|} = \mathrm{tr^{(0)}}(\mathcal C_{\mu_1} \cdots \mathcal C_{\mu_k}),\end{equation}
\end{theorem}

 where the sum ranges over the set of isomorphism classes $[f]$ of covers with the desired properties, $\mathcal C_{\mu_i}$ for $i = 1,\ldots,k$ is the corresponding class indicator, and  
 \[
 \mathrm{tr}^{(0)}(\mathcal C_{\mu}) = \begin{cases} \frac 1{n!}, \mbox{ if } \mu = (1^n),\\0, \mbox{ otherwise.} \end{cases}
 \] The trace function induces an inner product $\langle\cdot,\cdot\rangle_0$: 
 \[
 \langle \mathcal C_\mu, \mathcal C_\nu \rangle_0 = \delta_{\mu,\nu} \prod_{j\ge 1} \frac 1{j^{m_j}m_j!}
 \]
 for the partition $\mu = (1^{m_1}2^{m_2}\ldots).$ For any partition $\mu$, the value $\langle \mathcal C_\mu,\mathcal C_\mu\rangle_0$ equals the number of elements of conjugacy class $\mu$ in $S_{|\mu|}$ divided by $|\mu|!.$

 The operators of the Schur action are self-adjoint with respect to $\left< \cdot,\cdot \right>_0,$: the count does not depend on the order of the points $p_1,\ldots,p_k$. Hence, the primitive central idempotents $\mathrm s_\lambda$ are can be normalized to form an orthonormal basis in the Fock space.

The notion of (disconnected) weighted complex Hurwitz numbers (\cite{Alexandrov-Chapuy-Eynard-Harnad, Guay-Paquet-Harnad} specifies the idea of counting covers with prescribed ramification profiles. They can be defined using the Schur action we have introduced as follows. For a given \emph{weight generating function} $G(z)$ with $G(0) = 1$, consider the following element of $\mathcal Z\otimes \mathcal Z[[\hbar]]:$
\[
T^G(\hbar) = \sum_{n \ge 0}\sum_{\lambda\vdash n}  \check s_\lambda  \otimes  \check s_\lambda \prod_{\square \in \lambda} G(\hbar c_S(\square)),
\]
which, using Theorem~\ref{th:schur-spectrum} can be rewritten as
\[
T^G(\hbar) = \sum_\lambda \check  s_\lambda  \otimes\left[\prod_{k \ge 1} G(\hbar X_k)\right] \check s _\lambda.  
\]
The product $\left[\prod_{k \ge 1} G(\hbar X_k)\right]$ is a result of the evaluation of a symmetric function $\mathcal G \in \Lambda[[\hbar]]$ on the set of Jucys-Murphy elements. Using the equality
\[
 \check s_\lambda = \sum_\mu \langle \check  s_\lambda,\mathcal  C_\mu\rangle_0 \frac {\mathcal C_\mu}{\langle \mathcal C_\mu,\mathcal C_\mu\rangle_0}
\]
and the self-adjointness of the operator $\mathcal G(\hbar X_1,\ldots,\hbar X_n)$ we rewrite
\[
T^G(\hbar) = \sum_{n\ge 0}\sum_{\mu,\nu\vdash n} \langle \mathcal C_\mu, \mathcal G(\hbar X_1,\ldots \hbar X_n) \mathcal C_\nu\rangle_0 \frac {\mathcal C_\mu}{\langle \mathcal C_\mu, \mathcal C_\mu\rangle_0} \otimes \frac {\mathcal C_\nu}{\langle \mathcal C_\nu, \mathcal C_\nu\rangle_0}.
\]
\begin{definition}\label{def:cmplx}
   The coefficients $H^{G,\bullet}_k\left(\begin{smallmatrix}\nu\\\mu\end{smallmatrix};0\right) = [\hbar^k].\langle \mathcal C_\mu, \mathcal G(\hbar X_1,\ldots \hbar X_n) \mathcal C_\nu\rangle_0$ are referred to as the \emph{disconnected complex weighted Hurwitz numbers corresponding to the weight generating function $G$}\footnote{Given the unusual way to introduce the Schur function, the reader might find it not straightforward to relate our definition of the weighted Hurwitz numbers to the standard one (e.g. from~\cite{Alexandrov-Chapuy-Eynard-Harnad}). However, in Section~\ref{sc:ref-mon} we will see that the definitions are equivalent, as the corresponding \emph{action structural coefficients} coincide.}. 
\end{definition}
 These numbers are symmetric under the swap of $\mu$ and $\nu$. For several specific functions $ G$, the numbers bear special names:
\begin{itemize}
\item for $ G(z) = e^z$ we have \emph{simple Hurwitz numbers};
\item for $ G(z) = 1+ z$ we have \emph{strictly monotone Hurwitz numbers};
\item for $ G(z) = \frac 1{1 - z}$ we have \emph{monotone Hurwitz numbers}. 
\end{itemize}
This list is not exhaustive (see, e.g. \cite{BDBKS20} for a more populated zoo of different instances of weighted Hurwitz numbers; also, some more cases were considered in~\cite{Borot-Karev-Lewanski}).

The weighted Hurwitz numbers can be assembled in a slightly different generating function, taking values in the ring $\mathbb C[[\mathrm p_1,\mathrm p_2,\ldots;\mathrm q_1,\mathrm q_2,\ldots;\hbar]]$ with two infinite collections of free commuting multiplicative generators. Using the notation $\mathrm p_\lambda = \prod_{i} \mathrm p_{\lambda_i}$ we write
\[
\tau^{G} = \sum_{k\ge 0}\sum_{\nu,\mu} \hbar^k H_k^{G,\bullet}\left(\begin{smallmatrix}
    \nu\\ \mu 
\end{smallmatrix};0\right) \mathrm q_\nu \mathrm p_\mu,
\]
the coefficients of the expansion of the logarithm of $\tau^G$ are known as \emph{connected} (or \emph{transitive}) complex {Hurwitz numbers}:
\[
H_k^G\left(\begin{smallmatrix}
    \nu \\ \mu
\end{smallmatrix};0\right) = [\hbar^k \mathrm q_\nu \mathrm p_\mu] \log{\tau^G}.
\]
As the connected Hurwitz numbers are usually of much more interest compared to their disconnected counterpart (see, e.g.~\cite{ekedahl1999hurwitz,Karev-Do,ALS16,Borot-Karev-Lewanski}), the adjective 'connected' is often omitted. The inclusion-exclusion principle allows to express the connected numbers in terms of disconnected ones (see the Introduction)\footnote{As $\mathrm p_1(X_1,\ldots,X_n) = \mathcal C_{2^11^{n-2}}$, we see, that the case $G(z) = e^z$, up to the factorial of the number of simple branch points, corresponds to the simple Hurwitz numbers discussed in Subsection~\ref{sc:sing-doub}. The fact that the generating functions for connected and non-necessarily connected simple Hurwitz numbers are related by exponentiation is a standard exercise. }:
\[
H^G_k\left(\begin{smallmatrix}\nu\\ \mu\end{smallmatrix};0\right) =  \sum_{I = 1}^\infty \sum_{\begin{smallmatrix}
    \nu = \nu_1\cup\ldots\cup\nu_I\\ \mu = \mu_1\cup\ldots\cup\mu_I
\end{smallmatrix}} \frac{(-1)^{I-1}}{I}\sum_{\begin{smallmatrix}k = k_1 + \ldots + k_I\\ k_i \ge 0 \end{smallmatrix}} \prod_{i = 1}^I H^{G,\bullet}_{k_i}\left(\begin{smallmatrix}\nu_i\\ \mu_i\end{smallmatrix};0\right).
\]

Rewriting the generating function $T^G$ in terms of the inner products $\left< \mathcal C_\mu, \mathcal G(\hbar X_1,\ldots,\hbar X_n) \mathcal C_\nu\right>_0$ allows one to produce cut-and-join operators for the Hurwitz numbers based on the combinatorics of permutations. See, for example~\cite{Zhu12,do2018towards} for the cut-and-join for the simple Hurwitz numbers case, and~\cite{goulden2014monotone, Karev-Do} for the cut-and-join for single and orbifold monotone Hurwitz numbers. The cut-and-join operators allow us to develop the \emph{tropical count} of the corresponding Hurwitz numbers, which allows us to prove nice structural properties for them~(see, e.g.~\cite{cjm:wcfdhn, hahn2020wall}).

\emph{Monotone Hurwitz numbers} correspond to the weight generating function $G = (1 - z)^{-1}$. They make one of the most important examples of weighted Hurwitz numbers. The corresponding symmetric function $\mathcal G(\hbar X_1,\ldots,\hbar X_n) = \prod_{j = 2}^n \frac 1{1 - \hbar X_j},$  is the generating function for the sequence of homogeneous symmetric polynomials $\mathrm h_k,$ in the variables $\hbar X_j$. This case is particularly important because, as it was stated earlier, the homogeneous symmetric polynomials form a multiplicative generating set for the set of all symmetric polynomials, so the knowledge of the inner products, 
\[
\left< \mathcal C_\mu,  \mathrm h_k(X_1,\ldots,X_n) \mathcal C_\nu\right>_0
\]
 allows to compute the action of every symmetric polynomial $\mathrm F$, as, for any symmetric functions $\mathrm F,\mathrm G\in \Lambda$, due to orthogonality of the basis elements $\mathcal C_\lambda$, we have
 \begin{equation}\label{eq:strcoef}
 \left< \mathcal C_\mu, \mathrm F\mathrm G(X_1,\ldots,X_{|\mu|}) \mathcal C_\nu\right>_0 = \sum_{\lambda} \frac{\left<\mathcal C_\nu, \mathrm F(X_1,\ldots,X_{|\mu|}) \mathcal C_\lambda\right>_0 \left< \mathcal C_\lambda, \mathrm G(X_1,\ldots,X_{|\mu|})  \mathcal C_\nu\right>_0}{\left< \mathcal C_\lambda,\mathcal C_\lambda\right>_0}
 \end{equation}
  We will refer to the inner products
 \[
 f_{\mathrm F, \lambda}^{\mu,0} = \left< \mathcal C_\mu,  \mathrm F(X_1,\ldots,X_{|\mu|}) \mathcal C_\nu\right>_0
 \]
 as the \emph{structure coefficients} of the Schur action.
 
 The paper~\cite{Karev-Do} contains a recursion for connected \emph{orbifold} monotone Hurwitz numbers, that can be easily generalized to the double case: for partitions $\mu,\nu$ with a chosen  order of the parts $(m_1,\ldots,m_q)$ for $\mu$ and $(n_1,\ldots,n_p)$ for $\nu$  we introduce a set of numbers $\mathrm N^{0,a}_g\left(\begin{smallmatrix}
    n_1,\ldots,n_p\\ m_1,\ldots,m_q|r
\end{smallmatrix}\right),$ for $r =1 ,\ldots, q$ and $a = 1,\ldots,n_p$ with 
\[\mathrm{N}^0_g\left(
    \begin{smallmatrix}
    n_1,\ldots,n_p\\ m_1,\ldots, m_q
\end{smallmatrix}\right) = \sum_{a = 1}^{n_p} \sum_{r = 1}^{q} \mathrm N^{0,a}_g\left(\begin{smallmatrix}
    n_1,\ldots,n_p\\ m_1,\ldots, m_q|r
\end{smallmatrix}\right).\]
These numbers are related to the transitive monotone Hurwitz numbers as 
\[
H^{{(1 - z)^{-1}}}_{{2g - 2 + p + q}} \left(\begin{smallmatrix} \nu\\ \mu \end{smallmatrix};0\right) = \frac 1{|\mathrm{Aut}(\mu)\times\mathrm{Aut}(\nu)| } \mathrm{N}^0_g\left( \begin{smallmatrix} n_1,\ldots,n_p \\  m_1\ldots,m_q \end{smallmatrix}\right).
\]
Notice, that the subscript of $\mathrm{N}^0_g\left( \begin{smallmatrix} n_1,\ldots,n_p \\  m_1\ldots,m_q \end{smallmatrix}\right)$ refers to the expected \emph{genus} $g$ of the covering corresponding to the admissible monodromy data, while the subscript of $H^{{(1 - z)^{-1}}}_{{2g - 2 + p + q}} \left(\begin{smallmatrix} \nu\\ \mu \end{smallmatrix};0\right)$ refers to the expected number of simple branch points of the covers we count.

    \begin{theorem}
The numbers $\mathrm N^a_g\left(\begin{smallmatrix}
    n_1,\ldots,n_p\\ m_1,\ldots,m_q|r
\end{smallmatrix}\right)$ satisfy the  initial condition
\[ \mathrm N^{0,a}_0\left(\begin{smallmatrix}
    n_1,\ldots,n_p\\ m_1,\ldots,m_q|r
\end{smallmatrix}\right) = \delta_{a,1} \delta_{p,1}\delta_{q,1}\delta_{m_1,n_1} \frac{1}{n_1}, \]
and the recursion
\begin{gather}
\mathrm N^{0,a}_g\left(\begin{smallmatrix}
    n_1,\ldots,n_p\\ m_1,\ldots,m_q|r
\end{smallmatrix}\right) = \Theta(m_r  - n_p + a - 1) \left(  \sum_{j \ne r}\sum_{l = 1}^a \mathrm N_{g}^{0,l} \left( \begin{smallmatrix} n_1,\ldots,n_p \\ m_1,\ldots,\hat m_j,
 \ldots, m_r + m_j,\ldots,m_q| r'\end{smallmatrix}\right)  \right)\\
+\sum_{\alpha+ \beta = m_r} \sum_{l = 1}^a  \beta \Biggl( \mathrm N_{g-1}^{0,l} \left( \begin{smallmatrix} n_1,\ldots,n_p \\ m_1,\ldots,\alpha,\beta,\ldots,m_q| r\end{smallmatrix}\right) + \Biggr. \\ \Biggl.   \sum_{g_1 + g_2 = g} \sum_{
\begin{smallmatrix} K_1 \sqcup K_2 = \{1,\ldots,p\},\ p\in K_2 \\ I_1 \cup I_2 = \{1,\ldots,q\},\ I_1\cup I_2 = \{r\}
\end{smallmatrix}
}   \mathrm{N}^{0}_{g_1}(\begin{smallmatrix}n_{K_1}\\ m_{I_1}({m_r \mapsto  \beta)}\end{smallmatrix})N^{0,l}_{g_2}(\begin{smallmatrix} n_{K_2}\\ m_{I_2}(m_r \mapsto \alpha)|r\end{smallmatrix})\Biggr),
\end{gather}
where $\Theta$ is the Heaviside step function, $r'$ stands for the position of $m_r + m_j$ in the resulting composition, and $m_r \mapsto \alpha$ stands for the replacement of $m_r$ by $\alpha$.
\end{theorem}
(For the details on the notation, see Section \ref{sc:ref-mon}).




\subsection{Zonal action and purely real (twisted) Hurwitz numbers}\label{Sc:Centrality}  Consider the set $[\bar n] = \{1,\bar 1,\ldots, n,\bar n\}$ provided with the following order:
$$1 < \bar 1< 2 <\bar 2< \cdots <n <\bar n.$$ The group $S_{2n}$ is regarded as  the group of self-bijections of $[\bar n]$.

We use the following shorthand convention. For any $i \in [\bar n]$ we denote
\[
\bar i = \begin{cases}
    \bar i,\mbox{ if } i\in\{1,\ldots,n \},\\
    i, \mbox{ otherwise.}
\end{cases}
\]

 The \emph{hyperoctahedral group} $H_n \le S_{2n}$ is defined to be the centralizer of the fixed point-free involution
$$\tau = (1\, \bar 1)\cdots (n\,\bar n).$$
The group $S_{2n}$ acts 
on the subspace $M_n\le \mathbb C[S_{2n}]$ spanned by fixed point-free involutions $\rho \in S_{2n}$ 
 by the conjugation:
 \[
 (g,\rho) \mapsto g.\rho = g\rho g^{-1}.
 \]
This action extends by linearity to the action of the whole group algebra $\mathbb C[S_{2n}].$ By abuse of notation, we sometimes denote the vector space $M_n$ and the corresponding representation of $S_{2n}$ by the same symbol.

\begin{remark}
    In the following, we will rely on results obtained in~\cite{Macdonald,Matsumoto} and other papers, that use a slightly different model for the action of $S_{2n}$ on $M_n.$ Namely, let $e_n$ be the idempotent corresponding to the hyperoctahedral subgroup $H_n:$ 
    \[
    e_n = \frac 1{|H_n|}\sum_{\sigma \in H_n} \sigma \in \mathbb C[S_{2n}].
    \]
    Then the group $S_{2n}$ acts on the left ideal generated by $e_n$, or, equivalently, on the set of left coset classes of $H_n$,  by the usual group algebra multiplication. As a transitive group action on a set is uniquely determined by the stabilizer subgroup of an element, we see, that our action on the set of fixed point-free involutions, and the action on the left cosets are equivalent. 
\end{remark}

 The cycle type of the product $\rho\tau$ uniquely defines the orbit of $\rho$ under the action of $H_n:$

\begin{proposition}[\cite{Macdonald}]\label{Prop:type}
Let $\rho$ and $\rho'$ are two fixed point-free involutions in $S_{2n}$. Then they belong to the same orbit of the action of $H_n$ if and only if the cycle type of $\rho\tau$ coincides with the cycle type of $\rho'\tau$.
\end{proposition}
\begin{proof}

For any fixed points free involution $\rho$, the product $\rho\tau$ is an element of $S_{2n}$ of cycle type corresponding to a union $\lambda\cup \lambda$ of a partition $\lambda \vdash n$ with itself. The ``if'' part of the Proposition is obvious. The ``only if'' part can be justified by a direct construction of an element of $H_n$ relating the elements.
\end{proof}

For any fixed point-free involution $\rho,$ the cyclic decomposition of the product $\rho\tau$ contains an even number of cycles of any fixed length.
\begin{definition}
    We say that a fixed point-free involution $\rho$ is of the type $\lambda \vdash n$, $\lambda = (1^{m_1}2^{m_2}3^{m_3}\ldots)$, if the cycle type of the element $\rho \tau$ is $\lambda\cup \lambda = (1^{2m_1}2^{2m_2}3^{2m_3}\ldots).$
\end{definition}

Proposition~\ref{Prop:type} identifies the set of fixed point-free involutions of type $\lambda$ with an orbit of the action of $H_n$.

It is convenient to represent fixed point-free involutions as perfect matchings on the graph with nodes $\{1,\bar 1,\ldots,n,\bar n\}$ (see, e.g.~\cite{Macdonald}): the points interchanged by the involution get connected by an arc. The perfect matchings representation simplifies the calculation of the type of fixed point-free involution: the union of matchings corresponding to fixed point-free involutions $\rho$ and $\tau$  forms a graph with several cycles, and the halves of lengths of these cycles determine the parts of the partition $\lambda$ that specify the type (see Figure~\ref{fig:type}). Notice, that unlike the representation of elements of the symmetric group by a collection of cycles, in this case, the cycles are not provided with an orientation. We denote the graph formed by the union of matching corresponding to $\rho$ and $\tau$ as $\rho\cup \tau.$
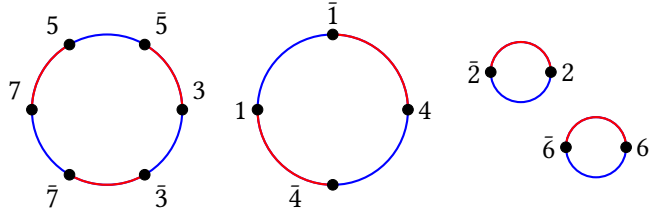
\begin{figure}
\begin{tikzpicture}

    \begin{scope}[shift={(0,0)}]
    \draw[thick,blue] (0,0) circle (1cm);

    \foreach \angle in {0, 120, 240} {
        \draw[thick,red,opacity=1] (\angle:1cm) arc[start angle=\angle, end angle=\angle+60, radius=1cm];
        
        \filldraw (\angle:1cm) circle (2pt); 
        \filldraw (\angle+60:1cm) circle (2pt);

    }

        \node[above right] at (0:1cm) {$3$}; 
        \node[above right] at (60:1cm) {$\bar 5$};
        \node[above left] at (120:1cm) {$5$};
        \node[above left] at (180:1cm) {$7$};
        \node[below left] at (240:1cm) {$\bar 7$};
        \node[below right] at (300:1cm) {$\bar 3$};

    \end{scope}

        \begin{scope}[shift={(3,0)}]
    \draw[thick,blue] (0,0) circle (1cm);

    \foreach \angle in {0, 180} {
        \draw[thick,red,opacity=1] (\angle:1cm) arc[start angle=\angle, end angle=\angle+90, radius=1cm];
        
        \filldraw (\angle:1cm) circle (2pt); 
        \filldraw (\angle+90:1cm) circle (2pt);
    }

    \node[right] at (0:1cm) {$4$}; 
    \node[above] at (90:1cm) {$\bar 1$};
    \node[left] at (180:1cm) {$1$};
    \node[below] at (240:1cm) {$\bar 4$};
    \end{scope}

            \begin{scope}[shift={(5.5,0.5)}]
    \draw[thick,blue] (0,0) circle (0.4cm);

        \draw[thick,red,opacity=1] (0:0.4cm) arc[start angle=0, end angle=180, radius=0.4cm];
        
        \filldraw (0:0.4cm) circle (2pt); 
        \filldraw (180:0.4cm) circle (2pt);

    \node[right] at (0:0.4cm) {$2$}; 
    \node[left] at (180:0.4cm) {$\bar 2$};

    \end{scope}

      \begin{scope}[shift={(6.5,-0.5)}]
    \draw[thick,blue] (0,0) circle (0.4cm);

        \draw[thick,red,opacity=1] (0:0.4cm) arc[start angle=0, end angle=180, radius=0.4cm];
        
        \filldraw (0:0.4cm) circle (2pt); 
        \filldraw (180:0.4cm) circle (2pt);

    \node[right] at (0:0.4cm) {$6$}; 
    \node[left] at (180:0.4cm) {$\bar 6$};

    \end{scope}

\end{tikzpicture}

    \caption{The perfect matchings corresponding to $\tau$ (in blue), and $\rho = (3~\bar 5)(5~7)(\bar 3~7)(\bar 1~4)(1~\bar 4)(2~\bar 2)(6~\bar 6)$ (in red) in $\rho\cup \tau$. The picture indicates that the type of $\rho$ is $(3^12^11^2)$.}
    \label{fig:type}
\end{figure}

We denote the set of fixed point-free involutions of type $\lambda\vdash n$ by $d_\lambda.$ Set
$$\mathcal D_\lambda = \sum_{\rho \in d_\lambda} \rho \in M_n. $$  The elements $\mathcal D_\lambda$ form a basis of a Fock space with an inner product to be specified. We refer to these elements as \emph{the type indicators}, and to the direct sum $I = \bigoplus_{n\ge 0} I_n = \bigoplus_{n\ge 0} \mathrm{Span}(\{\mathcal D_\lambda\}_{\lambda\vdash n})$ as \emph{the space of types.} Its representation-theoretical significance is justified by the following proposition:
\begin{proposition}\label{prop:endo}
    Let $n\in \mathbb N.$ Any endomorphism $\mathcal O$ of the $S_{2n}$-representation  $M_n$  preserves $I_n.$
\end{proposition}
\begin{proof}
    According to~\ref{Prop:type} the action of any element $h\in H_n$ preserves the element $\mathcal D_\mu$ for all $\mu \vdash n.$ Let $\mathcal O$ be the endomorphism of $M_n$. For any $h\in H_n$ and any $\mu\vdash n$ we have
    \[
    h.\mathcal O(\mathcal D_\mu) = \mathcal O(h.\mathcal D_\mu) = \mathcal O(\mathcal D_\mu)
    \]
    so the conjugation by any element of the octahedral group preserves the image $\mathcal O(\mathcal D_\mu)$. This means that if for a fixed point-free involution $\rho$ we have $[\rho].\mathcal O(\mathcal D_\mu) \ne 0,$ for any other fixed point-free involution $\varrho$ of the same type we have the same coefficient. So $\mathcal O(\mathcal D_\mu)$ is an element of $I_n$.
\end{proof} 

For any endomorphism $\mathcal O \colon M_n \to M_n$ the image $\mathcal O(\tau)\in I_n$ determines $\mathcal O$ in a unique way. Moreover, there is a homeomorphism (see~\cite{Macdonald}):
\[
\mathbb B_n \colon \left( \begin{smallmatrix}
    \mathrm{End(M_n)} \to I_n\\
    \phi \mapsto \phi(\tau)
\end{smallmatrix}\right).
\]
The algebra of endomorphisms of $M_n$ is called \emph{the Hecke algebra associated to the Gelfand pair $(S_{2n},H_n).$} It is commutative unital algebra (see~\cite{Macdonald,Matsumoto}). The homeomorphism $\mathbb B_n$ can be used to define an algebra structure on $I_n.$


Following \cite{Matsumoto} we introduce the odd Jucys-Murphy elements  $\mathcal X_i$, $i = 1,\ldots,n$ in $\mathbb C[{S_{2n}}].$
$$ \mathcal X_{i} = \sum_{\begin{smallmatrix}k \in [\bar n] \\ k < i \end{smallmatrix}} (k~i) = \sum_{\begin{smallmatrix} l \in \{1,\ldots,n\} \\ l < i\end{smallmatrix}} \left( (l\, i) + (\bar l\, i)\right). $$
As in the Schur case, $\mathcal X_1 = 0$. These elements form a subset of the whole set of Jucys-Murphy elements, they pairwise commute. The introduction of this family of elements is justified by the following statement:
\begin{proposition}[\cite{Matsumoto, Zinn-Justin}]\label{Prop:uniquefact}
Let $z$ be a formal variable. We have
$$ \prod_{i = 1}^n (z\cdot \mathrm{id} + \mathcal X_i). \tau = \sum_{j = 1}^{n} z^{j} \sum_{\begin{smallmatrix} \lambda\vdash n\\ \ell(\lambda) = n-j \end{smallmatrix}} \mathcal D_\lambda.$$
\end{proposition}
\begin{proof}
This statement essentially follows from the fact that for any fixed point free involution $\rho \ne \tau$ there exists a unique sequence of transpositions
$$t_1,\ldots, t_m$$
for some $m\le n$ such that
\begin{itemize}
\item for every $i = 1,\ldots,m$ the transposition $t_i$ is of the form $(r_i~ k_i)$ for $k_i \in \{1,\ldots,n\}$ and $r_i \in [\bar n], r_i < k_i$;
\item for $i =1, \ldots, m-1$ we have $k_i < k_{i+1}$;
\item $\rho = t_m \cdots t_1.\tau$.
\end{itemize}
This presentation is an analogue of the statement that every element of the symmetric group admits a unique strictly monotone factorisation.
\end{proof}

The following proposition relates the ring of symmetric functions and the endomorphism of $M_n$.
\begin{proposition}[\cite{Matsumoto}]
    For any symmetric function $\mathrm F\in \Lambda,$ the operator  $\mathcal O_\mathrm F$ defined on $M_n$ by
    \[
    \mathcal O_\mathrm{F}\colon g.\tau \mapsto (g \mathrm F(\mathcal X_1,\ldots,\mathcal X_n)).\tau 
    \]
    for all $g\in S_{2n}$ is an endomorphism of $M_n$ as a representation of $S_{2n}.$
\end{proposition}

This statement is non-trivial, as it involves a verification of well-definiteness of the endomorphism: replacing $g$ by $gh$ for $h\in H_n$ does not affect the result of evaluation of $\mathcal O_\mathrm F$ on an element of $M_n$. Due to Proposition~\ref{prop:endo}, for any symmetric polynomial $\mathrm F \in \Lambda$ the corresponding operator $\mathcal O_\mathrm F$ preserves $I_n.$ 

We will need one more set of endomorphisms of $M_n$ indexed by partitions $\nu\vdash n$:

\begin{proposition}
                Let $n\in \mathbb N.$  We have
                \begin{enumerate}

                    \item For all $\nu \vdash n$ the operator, defined on fixed point-free involutions
                    \[
                    R_\nu \colon \left(\begin{matrix}
                        M_n \to M_n\\
                        \rho\mapsto \sum_{\{\sigma\in S_{2n}| r(\sigma) = \nu\cup \nu, \sigma\rho \in M_n  \}} \sigma\rho 
                    \end{matrix}\right)
                    \]
                     is an endomorphism of $M_n$\footnote{Notice, that the definition of $R_\nu$ involves the usual group multiplication, not a conjugation.};
                    \item For any $\nu \vdash n$ we have $R_\nu(\tau) = \mathcal D_\nu;$
                    \item  $R_{(1^{n-2}2^1)} = \mathcal O_{\mathrm p_1}.$
                \end{enumerate}
            \end{proposition}

            \begin{proof}
            \begin{enumerate}
            \item This follows from the equality
            \[
            \varrho( \sigma \rho) \varrho^{-1} = (\varrho \sigma \varrho^{-1})( \varrho \rho \varrho^{-1}),
            \]
            and the fact that the cycle types of $\sigma$ and $\varrho \sigma\varrho^{-1}$ coincide.
            \item Recall that the \emph{type} of a fixed point-free involution $\rho$ is related to the cycle type of the product $\rho\tau.$ According to the definition of $R_\nu,$ it maps $\tau$ to the formal sum of fixed point-free involutions of type $\nu,$ each coming with the coefficient 1. However, we already know that $R_\nu$ is an endomorphism of $M_n,$ so $R_\nu(\tau) = \mathcal D_\nu.$
            \item As for any endomorphism $\mathcal O$ of the representation $M_n$ the image $\mathcal O(\tau)$ determines $\mathcal O$ in a unique way, it is enough to compare  $R_{1^{n-1}2^1}(\tau)$ and $\mathrm p_1(\mathcal X_1,\ldots,\mathcal X_n).\tau.$ The latter, according to \ref{Prop:uniquefact} is the element $\mathcal D_{(1^{n-2}2^1)} \in I_n$.
            \end{enumerate}
            \end{proof}

Introduce the following inner product on $M_n$: for two fixed point-free involutions $\rho_1,\rho_2\in S_{2n}$ define \[\left< \rho_1,\rho_2 \right>_1 = \frac 1{2^nn!} \delta_{\rho_1,\rho_2}.\]
Here, the factor $\frac 1{2^nn!}$ is the inverse of the order of the hyperoctahedral group $H_n.$

\begin{proposition}\label{prop:self-adj}
    Let $n\in \mathbb N.$ For any $\nu\vdash n$ the operator $R_\nu\colon M_n \to M_n$ is self-adjoint with respect to $\left<\cdot,\cdot\right>_1.$
\end{proposition}
\begin{proof}
    It is enough to check that for any two fixed point-free involutions $\rho_1,\rho_2 \in S_{2n}$ we have
    \[
    \left< R_\nu(\rho_1),\rho_2\right>_1 = \left< \rho_1,R_\nu(\rho_2)\right>_1.
    \]
    Suppose $\left< R_\nu(\rho_1),\rho_2\right>_1 \ne 0$. It is equivalent to the existence of $\sigma\in S_{2n}, r(\sigma) = \nu\cup\nu,$ such that $\sigma\rho_1 = \rho_2.$ But it means that $\rho_1 = \sigma^{-1}\rho_2,$ and $r(\sigma^{-1}) = r(\sigma)$, so $\left< R_\nu(\rho_1),\rho_2\right>_1 = \left< \rho_1,R_\nu(\rho_2)\right>_1$.
\end{proof}

The existence of the homeomorphism $\mathbb B_n$ implies that the operators $\{R_\nu\}_{\nu \vdash n}$ form a basis in $\mathrm{End}(M_n),$ so every element of $\mathrm{End}(M_n)$ is self-adjoint with respect to $\left<\cdot,\cdot\right>_1.$ K.~Aker and M.~B.~Can prove in~\cite{Aker-Can}, that the operators $\mathcal O_\mathrm F$ generate the algebra $\mathrm{End}(M_n)$, so for any $\nu\vdash n$ there exists a symmetric polynomial $\mathrm F_\nu$ such that $R_\nu = \mathcal O_{\mathrm F_\nu}.$

Now, for $n \in \mathbb N, \nu \vdash n,$ and a symmetric polynomial $\mathrm F\in \Lambda$ consider the result of the conjugation:
\[
\mathrm F(\mathcal X_1,\ldots, \mathcal X_n). \mathcal D_\nu = \mathrm F(\mathcal X_1,\ldots,\mathcal X_n).(R_\nu(\tau)) = \mathcal O_{\mathrm F_\nu \cdot \mathrm F}(\tau) \in I_n.
\]

And moreover, we have  $\mathrm F(\mathcal X_1,\ldots, \mathcal X_n). \mathcal D_\nu = \mathcal O_\mathrm F(\mathcal D_\nu)$. This implies the following:
\begin{theorem}
    A correspondence that maps a symmetric function $\mathrm F \in \Lambda$ to an operator defined on $I_n$ by
    \[
    \mathcal D_\lambda \mapsto \mathrm F(\mathcal X_1,\ldots,X_{|\lambda|}).\mathcal D_\lambda =\mathcal O_\mathrm F(\mathcal D_\lambda)
    \]
    defines a representation of the ring of symmetric functions on the Fock space $I$.
\end{theorem}
We refer to this representation of $\Lambda$ on $I$ the \emph{zonal action.}
 As all the representation operators commute and are self-adjoint, similarly to the Schur representation, they are simultaneously diagonalizable. We refer to their common orthonormal eigenvectors as the \emph{zonal idempotents} $\check Z_\lambda, \lambda\vdash n$ (see~\cite{Macdonald, Matsumoto})\footnote{ The elements $\check Z_\lambda$ correspond to rescaled zonal polynomials.}. 

 \begin{definition}
     In the setup of the Definition~\ref{def:cont}, the zonal content of the box $\square\in \lambda$ equals
     \[
     c_Z(\square) = 2(x-1) -(y-1).
     \]
 \end{definition}

 \begin{theorem}[\cite{Matsumoto}]\label{thm:zonalcontent} For any symmetric function $\mathrm F\in \Lambda$ we have
 \[
\mathrm F(\mathcal X_1,\ldots,\mathcal X_{|\lambda|}). \check Z_\lambda = \mathrm F(\{c_Z(\square)\}_{\square \in \lambda}) \check Z_\lambda.
\]

 \end{theorem}

The Fock space $I$ is endowed with the  inner product, obtained by the restriction of $\left<\cdot,\cdot\right>_1$ to $I_n$. Let the partition $\mu = (1^{m_1}2^{m_2}\ldots).$ We compute
\[
\langle \mathcal  D_\mu, \mathcal D_\nu\rangle_1 = \delta_{\mu,\nu}\prod_{j \ge 1} \frac 1{(2j)^{m_j}m_j!}.
\]
Said differently, the inner product $\left<\mathcal D_\mu,\mathcal D_\nu\right>_1$ equals $\frac 1{2^{|\lambda|}|\lambda|!}$ multiplied by the coefficient
\[
[\tau]. R_\mu R_\nu(\tau),
\]
indicating that this inner product arises from a trace function $\mathrm{tr}^{(1)}\colon \mathrm{End}(M_n) \to \mathbb C$ defined on the basis as
\[
\mathrm{tr}^{(1)}(R_\mu) = \begin{cases}\frac 1{2^n n!}, \mbox{ if $\mu = (1^n)$},\\ 0,\mbox{ otherwise}. \end{cases}
\]




            The zonal action on $I$ appears to be related to a real geometry enumerative problem. Let $
            \mathcal J\colon \mathbb CP^1 \to \mathbb CP^1$ be the complex conjugation. For a Riemann surface $\Sigma,$ a \emph{real structure} is an orientation reversing involution $\mathcal T,$ and the pair $(\Sigma, \mathcal T)$ is referred to as a \emph{real surface}. We say, that the holomorphic map $f$ from Riemann surface $\Sigma$ endowed with a {real structure} $\mathcal T\colon \Sigma \to \Sigma$ to $\mathbb CP^1$ is a \emph{real cover of $\mathbb CP^1$,} if $f\circ \mathcal T = \mathcal J\circ f.$ It follows from the definition, that the critical values of a real map $f$ either come in complex conjugated pairs $(a,\mathcal J(a))$, or purely real. We focus on the latter case.

            Two real covers $(\Sigma_1,\mathcal T_1,f_1)$ and $(\Sigma_2,\mathcal T_2, f_2)$ are \emph{equivalent}, if there is a holomorphic isomorphism $D\colon \Sigma_1 \to \Sigma_2$ with the properties
            \begin{itemize}
            \item $D\circ \mathcal T_1 = \mathcal T_2\circ D$ (preservation of the real structure);
            \item $D\circ f_2 = f_1$ (equivariance).
            \end{itemize}
            The map $D$ is called the \emph{equivalence} of the real covers $(\Sigma_1,\mathcal T_1,f_1)$ and $(\Sigma_2,\mathcal T_2,f_2)$.  The notion of equivalent real covers gives rise to the notion of the automorphism group $\mathrm{Aut}(\mathcal T,f)$ of the real cover, which is formed by self-equivalences. The inverse order of the automorphism group of the real cover $(\Sigma,\mathcal T,f)$ is called its \emph{weight}.

            Let $u\in \mathbb RP^1$ be a regular value of $f;$ the preimage $f^{-1}(u)$ consists of $n$ points. The restriction of the real structure $\mathcal T$ to $f^{-1}(u) $ provides it with an involution $\mathrm t\colon f^{-1}(u) \to f^{-1}(u)$. It was proven in~\cite{Guay-Paquet-Markwig-Rau, Lozhkin}, by adapting the classical complex argument for the real case, that the weighted number of possibly disconnected real covers ramified over the points $(p_1,\ldots,p_k) \in \mathbb RP^1 $  can be computed using the following theorem:
            \begin{theorem}[\cite{Guay-Paquet-Markwig-Rau}]\label{th:realHN}
            The weighted number of possibly disconnected real covers of $(\mathbb CP^1,\mathcal J)$ of degree $d$ ramified over the points $(p_1,\ldots,p_k) \in \mathbb RP^1 $, with the corresponding ramification profiles $\nu_1,\ldots,\nu_k \vdash d$ is $\frac 1{d!}$ multiplied by the number of tuples
            \[
            (\mathrm t;\varsigma_1,\ldots,\varsigma_k)
            \]
            of elements of $S_{d}$ with the following properties:
            \begin{itemize}
            \item $\varsigma_i$ is of cycle type $\nu_i$ for all $i  = 1,\ldots,k$;
            \item $\mathrm t$ is an involution (that is $\mathrm t^2 = \mathrm{id}$);
            \item the product $\varsigma_k\cdots \varsigma_1$ is the unit element of $S_{d}$;
            \item the partial products $\varsigma_i \cdots \varsigma_1 \mathrm t$ are  involutions for all $i = 1,\ldots,k$.
            \end{itemize}
            \end{theorem}

            We note here, that the operators $\mathcal O_\mathrm F$ also appear in the real geometry in a natural way -- they are related to real covers of $(\mathbb CP^1,\mathcal J)$ ramified over a pair of complex conjugated points (see~\cite{Lozhkin} for details). Yu. Burman and R. Fesler use this correspondence in~\cite{burman2024real}, identifying a cover with a real critical value with a limit of covers with two complex conjugated critical values tending towards a point on $\mathbb RP^1$. 
            
            In the following, we focus on the case, when the real locus of the covering real surface $(\Sigma,\mathcal T)$ is empty. It implies that the degree of $f$ is even and equals $2n$, and the involution $\mathrm t$ and all the partial products $\varsigma_i \cdots \varsigma_1 \mathrm t$ are fixed point-free involutions (see the relation between the behaviour of the real structure over a point of $\mathbb RP^1$ and the corresponding involutions in \cite{Guay-Paquet-Markwig-Rau}). We can choose the labelling of the sheets of $f^{-1}(u)$ in such a way that the corresponding involution $\mathrm t = \tau,$ and that replaces the factor $\frac 1{d!}$ in the statement of Theorem~\ref{th:realHN} by $\frac 1{2^n n!}$, which is the inverse order of the centralizer of $\tau$. Moreover, recall, that we have defined a \emph{type} of a fixed point-free involution $\rho$ basing on the cycle type of the product $\rho\tau.$ In particular, it means that the products $\varsigma_i\cdots\varsigma_1$ for all $i = 1,\ldots,k$ can be of a specific cycle type only.
            Yu. Burman and R. Fesler in  \cite{burman2021ribbon} give a description of such transpositions:
            \begin{proposition}[\cite{burman2021ribbon}]
                A permutation $\varsigma\in S_{2n}$ is such that $\varsigma\tau$ is a fixed point-free involution if and only if the set of independent cycles of $\varsigma$ is naturally split into pairs: for an independent cycle $c$  of $\varsigma,$  $\tau c^{-1} \tau$ is also an independent cycle of $\varsigma$ disjoint from $c$.  
            \end{proposition}

            So all the permutations $\varsigma_i$ have the cycle type  of the form $(1^{2m_1}2^{2m_2}3^{2m_3}\ldots)$ for some non-negative integers $m_1,m_2,\ldots$ Keeping the data about the half of the parts of the corresponding partition only, we get the notion of \emph{the reduced ramification profile} $\mu_i \vdash n.$ Moreover, the reduced ramification profile of the permutation $\varsigma_1$ that contributes to the real covers with empty real part count, is precisely the \emph{type} of the fixed point-free involution $\varsigma_1\tau.$ Translating the tuples of permutation counting problem to the language of representation theory, we obtain

            \begin{theorem}
                The weighted number of possibly disconnected real covers with empty real locus of $(\mathbb CP^1,\mathcal J)$ of degree $2n,$ ramified over the points $(p_1,\ldots,p_k)\in \mathbb RP^1,$ with the corresponding reduced ramification profiles $\mu_1,\ldots,\mu_k$, is given by 
                \[
                 \sum_{\begin{smallmatrix} [f]\\ f \colon (\Sigma,\mathcal T)\to (\mathbb CP^1,\mathcal J)\\ \mbox{\small real locus of $(\Sigma,\mathcal T)$ is empty} \\p_i\in \mathbb RP^1, r(p_i) = \mu_i\cup \mu_i 
,\ \forall i= 1,\ldots,n\\ f  \mbox{ \small is unramified over } \mathbb CP^1\setminus \{p_1,\ldots,p_k\}\end{smallmatrix}} \frac 1{|\mathrm{Aut}(f)|} = \mathrm{tr}^{(1)} \left( R_{\mu_1}\cdots R_{\mu_k}\right).
                \]
                In particular, this number does not depend on the order in which  the points $(p_1,\ldots,p_k)$ are distributed along $\mathbb RP^1$.
            \end{theorem}

            We can restrict ourselves to the simplest possible case, when all the ramification profiles but the first one and the last one are $(1^{n-2}2^1),$ trying to state a problem resembling the classical double simple Hurwitz numbers problem. A straightforward rephrasing of the previous theorem to the double simple Hurwitz case produces:

         \begin{proposition} For $\mu,\nu \vdash n$ the inner product
         \[
         \left<\mathcal D_\mu, {\mathrm p_1^k(\mathcal X_1,\ldots,\mathcal X_n)}.\mathcal D_\nu\right>_1,
         \]
         where $\mathrm p_1$ stands for the first power sum, equals the weighted  number of possibly disconnected real covers with empty real locus of $(\mathbb CP^1,\mathcal J)$ with the following properties:
         \begin{itemize}
         \item The critical values of the covers are fixed distinct points $p_1,\ldots,p_k \in \mathbb RP^1$;
         \item The reduced ramification profile over the point $p_1$ is $\nu\vdash n$, the reduced ramification profile over $p_k$ is $\mu \vdash n$, over the remaining points $p_i, i \ne 1,k$ the reduced ramification profile is $(1^{n-2}2)$.
         \end{itemize}
         \end{proposition}

          This version of Hurwitz numbers, enumerating real covers with the special properties, is known in the literature as the \emph{twisted simple Hurwitz numbers} (~\cite{burman2024real, HMtwisted1}). The tropical formalism for computation of twisted simple Hurwitz numbers is presented in~\cite{HMtwisted1}.

Generalizing the above, we can choose a weight generating function $G(z)$ with $G(0) = 1,$ and  consider the following element of $I\otimes I[[\hbar]]:$
\[
\mathcal T^G(\hbar) = \sum_{\lambda}{\check Z_\lambda \otimes \check Z_\lambda} \prod_{\square \in \lambda}  G(\hbar c_Z(\square)),
\] where $\check Z_\lambda$ is the corresponding zonal idempotent, and $G$ is a weight generating function. Similarly to the Schur case, using Theorem~\ref{thm:zonalcontent}, orthogonality of generators $D_\mu,$ and self-adjointness of $\mathcal G(\hbar \mathcal X_1,\hbar \mathcal X_2,\ldots,\hbar \mathcal X_n) = \prod_{k\ge 1} G(\hbar\mathcal X_k),$ we can rewrite the element $\mathcal T^G(\hbar)$ as
\[
\mathcal T^{G}(\hbar) = \sum_{n \ge 0}\sum_{\mu, \nu\vdash n} \left< \mathcal D_\mu, \mathcal G(\hbar \mathcal X_1,\hbar \mathcal X_2,\ldots,\hbar \mathcal X_n) \mathcal D_\nu\right>_1 \frac {\mathcal D_\mu}{\left< \mathcal D_\mu,\mathcal D_\mu\right>_1} \otimes \frac{\mathcal D_\nu}{\left< \mathcal D_\nu,\mathcal D_\nu \right>_1}.
\]

\begin{definition}\label{def:real}
    The expansion coefficients $[\hbar^k].\left< D_\mu, \mathcal G(\hbar \mathcal X_1,\ldots,\hbar \mathcal X_n) D_\nu\right>_1$ are called the \emph{not necessarily connected purely real Hurwitz numbers} $H_k^{G,\bullet}\left(\begin{smallmatrix}
    \nu\\ \mu
\end{smallmatrix};1\right).$
\end{definition}

Similarly to the complex case, purely real Hurwitz numbers enumerate real covers with the ramification data restricted by the weight generating function $G$. The interpretation of purely real Hurwitz numbers in terms of odd Jucys-Murphy elements allows to compute them using the combinatorics of symmetric group.

Similarly to the Schur case, we  define the following element of the ring $\mathbb C [[\mathrm p_1,\mathrm p_2,\ldots;\mathrm q_1,\mathrm q_2,\ldots;\hbar]]$
\[
\mathbf t^G = \sum_{k \ge 0}\sum_{\nu,\mu} \hbar ^k  H_k^{G,\bullet}\left(\begin{smallmatrix} \nu \\ \mu \end{smallmatrix};1 \right) \mathrm q_\nu \mathrm p_\mu.
\]
The coefficients of the expansion of the logarithm of $\mathrm t^G$
\[
 H_k^G\left( \begin{smallmatrix}
    \nu \\ \mu
\end{smallmatrix};1\right) = [\hbar^k \mathrm q_\nu \mathrm p_\lambda] \log{\mathrm t^G}
\]
are naturally called \emph{(transitive or connected) purely real Hurwitz numbers.} As in the case of the complex Hurwitz numbers, connected purely real Hurwitz numbers enumerate irreducible real covers\footnote{A way to state the irreducibility of a real cover is the following: the quotient of the covering space modulo its real structure is connected.}.    

\subsection{$b$-Hurwitz numbers}\label{sec:b-hurwitz} The notion of $b$-Hurwitz numbers was introduced by G. Chapuy and M. Doł{\k e}ga
in~\cite{Chapuy-Dolega} and developed in~\cite{Bonzom-Chapuy-Dolega} (see also~\cite{Ch-Dolega-Osuga1, Ch-Dolega-Osuga2}).
To introduce the $b$-Hurwitz numbers we consider the Fock space of symmetric functions $\Lambda,$ endowed with the basis of Newton power-sums $\mathrm p_\mu.$ We declare the Newton power-sums orthogonal with respect to an inner product $(\cdot,\cdot)$ depending on the parameter $b\in \mathbb C$, defined by\footnote{ Macdonald in~\cite{Macdonald} uses the parameter $a = b-1$ instead. However, in the literature on  the $b$-Hurwitz numbers, the parameter $b$ is more common. Notice, that for $b> -1$ this inner product is positively defined.}. 
\[
\left( \mathrm p_\mu, \mathrm p_\nu\right)_b = \delta_{\mu,\nu} \prod_{j \ge 1} { (1+b)^{m_j}j^{m_j} m_j!}.
\]

 The \emph{Jack functions} $ J^{(b)}_\lambda$ for partitions $\lambda$ are orthogonal with respect to $(\cdot,\cdot)$, and are upper-triangular in the monomial basis with respect to the  dominance order of integer partitions (see, e.g.~\cite{Macdonald} for a precise definition). The Jack functions specialize to the Schur functions $s_\lambda$ at $b=0$, and to the zonal functions $Z_\lambda$ at $b = 1.$ 

 Chapuy and Doł{\k e}ga introduce the following generating function
 \[
 \hat F(t,\hbar) = \sum_{n \ge 0} t^n \sum_{\lambda \vdash n} \frac{J_\lambda^{(b)} \otimes J_\lambda^{(b)} }{\mathrm j_\lambda^{(b)}}\prod_{\square \in \lambda} G(\hbar c_{b}(\square)).
 \]
 Here $c_b(\square)$ is the $b$-content defined in the setup of Definition~\ref{def:cont} as
 \[
 c_b(\square) = (b+1)(x-1) - (y-1);
 \]
 and $\mathrm j_\lambda^{(b)} = \left( J_\lambda^{(b)},J_\lambda^{(b)}\right).$ The coefficients of decomposition of this functions in terms of $p_\nu \otimes p_\mu$ were identified with the weighted count of \emph{generalized covers} provided with a \emph{measure of non-orientability (MON)} (see~\cite{Chapuy-Dolega} for the details). We denote these coefficients as $H_k^{G,\bullet}\left(\begin{smallmatrix} \nu \\ \mu \end{smallmatrix};b \right),$ and the corresponding connected numbers (the coefficients of the logarithm of the generating function) by $H_k^{G}\left(\begin{smallmatrix} \nu \\ \mu \end{smallmatrix};b \right).$
 Despite of strong resemblance to the Schur and zonal cases, the analogues of the Jucys-Murphy elements for the Jack functions, that would allow producing the cut-and-join equations from the representation-theoretical considerations, is missing. Their existence is assumed by the Coulter-Do conjecture mentioned in Subsection~\ref{subsec:symm}. The cut-and-join equations for the weight generating functions $G(z) = e^z$ (double case, see \cite{Chapuy-Dolega}) and rational weight generating functions (see~\cite{Bonzom-Chapuy-Dolega} and~\cite{Ch-Dolega-Osuga1}) are known, but their derivation is not based on the combinatorics of permutations, and the corresponding operators in non-polynomial weight generating function cases usually lead to infinite sums.

\subsection{Preliminaries on tropical curves and covers}\label{sec:tropprelim}

We start by introducing tropical curves, and then discuss their covers.

\begin{definition}[Abstract tropical curves]
An abstract tropical curve is a graph $\Gamma$ with the following data:
\begin{enumerate}
    \item The vertex set of $\Gamma$ is denoted by $V(\Gamma)$ and the edge set of $\Gamma$ is denoted by $E(\Gamma)$.
    \item The $1$-valent vertices of $\Gamma$ are called \textit{leaves} and the edges adjacent to leaves are called \textit{ends}.
    \item The set of edges $E(\Gamma)$ is partitioned into the set of ends $E^\infty(\Gamma)$ and the set of \textit{internal edges} $E^0(\Gamma)$.
    \item There is a length function
    \begin{equation}
        \ell\colon E(\Gamma)\to\mathbb{R}\cup\{\infty\},
    \end{equation}
    such that $\ell^{-1}(\infty)=E^\infty(\Gamma)$.
\end{enumerate}
The genus of a connected abstract tropical curve $\Gamma$ is defined as the first Betti number of the underlying graph, i.e.\ $g=1+ |E^0(\Gamma)|- |V(\Gamma)|$. The genus of an abstract tropical curve with $c$ connected components of genus $g_1,\ldots,g_c$ equals $g=\sum g_i-c+1$.

An isomorphism of abstract tropical curves is an isomorphism of the underlying graphs that respects the length function. The combinatorial type of an abstract tropical curve is the underlying graph without the length function.
\end{definition}

We are now ready to define the notion of a tropical cover. We restrict to the case where the target $\Gamma_2$ is a subdivided version of $\mathbb{R}$, i.e.\ a line with some $2$-valent vertices.

\begin{definition}[Tropical covers]
A tropical cover of a subdivided version of $\mathbb{R}$, $\Gamma_2$, is a surjective harmonic map between abstract tropical curves $\pi\colon \Gamma_1\to\Gamma_2$, i.e.:
\begin{enumerate}
    \item We have $\pi(V(\Gamma_1))\subset V(\Gamma_2)$.
    \item Let $e\in E(\Gamma_1)$. Then, we interpret $e$ and $\pi(e)$ as intervals $[0,\ell(e)]$ and $[0,\ell(\pi(e))]$ respectively. We require $\pi$ restricted to $e$ to be a bijective integer linear function $[0,\ell(e)]\to[0,\ell(\pi(e))]$ given by $t\mapsto \omega(e)\cdot t$, with $\omega(e) \in \mathbb{Z}$. If $\pi(e)\in V(\Gamma_2)$, we define $\omega(e)=0$. We call $\omega(e)$ the \textit{weight} of $e$.
    \item For a vertex $v\in V(\Gamma_1)$, we denote by $\mathrm{Inc}(v)$ the set of incoming edges at $v$ (edges adjacent to $v$ mapping to the left of $\pi(v)$) and by $\mathrm{Out}(v)$ the set of outgoing edges at $v$ (edges adjacent to $v$ mapping to the right of $\pi(v)$). We then require
    \begin{equation}
            \sum_{e\in\mathrm{Inc}(v)}\omega(e)=\sum_{e\in\mathrm{Out}(v)}\omega(e).
    \end{equation}
    This number is called the local degree of $\pi$ at $v$.
    We call this equality the \textit{harmonicity} or \textit{balancing condition}.
    For a point $v$ in the interior of an edge $e$ of $\Gamma_1$, the local degree of $\pi$ at $v$ is defined to be the weight $\omega(e)$.
\end{enumerate}
Moreover, we define the \textit{degree} of $\pi$ as the sum of local degrees of all preimages in $\Gamma_1$ of a given point of $\Gamma_2$. The degree is independent of the choice of point of $\Gamma_2$. This follows from the harmonicity condition.\\
For any end $e$ of $\Gamma_2$, we define a partition $\mu_e$ as the partition of weights of ends of $\Gamma_1$ mapping to $e$. We call $\mu_e$ the \textit{ramification profile} of $e$.\\
Two tropical covers $\pi_1\colon\Gamma_1\to\Gamma_2$ and $\pi_2\colon\Gamma_1'\to\Gamma_2$ are isomorphic if there exists an isomorphism $g\colon\Gamma_1\to\Gamma_1'$ of metric graphs, such that $\pi_2\circ g=\pi_1$.
\end{definition}

We now recall the definition of twisted tropical cover from \cite{HMtwisted1}:

\begin{definition}[Twisted tropical covers]\label{def-twistedtropcover}
We define a twisted tropical cover of type $(k,\underline{m},\underline{n})$ to be a tropical cover $\pi\colon\Gamma_1\to\Gamma_2$ with an involution $\iota\colon\Gamma_1\to\Gamma_1$ which respects the cover $\pi$, such that:
\begin{itemize}
    \item The target $\Gamma_2$ is a subdivided version of $\mathbb{R}$ with vertices $\{p_1,\dots,p_k\}=V(\Gamma_2)$, where $p_i<p_{i+1}$. 
    These points are called the \emph{branch points}.
    \item There are $\ell(\underline{m})$ pairs of ends mapping to $(-\infty,p_1]$ with weights $\underline{m}_1,\dots,\underline{m}_{\ell(\underline{m})}$ and $\ell(\underline{n})$ pairs of ends mapping to $[p_k,\infty)$ with weights $\underline{n}_1,\ldots,\underline{n}_{\ell(\underline{n})}$. \emph{The ends are labelled.}
    \item in the preimage of each point $p_i$, there are either two $3$-valent vertices or one $4$-valent vertex.
    \item the edges adjacent to a $4$-valent vertex all have the same weight.
    \item the fixed locus of $\iota$ is exactly the set of $4$-valent vertices.
\end{itemize}
\end{definition}

\begin{definition}[Quotient cover $\Gamma/\iota$]
Let $\pi\colon \Gamma\to \mathbb{R}$ be a twisted tropical cover with involution $\iota\colon\Gamma\to\Gamma$. The involution $\iota$ induces a symmetric relation on the vertex and edge sets of $\Gamma$: We define for $v,v'\in V(\Gamma)$  (resp. $e,e'\in E(\Gamma)$) that $v\sim v'$ (resp. $e\sim_\iota e'$) if and only $\iota(v)=v'$ (resp. $\iota(e)=e'$). We define $\Gamma/\iota$ as the graph with vertex set $V(\Gamma)/\sim$ and edge set $E(\Gamma)/\sim$ with natural identifications. For $e=[e',e'']\in E(\Gamma/\iota)$ we define the length $\ell(e)$ as $\ell(e')=\ell(e'')$ and its weight $\omega(e)$ with respect to $\pi$ to be the weight $\omega(e')=\omega(e'')$. The quotient graph $\Gamma/\iota$ also comes with a map to $\mathbb{R}$, which we call the \emph{quotient cover}.
\end{definition}

We define a twisted tropical cover to be \emph{connected} if its quotient graph is connected.

\begin{definition}[Automorphisms]
    Let $\pi:\Gamma\rightarrow \mathbb{R}$ be a twisted tropical cover with involution $\iota:\Gamma\rightarrow \Gamma$. An automorphism of $\pi$ is a morphism of abstract tropical curves (i.e.\ a map of metric graphs) $f:\Gamma\rightarrow\Gamma$ respecting the cover and the involution, i.e.\ $\pi\circ f = \pi$ and $f\circ \iota = \iota \circ f$. We denote the group of automorphisms of $\pi$ by $\Aut(\pi)$.

    An automorphism of the quotient cover is a morphism of abstract tropical curves which respects the cover.
\end{definition}

\begin{proposition}\label{prop-automquotient}
Let  $\overline{\pi}:\overline{\Gamma}\rightarrow \mathbb{R}$ be a connected quotient of a tropical twisted cover. Assume $\overline{\Gamma}$ has $c$ $2$-valent vertices and is of genus $g$.
Then
$$ \sum_{\pi} \frac{1}{|\Aut(\pi)|}= \frac{2^{g}}{2^{c+1}\cdot |\Aut(\overline{\pi})|},$$
where the sum goes over all twisted tropical covers  $\pi:\Gamma\rightarrow \mathbb{R}$  with involution $\iota$ whose quotient $\Gamma/\iota\rightarrow \mathbb{R}$ equals $\overline{\pi}:\overline{\Gamma}\rightarrow \mathbb{R}$.
\end{proposition}
For a proof, see Proposition 16 in \cite{HMtwistelliptic}. (Note the slight difference in convention, see Remark \ref{rem:conventiondifferences}.)

The idea of the proof is as follows: a quotient graph which is a tree without $2$-valent vertices has precisely one lift, by taking two copies which are switched by the involution. Let us call the two copies the top and bottom copy. Whenever we close a cycle, we have two options to do so: we can either connect top with top and bottom with bottom, or top with bottom and bottom with top. Every $2$-valent vertex lifts to a $4$-valent vertex which yields an extra automorphism factor, as we can switch the adjacent edges. The involution accounts for another extra automorphism.

\section{CJT-refinement: a preliminary version}\label{sc:cjt-prep}

In this section, we present an action of the ring of the symmetric function on a Fock space, realized by the basis of type indicators $\mathcal D_\lambda$, that refines the zonal action. Let $C,J,T$ be independent variables, $n\in \mathbb N, i,j \in [\bar n] = \{1,\bar1,\ldots,n,\bar n\}$.


\begin{definition}\label{Def:CJTweight}
For a transposition $\sigma = (i~j) \in S_{2n}$ and a fixed point-free involution $\rho \in S_{2n}$ we define the $CJT-$weight $w(\rho,\sigma)$ of the pair $(\rho,\sigma)$ as follows:
\begin{itemize}
    \item $w(\rho,\sigma) =C$, if $\sigma$ makes a cut for $\rho$, i.e. the number of parts in the type of $\rho$ is one less than the number of parts in the type $\sigma.\rho$.

    Following the convention of representing fixed point-free involutions as perfect matchings, a transposition $\sigma = (i~j)$ makes a cut for $\rho$ if the both elements $i,j$ belong to the same connected component of $\rho\cup \tau,$ and belong to the same cycle of the product $\rho\tau$. A schematic picture of a cut is on the figure~\ref{fig:cut-join}.

    \begin{figure}\begin{tikzpicture}

    \begin{scope}[shift={(-1.5,0)}]
    \draw[dotted] (0,0) circle (1cm);


        \draw[thick,blue,opacity=1] (-10:1cm) arc[start angle=-10, end angle=10, radius=1cm];
        \draw[thick,red,opacity=1] (10:1cm) arc[start angle=10, end angle=30, radius=1cm];
        \draw[thick,blue,opacity=1] (130:1cm) arc[start angle=130, end angle=150, radius=1cm];
        \draw[thick,red,opacity=1] (150:1cm) arc[start angle=150, end angle=170, radius=1cm];
        \draw[thick,->,opacity=1] (60:1cm) arc[start angle=60, end angle=100, radius=1cm];
        \draw[thick,->,opacity=1] (240:1cm) arc[start angle=240, end angle=280, radius=1cm];
        
        \filldraw (10:1cm) circle (2pt); 
        \filldraw (150:1cm) circle (2pt);

        \node[above right] at (10:1cm) {$i$}; 
        \node[above left] at (150:1cm) {$j$};
        \node[above right] at (80:1cm) {$A$};
        \node[below left] at (260:1cm) {$B$};

    \end{scope}

    $\mapsto$

    \begin{scope}[shift={(1,0.5)}]
    \draw[dotted] (0,0) circle (0.4cm);


        \draw[thick,blue,opacity=1] (-30:0.4cm) arc[start angle=-30, end angle=10, radius=0.4cm];
        \draw[thick,red,opacity=1] (10:0.4cm) arc[start angle=10, end angle=50, radius=0.4cm];
        \draw[thick,->,opacity=1] (150:.4cm) arc[start angle=150, end angle=210, radius=.4cm];

        \filldraw (10:0.4cm) circle (2pt); 

        \node[above right] at (10:0.4cm) {$i$};
        \node[left] at (180:0.4cm) {$B$};
    \end{scope}

    \begin{scope}[shift={(2,-0.5)}]
    \draw[dotted] (0,0) circle (0.4cm);


        \draw[thick,blue,opacity=1] (-30:0.4cm) arc[start angle=-30, end angle=10, radius=0.4cm];
        \draw[thick,red,opacity=1] (10:0.4cm) arc[start angle=10, end angle=50, radius=0.4cm];
        \draw[thick,->,opacity=1] (150:.4cm) arc[start angle=150, end angle=210, radius=.4cm];

        \filldraw (10:.4cm) circle (2pt); 

        \node[above right] at (10:.4cm) {$j$};
        \node[left] at (180:0.4cm) {$A$};
    \end{scope}
    
    \end{tikzpicture}
    \caption{The transposition $(i~j)$ makes a cut of one of the cycles. Here, $A$ and $B$ denote sequences of elements separated by elements $i$ and $j$, that are shared on the pictures on the right and on the left. Arrows show the orders in which the sequences $A$ and $B$ should be read. The same picture read right to left represents a join of two cycles performed by the transposition $(i~j).$}\label{fig:cut-join}
    \end{figure}
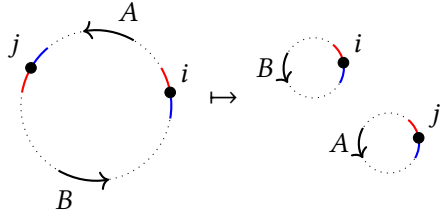

    \item $w(\rho,\sigma) = J$, if $\sigma$ makes a join for $\rho,$ i.e. the number of parts in the type of $\rho$ is one more than the number of parts in the type $\sigma.\rho;$

    A transposition $\sigma = (i~j)$ makes a join if elements $i$ and $j$ belong to different connected components of $\rho\cup \tau.$
    
    \item $w(\rho,\sigma) = T$, if $\sigma$ makes a twist for $\rho$, i.e. the both $\rho$ and $\sigma.\rho$ are of the same type. 

    A transposition $\sigma = (i~j)$ makes a twist for $\rho$ if both elements $i,j$ belong to the same connected component of $\rho\cup \tau$ but belong to different cycles of the product $\rho\tau.$ A twist keeps all the cycles of $\rho\cup \tau$, except for the one containing $i,j$, intact. A schematic picture of a twist is presented in figure~\ref{fig:twist}.
        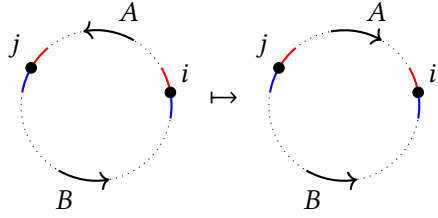
\begin{figure}\begin{tikzpicture}

    \begin{scope}[shift={(-1.5,0)}]
    \draw[dotted] (0,0) circle (1cm);


        \draw[thick,blue,opacity=1] (-10:1cm) arc[start angle=-10, end angle=10, radius=1cm];
        \draw[thick,red,opacity=1] (10:1cm) arc[start angle=10, end angle=30, radius=1cm];
        \draw[thick,red,opacity=1] (130:1cm) arc[start angle=130, end angle=150, radius=1cm];
        \draw[thick,blue,opacity=1] (150:1cm) arc[start angle=150, end angle=170, radius=1cm];
        \draw[thick,->,opacity=1] (60:1cm) arc[start angle=60, end angle=100, radius=1cm];
        \draw[thick,->,opacity=1] (240:1cm) arc[start angle=240, end angle=280, radius=1cm];
        
        \filldraw (10:1cm) circle (2pt); 
        \filldraw (150:1cm) circle (2pt);

        \node[above right] at (10:1cm) {$i$}; 
        \node[above left] at (150:1cm) {$j$};
        \node[above right] at (80:1cm) {$A$};
        \node[below left] at (260:1cm) {$B$};

    \end{scope}

    $\mapsto$

    \begin{scope}[shift={(1.4,0)}]
    \draw[dotted] (0,0) circle (1cm);


        \draw[thick,blue,opacity=1] (-10:1cm) arc[start angle=-10, end angle=10, radius=1cm];
        \draw[thick,red,opacity=1] (10:1cm) arc[start angle=10, end angle=30, radius=1cm];
        \draw[thick,red,opacity=1] (130:1cm) arc[start angle=130, end angle=150, radius=1cm];
        \draw[thick,blue,opacity=1] (150:1cm) arc[start angle=150, end angle=170, radius=1cm];
        \draw[thick,<-,opacity=1] (60:1cm) arc[start angle=60, end angle=100, radius=1cm];
        \draw[thick,->,opacity=1] (240:1cm) arc[start angle=240, end angle=280, radius=1cm];
        
        \filldraw (10:1cm) circle (2pt); 
        \filldraw (150:1cm) circle (2pt);

        \node[above right] at (10:1cm) {$i$}; 
        \node[above left] at (150:1cm) {$j$};
        \node[above right] at (80:1cm) {$A$};
        \node[below left] at (260:1cm) {$B$};

    \end{scope}

    \end{tikzpicture}
    \caption{The transposition $(i~j)$ makes a twist of one of the cycles. Here, $A$ and $B$ denote sequences of elements separated by elements $i$ and $j$, that are shared on the pictures on the right and on the left. Arrows show the orders in which the sequences $A$ and $B$ should be read.}\label{fig:twist}
    \end{figure}
\end{itemize}
\end{definition}

The $CJT$-refined action space is $I(C,J,T)\simeq I\otimes \mathbb C(C,J,T),$ where $I$ stands for the space of type indicators (see subsection~\ref{Sc:Centrality}). First, we define the action of a transposition $\sigma$ on a fixed point-free involution $\rho$ as follows:
\[
\hat \sigma.\rho = w(\rho, \sigma)  \sigma.\rho 
\]

We do not define a refined action of an arbitrary permutation: it is not relevant for the introduction of the Jucys-Murphy elements. Recall, that $M_n$ denotes the subspace of $\mathbb C[S_{2n}]$ spanned by fixed point-free involutions (see subsection~\ref{Sc:Centrality}). Define the \emph{preliminary refined Jucys-Murphy element} $\mathbf X_k$ with $1 \le k \le n$ as a $\mathbb C(C,J,T)$ -linear operators $\mathbf X_i\colon M_n(C,J,T) \to M_n(C,J,T)$ that acts on a fixed points-free involution $\rho$ by 
\[
\mathbf X_k(\rho) = \sum_{\begin{smallmatrix}j\in [\bar n] \\j < k\end{smallmatrix} } \widehat{(j~k)}.\rho =  \sum_{i \in \{1,\ldots ,k-1\}}\left( \widehat{(i\,k)}.\rho + \widehat{(\bar i\, k)}.\rho\right).
\]
As it is usual for the Jucys-Murphy elements, $\mathbf X_1 = 0.$

We will need the following lemma, which can be verified directly:

\begin{lemma}\label{lm:cut-twist}
Let $\rho\in S_{2n}$ be a fixed point free involution, $i,j\in [\bar n],  i\ne j,\bar j.$ 

Then 
\begin{enumerate}
\item $(i~j)$ makes a cut for $\rho$  if and only  if $ (\bar i~j)$ makes a twist for $\rho$.

\item $(i~j)$ makes a join for $\rho,$ if and only if $(\bar i~k)$ makes a join for $\rho$ as well. In this case, the types of $(i~j).\rho$ and $(\bar i~j).\rho$ coincide.
\end{enumerate}
\end{lemma}

Notice, that being written in a basis that consists of fixed point-free involutions, the only possible elements of matrices of the operators $\mathbf X_k$ in the basis of fixed point-free involutions are $0,C,J,T$. For all $k = 2,\ldots,n$ we introduce the decomposition
\[
\mathbf X_k = \mathbf T_k + \mathbf Q_k,
\]
where $\mathbf T_k$ is the ``twist part'' of the Jucys-Murphy element. It is a $\mathbb C(C,J,T)-$linear operator defined on a fixed point-free involution $\rho\in d_\lambda$ for $\lambda\vdash n$  as
\[\mathbf T_k( \rho) = P_\rho \mathbf X_k(\rho), \]
where  \[P_\rho (\varrho) = \begin{cases}
    \varrho,\ \mbox{if the types of $\rho$ and $\varrho$ are the same,}\\
    0,\ \mbox{otherwise.}
\end{cases}
\] stands for the projection to the subspace of elements of $M_n(C,J,T)$ spanned by the elements of $d_\lambda$ along the subspace generated by fixed point-free involution that do not belong to $d_\lambda$. In particular, it means, that in the basis of fixed point-free involutions, the only possible entries of the matrix of $\mathbf T_k$ are either $0,$ or $T.$ Similarly, $\mathbf Q_k$ is a ``cut-join'' part of $\mathbf X_k$, whose value on a fixed point-free involution $\rho \in d_\lambda$ is defined as
\[
\mathbf Q_k(\rho) = (1 - P_\rho)\mathbf X_k(\rho),
\]
so the matrix elements of $\mathbf Q_k$ are polynomials in $C,J$ only.

The proof of the following lemma relies on the fact that the odd Jucys-Murphy elements $\mathcal X_k, 1\le k\le n,$ commute: namely for a pair  $(i,j)\in[\widebar{k-1}]\times[\widebar{l-1}] $ one can associate uniquely a pair $(i',j')\in [\widebar{k-1}]\times[\widebar{l-1}] $ such that $(i~k)(j~l) = (j'~l)(i'~k).$ We denote this bijection $\mathcal B,$
so $\mathcal B\left((i~k),(j~l)\right) = \left((j',l),(i'~k) \right).$

\begin{lemma}\label{lm:coef}
    For all $1 \le k,l \le n,$ the operators $\mathbf Q_k,\mathbf T_l$ satisfy the following relations:
    \begin{gather}
    [\mathbf T_k, \mathbf Q_l] = 0,\\
    CJ[\mathbf T_k, \mathbf T_l] + T^2[\mathbf Q_k,\mathbf Q_l] = 0.    
    \end{gather}
\end{lemma}

\begin{proof}
    For a fixed point-free involution $\rho$ consider
    \begin{gather}
    \mathbf X_k \mathbf X_l(\rho) = (\mathbf T_k + \mathbf Q_k)(\mathbf T_l + \mathbf Q_l)(\rho) = \sum_{\varrho} F_{kl}^\varrho \varrho,\\
    \mathbf X_l \mathbf X_k(\rho) = (\mathbf T_l + \mathbf Q_l)(\mathbf T_k + \mathbf Q_k)(\rho) = \sum_{\varrho} F_{lk}^\varrho \varrho
    \end{gather}
    where the sum ranges over the set of fixed point-free involutions. The coefficients $F_{kl}^\varrho, F_{lk}^\varrho$ are quadratic polynomials in the variables $C,J,T.$ For this coefficient to be non-zero, the number of parts in the type of $\varrho$ can not differ from the number of parts in the type of $\rho$ by more than 2.

    Notice, the following:
    \begin{itemize}
    \item  If the type of element $\varrho$ has one more or one fewer parts than the type of $\rho,$ then both the coefficients $F_{kl}^\varrho, F_{lk}^\varrho$ are integer multiples of $CT$ or $JT$ respectively. Due to the equality $[\mathcal X_k, \mathcal X_l] = 0,$ in this case we have $F_{kl}^\varrho = F_{lk}^\varrho.$ It implies that $[\mathbf T_k, \mathbf Q_l] = 0.$
    
    \item If the type of element $\varrho$ has two more or two fewer parts than the type of $\rho,$ then both the coefficients $F_{kl}^\varrho, F_{lk}^\varrho$ are integer multiples of $C^2$ or $J^2$ respectively. Due to the equality $[\mathcal X_k, \mathcal X_l] = 0,$ in this case we have $F_{kl}^\varrho = F_{lk}^\varrho.$

    \item \ If the type of element $\varrho$ has the same number of parts as the type of $\rho,$ but these types are different, then $F_{kl}^\varrho, F_{lk}^\varrho$ are integer multiples of $CJ$. Due to the equality $[\mathcal X_k, \mathcal X_l] = 0,$ in this case we also have $F_{kl}^\varrho = F_{lk}^\varrho.$

    \item The only non-trivial case occurs when the type of $\rho$ coincides with the type of $\varrho$ In this case notice the following. Both the coefficients $F_{kl}^\varrho$ and $F_{lk}^\varrho$ are of the form $aCJ + bT^2,$ where $a,b \in \mathbb Z.$ The difference may occur when the pair of transpositions $(i~k)(j~l)$ make two twist (cut and join, respectively) for $\rho,$ but the corresponding pair $\mathcal B\left((i~k),(j~l)\right)$ makes cut and join (two twists, respectively) for $\rho$. It implies that $F_{kl}^\varrho - F_{lk}^\varrho = a (CJ - T^2)$ for some $a \in \mathbb Z.$ A example of a cycle $\rho$ and transpositions $(i~k),(j~l)$ such that $[\widehat{(i~k)},\widehat{(j~l)}].\rho = (T^2 - CJ)\varrho$ is on the Figure~\ref{fig:example2a}.
\end{itemize}

 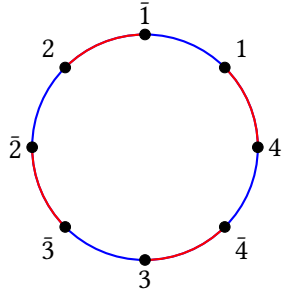
\begin{figure}
\begin{tikzpicture}

    \begin{scope}[shift={(0,0)}]
    \draw[thick,blue] (0,0) circle (1.5cm);

    \foreach \angle in {0, 90, 180, 270} {
        \draw[thick,red,opacity=1] (\angle:1.5cm) arc[start angle=\angle, end angle=\angle+45, radius=1.5cm];
        
        \filldraw (\angle:1.5cm) circle (2pt); 
        \filldraw (\angle+45:1.5cm) circle (2pt);

    }

        \node[ right] at (0:1.5cm) {$4$}; 
        \node[above right] at (45:1.5cm) {$1$};
        \node[above ] at (90:1.5cm) {$\bar 1$};
        \node[above left] at (135:1.5cm) {$2$};
        \node[ left] at (180:1.5cm) {$\bar 2$};
        \node[below left] at (225:1.5cm) {$\bar 3$};
        \node[below] at (270:1.5cm) {$3$};
        \node[below right] at (315:1.5cm) {$ \bar 4$};
  \end{scope}
    \end{tikzpicture}
    \caption{  
    A fixed point-free involution $\rho = (\bar 1~2)(\bar 2~3)(3~\bar 4)(1~4)$ such that $
    [\widehat{(\bar 1~3)},\widehat{(2~4)}].\rho = (T^2 - CJ)\left((1~2)(\bar 2~3)(3~4)(\bar 1~\bar 4)\right)  .$ }  \label{fig:example2a}  

    \end{figure}

    It implies that the matrix elements of the commutator $[\mathbf X_k, \mathbf X_l]$ are all of the form $a(CJ - T^2)$ for some $a \in \mathbb Z$. But notice that it means the following: if the matrix element of $[\mathbf Q_k,\mathbf Q_l]$ equals $aCJ,$ then the corresponding matrix element of $[\mathbf T_k,\mathbf T_l]$ equals $-aT^2.$ The statement follows.
    \end{proof}

In particular, we have the relation
\[
[\mathbf X_k,\mathbf X_l] = [\mathbf T_k,\mathbf T_l] + [\mathbf Q_k, \mathbf Q_l] = (1 - \frac {CJ}{T^2})[\mathbf T_k,\mathbf T_l].
\]

Recall, that the space of types $I$ is spanned by the elements $\mathcal D_\lambda = \sum_{\rho \in d_\lambda} \rho$ (see subsection~\ref{Sc:Centrality}).


\begin{proposition}    \label{th:comm}
    Let $p \in \mathbb N$, and $k_1,\ldots,k_p \in \{2,\ldots,n\}$ -- a collection of natural numbers, $\lambda\vdash n$. Then for any permutations $\alpha \in S_p$, we have the following equality of elements of $M_n(C,J,T)$:
    \[
    \mathbf X_{k_p} \cdots \mathbf X_{k_1}( \mathcal D_\lambda) = \mathbf X_{k_{\alpha(p)}} \cdots \mathbf X_{k_{\alpha(1)}}(\mathcal D_\lambda).
    \]
Here, by abuse of notation, $\mathcal D_\lambda$ stands for $\mathcal D_\lambda \otimes 1 \in I_n(C,J,T).$    
\end{proposition}

\begin{proof}
    We start the proof with the following observation. For any $1\le k,l\le n$ and for any $\lambda \vdash n$ we claim 
    \[
    [\mathbf X_k,\mathbf X_l](\mathcal D_\lambda) = 0.
    \]
    Indeed, working over the field $\mathbb C(C,J,T)$, Lemma~\ref{lm:coef} implies
    \[
    [\mathbf X_k,\mathbf X_l](\mathcal D_\lambda) = (1 - \frac {CJ}{T^2})[\mathbf T_k,\mathbf T_l](\mathcal D_\lambda) =(1 - \frac{CJ}{T^2}) \sum_{\varrho} G_{kl}^\varrho \varrho,
    \]
    where the coefficients $G_{kl}^\varrho = a_{kl}^\varrho T^2$ for some $a_{kl}^\varrho \in \mathbb Z.$ Moreover, the coefficients $a_{kl}^\varrho$ are non-zero only if $\varrho \in d_\lambda.$ 
    Given a pair of permutations, such that $(i~k)(j~l)$ make two twists for $\rho$,  there is a unique pair $(i',j') \in \{i,\bar i\}\times \{j,\bar j\}$ such that $(i'~k)(j' l)$ make two twists for $(i~k)(j~l).\rho:$ as elements $j,l$ belong to the same cycle of $(i~k)(j~l).\rho$, we can apply Lemma~\ref{lm:cut-twist}; similar argument works for the transposition $(i~k)$. So we denote
    \[
    \varrho = (i'~k)(j' l)(i~k)(j~l).\rho \in d_\lambda
    \]
    and $(j'~l)(i'~k).\varrho = (i~k)(j~l).\rho$ where on the both sides of the equality we have two twists. Clearly the correspondence $\rho \mapsto \varrho$ is involutive and we have
    \[\widehat{(i~k)}\widehat{(j~l)}.\rho = \widehat{(j'~l)}\widehat{(i'~k)}.\varrho.
    \]

    The claim follows. By induction, it is enough to verify that the statement holds for $\alpha = (p-1~p)$ only. We write
    \begin{gather}
    [\mathbf X_{k_p},\mathbf X_{k_{p-1}}]\mathbf X_{k_{p-2}} \cdots \mathbf X_{k_1}(\mathcal D_\lambda) = 
    (1 - \frac{CJ}{T^2})[\mathbf T_{k_p},\mathbf T_{k_{p-1}}]\mathbf X_{k_{p-2}} \cdots \mathbf X_{k_1}(\mathcal D_\lambda) \\  
    = (1 - \frac{CJ}{T^2})\mathbf X_{k_{p-2}}[\mathbf T_{k_p},\mathbf T_{k_{p-1}}] \cdots \mathbf X_{k_1}(\mathcal D_\lambda) + (1- \frac{CJ}{T^2})[[\mathbf T_{k_p}, \mathbf T_{k_{p-1}}],\mathbf T_{k_{p-2}}]... \mathbf X_{k_1}(\mathcal D_\lambda),
    \end{gather}
    and we continue the commutation pattern for each term until a group of commutators of operators $\mathbf T_j$ hits $\mathcal D_\lambda.$ It is possible to do it 
    due to the commutation relations of Lemma~\ref{lm:coef}; every term of the expression can be further rewritten making the number of operators 
    $\mathbf X_k$ between the group of commutators of $\mathbf T_l$ and $D_\lambda$ one less. Now we claim, that for all
    $n,q \in \mathbb N$, $q\ge 2,$ and for all $\lambda \vdash n$, for all collection of numbers $l_1,\ldots, l_q \in \{2,\ldots,n \},$ and any permutation $\beta \in S_{q} $ we have
    \[
        \mathbf T_{l_q} \cdots \mathbf T_{l_1}(\mathcal  D_\lambda) = \mathbf T_{l_{\beta(q)}} \cdots \mathbf T_{l_{\beta(1)}}(\mathcal D_\lambda).
    \]
    The argument is again inductive and similar to $q = 2$ case: 
    for any chain of transpositions  
    \[
    (j_{q_l}~ q_l)(j_{q_{l-1}}~ q_{l-1})\cdots (j_{q_1}~q_1).\rho = \varrho
    \]
    such that every following transposition makes a twist, there exists a unique sequence of elements $j'_{q_r} \in \{j_{q_r},\bar j_{q_r}\}$ such that
    \[
    \left[(j'_{q_{l-1}}~q_{{l-1}}) (j'_{q_l}~q_l) \cdots (j'_{q_{1}}~q_{1})\right]^{-1}.\varrho
    \]
    with the same property -- it is again implied by Lemma~\ref{lm:cut-twist}. So we have
    \[
    \widehat{(j_{q_l}~ q_l)}\cdots \widehat{(j_{q_1}~q_1)}.\rho = \widehat{(j'_{q_l-1}~ q_{l-1})}\widehat{(j'_{q_l}~ q_l)}\cdots \widehat{(j'_{q_1}~q_1)}.\varrho.
    \]
    The claim and Proposition follow.   
\end{proof}

The following is the main theorem of this section and also gives a partial affirmative answer to \cite[Conjecture 5.4]{Coulter-Do}.

\begin{theorem}\label{th:action}
    The refined Jucys-Murphy elements define an action of the ring of symmetric function in the space of type indicators $I(C,J,T) = \bigoplus_n I_n(C,J,T) \le M(C,J,T).$
\end{theorem}

\begin{proof}
    By Proposition~\ref{th:comm}, we know that the refined Jucys-Murphy elements commute over $\mathbb C(C,J,T)$. 
    
    We start with the following observation. Suppose $\rho$ is a fixed point-free involution, $\sigma$ is a transposition, and $\alpha$ -- element of the hyperoctahedral group $H_n$. Then we have the following identity of $CJT-$weights:
    \[
    w(\rho,\sigma) = w( \alpha.\rho,\alpha\sigma\alpha^{-1}).
    \]
    Indeed, the type of $\alpha.\rho$ coincides with the type of $\rho,$ (see subsection~\ref{Sc:Centrality}), and the type of $\alpha\sigma\alpha^{-1}.(\alpha.\rho) = \alpha \sigma \rho \sigma \alpha^{-1}$ is the type of $\sigma.\rho.$ In particular for any $\alpha \in H_n,$ fixed point-free involution $\rho$ and a transposition $\sigma$ we have
    \[
    \alpha.(\hat \sigma.\rho) = \widehat{\alpha \sigma \alpha^{-1}}.(\alpha.\rho)
    \]

     Elementary symmetric functions form a multiplicative system of generators in the ring of symmetric functions. As the fixed point-free involutions of the same type form an orbit of $H_n$-action (see subsection~\ref{Sc:Centrality},  it is enough to show that elementary symmetric polynomials evaluated on the refined Jucys-Murphy elements preserve the space of type indicators, namely for all $1\le l \le n$ and for all $\alpha \in H_n$ we have 
    \[
    \alpha.(\mathrm e_l(\mathbf X_2,\ldots, \mathbf X_n)(\mathcal D_\lambda)) = \mathrm e_l((\alpha.\mathbf X_2),\ldots, (\alpha.\mathbf X_n))(\alpha.\mathcal D_\lambda) = \mathrm e_l(\mathbf X_2,\ldots, \mathbf X_n)(\mathcal D_\lambda).
    \]
    Here, by $\alpha.\mathbf X_l$ we mean the operator defined as
    \[
    \alpha.\mathbf X_k(\rho) = \sum_{\begin{smallmatrix}
        i \in [\bar n ]\\ i<k
    \end{smallmatrix}}\widehat{\alpha(i\,k)\alpha^{-1}}.\rho .
    \]

    For any $\alpha \in H_n$ we have $\alpha.\mathcal D_\lambda = \mathcal D_\lambda$. 

    Recall
    \[
    e_l(\mathbf X_2,\ldots,\mathbf X_n) = \sum_{\begin{smallmatrix}
        2\le k_1<\ldots<k_l\le n
    \end{smallmatrix}} 
    \mathbf X_{k_1}\cdots \mathbf X_{k_l}.
    \]

    In order to show that the result of evaluation $\mathrm e_l(\mathbf X_2,\ldots,X_n)(\mathcal D_\lambda)$ is invariant under the action of the group $H_n$ it is enough to check that for any $\alpha\in H_n$  and for $\rho  \in d_\lambda$ such that 
    \[
    \widehat{\alpha(i_1~k_1)\alpha^{-1}}\cdots \widehat{\alpha(i_l~k_l)\alpha^{-1}}.\rho = P\in M_{n}(C,J,T)
    \]
    one can find a sequence $i_1',\ldots,i_l'$ and $\varrho \in \mathcal D_\lambda$ with
    \[
    \widehat{(i'_1~k_1)}\cdots \widehat{(i'_l~k_l)}.\varrho = P,
    \]
    and moreover,  a correspondence between the sequences $(i_1,\ldots,i_l)$ and $(i_1',\ldots ,i_l')$ is bijective.

    The group $H_n$ is generated by the transposition $(1~\bar 1)$ and products of pairs of transpositions $(k~k+1)(\bar k~\widebar{k+1})$ for all $1 \le k \le n-1.$ 

    The conjugation by $(1~\bar 1)$  just shuffles the terms of all the $\mathbf X_l,$ so the action by this element does not change the operators. So we have for all $\rho \in \mathcal D_\lambda$
    \[
    (1~\bar 1).\left(\mathbf X_{k_1}\cdots \mathbf X_{k_l}(\rho)\right) = \mathbf X_{k_1}\cdots \mathbf X_{k_l}\left( (1~\bar 1).\rho\right).
    \]

    Now, consider the conjugation by $\alpha = (k~k+1)(\bar k~\widebar{k+1})$. Due to Proposition~\ref{th:comm}, we can always rearrange the operators in every term of the sum $\sum_{\begin{smallmatrix}
        2\le k_1<\ldots<k_l\le n
    \end{smallmatrix}} 
    \mathbf X_{k_1}\cdots \mathbf X_{k_l}$ in such a way, that $\mathbf X_{k}$ and $\mathbf X_{k+1}$ are on the rightmost positions. Now consider the cases:
    \begin{itemize}
    \item If neither of $\mathbf X_k,\mathbf X_{k+1}$ is involved in the term, then the conjugation by $\alpha$ just shuffles the summands in expressions for all the remaining $\mathbf X_{l}.$ Thus the consideration of this case is similar to the case of conjugation by $(1~\bar 1)$.
    \item The terms with only one of the $\mathbf X_{k},\mathbf X_{k+1}$  involved can be grouped as $\mathbf P (\mathbf X_k + \mathbf X_{k+1}),$ where $\mathbf P$ is a product of refined Jucys-Murphy elements. The conjugation of $\mathbf P$ by $\alpha$ just lead to reshuffle of the terms in the its factors, so $\alpha.\left(\mathbf P(\rho)\right) = \mathbf P(\alpha.\rho)$. However, the conjugation of $\mathbf X_k + \mathbf X_{k+1}$ by $\alpha$ preserves almost all the terms of the sum except that $(\bar k~k+1)$ gets replaced by $(k~\widebar{k+1}).$ So
    \begin{gather}
    \widehat{(k~\widebar{k+1})}.\rho = w\left(\rho,(k~\widebar{k+1})\right)(k~\widebar{k+1})).\rho\\
    =w\left(\rho,(k~\widebar{k+1})\right)(k~\widebar{k+1})(k~\widebar{k+1})(\bar k~k+1).\left((\bar k~k+1)(k~\widebar{k+1}).\rho\right) \\= w\left(\rho,(k~\widebar{k+1})\right)(\bar k~k+1).\left((\bar k~k+1)(k~\widebar{k+1}).\rho \right). 
    \end{gather}
    We see that as $(\bar k~k+1)(k~\widebar{k+1})\in H_n,$ the element $\varrho = (\bar k~k+1)(k~\widebar{k+1}).\rho$ has the same type as $\rho.$ Also, $(\bar k~k+1).\left((\bar k~k+1)(k~\widebar{k+1}).\rho \right) = (k~\widebar{k+1}).\rho)$, so the we have equality of the weights $w\left(\rho,(k~\widebar{k+1})\right) = w\left(\varrho, (\bar k~k+1) \right)$. Thus 
    
    \begin{equation}\label{eq:bar}
    \widehat{(k~\widebar{k+1})}.\rho = \widehat{(\bar k~k+1)}.\varrho.
    \end{equation}
    
    As the correspondence $\rho \mapsto \varrho$ is bijective, the case follows.

    \item If the both $\mathbf X_k, \mathbf X_{k+1}$ are involved, so we consider the term $\mathbf P\mathbf X_k \mathbf X_{k+1}.$ Conjugation by $\alpha$ does not change $\mathbf P.$ 

    Due to Proposition~\ref{th:comm} we have
    \[
    [\mathbf X_k,\mathbf X_{k+1}](\mathcal D_\lambda)  = 0.
    \]

    It means that for every triple $(\sigma_0, \sigma_1,\rho)$ with $ \sigma_0 = (i~k), \sigma_1 = (j~k+1), \rho \in d_\lambda $ we can bijectively associate a triple $(\varsigma_0,\varsigma_1,\varrho)$ with $\varsigma_0 = (i'~k), \varsigma_1 = (j'~k+1), \varrho \in d_\lambda$ such that
    \[
    \widehat {\sigma_0 }\widehat {\sigma_1 }.\rho = \widehat {\varsigma_1 }\widehat {\varsigma_0 }.\varrho.
    \]
    Denote the bijection mapping $(\sigma_0,\sigma_1,\rho)$ to $(\varsigma_0,\varsigma_1,\varrho)$ by the letter $\mathbf{B}.$ We have the following subcases:

    \begin{itemize}
        \item for any $i,j \in [\widebar{k-1}]$ we have
        \[
        (\alpha.\widehat{(i~k)})(\alpha.\widehat{(j~k+1)}).\rho = \widehat{(i~k+1)}\widehat{(j~k)}.\rho 
        \]
        Evaluating $\mathbf B
        \left( (i~k+1),(j~k),\rho\right)$ produces the counterpart with the correct order of transpositions.
        \item for any $j\in [\widebar{k-1}]$ we have
        \[
         (\alpha.\widehat{(j~k)})(\alpha.\widehat{(k~k+1)}).\rho = \widehat{(j~k+1)}\widehat{(k~k+1)}.\rho.
        \]
        As $(j~k+1)(k~k+1) = (k~k+1)(j~k)$, we have 
        \begin{itemize}
        
        \item if $CJ \widehat{(j~k+1)}\widehat{(k~k+1)}.\rho = T^2 \widehat{(k~k+1)}\widehat{(j~k)}.\rho$, then $\rho\cup\tau$ has a cycle of the following structure:

        \begin{center}
        \begin{tikzpicture}

    \begin{scope}
    \draw[dotted] (0,0) circle (1cm);


        \draw[thick,blue,opacity=1] (-10:1cm) arc[start angle=-10, end angle=10, radius=1cm];
        \draw[thick,red,opacity=1] (10:1cm) arc[start angle=10, end angle=30, radius=1cm];
        \draw[thick,red,opacity=1] (130:1cm) arc[start angle=130, end angle=150, radius=1cm];
        \draw[thick,blue,opacity=1] (150:1cm) arc[start angle=150, end angle=170, radius=1cm];
        \draw[thick,red,opacity=1] (250:1cm) arc[start angle=250, end angle=270, radius=1cm];
        \draw[thick,blue,opacity=1] (270:1cm) arc[start angle=270, end angle=290, radius=1cm];
        
        \filldraw (10:1cm) circle (2pt); 
        \filldraw (150:1cm) circle (2pt);
        \filldraw (270:1cm) circle (2pt);

        \node[above right] at (10:1cm) {$k+1$}; 
        \node[above left] at (150:1cm) {$k$};
        \node[below] at (270:1cm) {$j$};

    \end{scope}
    \end{tikzpicture}
    \end{center}
    The fixed point-free involution $\rho' = (j~k)(\bar k~k+1)(j~k+1)(k~k+1).\rho$ differs from $\rho$ by four twists, so it has the same type. Clearly, in this case, we have
    \[
    \widehat{(\bar k~k+1)}\widehat{(j~k)}.\rho' = \widehat{(j~k+1)}\widehat{(k~k+1)}.\rho,
    \]
    and $\rho'\cup\tau$ has a cycle of the following structure:
            \begin{center}
        \begin{tikzpicture}

    \begin{scope}
    \draw[dotted] (0,0) circle (1cm);


        \draw[thick,blue,opacity=1] (-10:1cm) arc[start angle=-10, end angle=10, radius=1cm];
        \draw[thick,red,opacity=1] (10:1cm) arc[start angle=10, end angle=30, radius=1cm];
        \draw[thick,red,opacity=1] (130:1cm) arc[start angle=130, end angle=150, radius=1cm];
        \draw[thick,blue,opacity=1] (150:1cm) arc[start angle=150, end angle=170, radius=1cm];
        \draw[thick,red,opacity=1] (250:1cm) arc[start angle=250, end angle=270, radius=1cm];
        \draw[thick,blue,opacity=1] (270:1cm) arc[start angle=270, end angle=290, radius=1cm];
        
        \filldraw (10:1cm) circle (2pt); 
        \filldraw (150:1cm) circle (2pt);
        \filldraw (270:1cm) circle (2pt);

        \node[above right] at (10:1cm) {$j$}; 
        \node[above left] at (150:1cm) {$k$};
        \node[below] at (270:1cm) {$k+1$};

    \end{scope}
    \end{tikzpicture}
    \end{center}
    Evaluating $\mathbf B\left((\bar k~k+1), (j~k),\rho'\right)$ produces the operators positioned in the desired order.

    \item if $T^2 \widehat{(j~k+1)}\widehat{(k~k+1)}.\rho = CJ \widehat{(k~k+1)}\widehat{(j~k)}.\rho$, then $\rho\cup\tau$ has a cycle of the following structure:
     
    \begin{center}
    \begin{tikzpicture}

    \begin{scope}
    \draw[dotted] (0,0) circle (1cm);


        \draw[thick,blue,opacity=1] (-10:1cm) arc[start angle=-10, end angle=10, radius=1cm];
        \draw[thick,red,opacity=1] (10:1cm) arc[start angle=10, end angle=30, radius=1cm];
        \draw[thick,blue,opacity=1] (130:1cm) arc[start angle=130, end angle=150, radius=1cm];
        \draw[thick,red,opacity=1] (150:1cm) arc[start angle=150, end angle=170, radius=1cm];
        \draw[thick,red,opacity=1] (250:1cm) arc[start angle=250, end angle=270, radius=1cm];
        \draw[thick,blue,opacity=1] (270:1cm) arc[start angle=270, end angle=290, radius=1cm];
        
        \filldraw (10:1cm) circle (2pt); 
        \filldraw (150:1cm) circle (2pt);
        \filldraw (270:1cm) circle (2pt);

        \node[above right] at (10:1cm) {$k$}; 
        \node[above left] at (150:1cm) {$k + 1$};
        \node[below] at (270:1cm) {$j$};

    \end{scope}
    \end{tikzpicture}
    \end{center}
    The fixed point free involution $\rho' = (j k)(\bar k~k+1)(j~k+1)(k~k+1).\rho$ differs from $\rho$ by four twists as $(j~k+1)(k~k+1) = (k~k+1)(j~k)$, and $(j~k)$ commutes with $(\bar k~k+1).$ The union $\rho'\cup \tau$ has a cycle of the following structure:
       
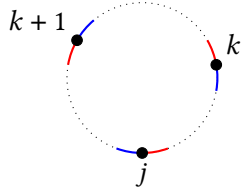
\begin{center}
    \begin{tikzpicture}

    \begin{scope}
    \draw[dotted] (0,0) circle (1cm);


        \draw[thick,blue,opacity=1] (-10:1cm) arc[start angle=-10, end angle=10, radius=1cm];
        \draw[thick,red,opacity=1] (10:1cm) arc[start angle=10, end angle=30, radius=1cm];
        \draw[thick,blue,opacity=1] (130:1cm) arc[start angle=130, end angle=150, radius=1cm];
        \draw[thick,red,opacity=1] (150:1cm) arc[start angle=150, end angle=170, radius=1cm];
        \draw[thick,blue,opacity=1] (250:1cm) arc[start angle=250, end angle=270, radius=1cm];
        \draw[thick,red,opacity=1] (270:1cm) arc[start angle=270, end angle=290, radius=1cm];
        
        \filldraw (10:1cm) circle (2pt); 
        \filldraw (150:1cm) circle (2pt);
        \filldraw (270:1cm) circle (2pt);

        \node[above right] at (10:1cm) {$k$}; 
        \node[above left] at (150:1cm) {$k + 1$};
        \node[below] at (270:1cm) {$j$};

    \end{scope}
    \end{tikzpicture}
    \end{center}
    And
    
    \[
    \widehat{(\bar k~k+1)}\widehat{(j~k)}.\rho' = \widehat{(j~k+1)}\widehat{(k~k+1)}.\rho.
    \]
    Evaluating $\mathbf B((\bar k~k+1),(j~k),\rho')$ produces the operators positioned in the desired order.
        
        \item Otherwise, $\widehat{(j~k+1)}\widehat{(k~k+1)}.\rho = \widehat{(k~k+1)}\widehat{(j~k)} .\rho$. Evaluating $\mathbf B((k~k+1),(j~k),\rho)$ produces the operators positioned in the correct order. 
        
        \end{itemize}

        \item Similarly, for an element $j\in [\widebar{k-1}]$ and a fixed point-free involution $\rho'''$ we have
        \[
        (\alpha.\widehat{(\bar j~k)})(\alpha.\widehat{(\bar k~k+1)}).\rho''' = \widehat{(\bar j~k+1)}\widehat{(k~\widebar{k+1})}.\rho'''.
        \]
        Formula~\ref{eq:bar} suggests, that for any given element $\rho'''$ we can bijectively find a fixed point-free involution $\rho''$ with  
        \[
        \widehat{(\bar j~k+1)}\widehat{(k~\widebar{k+1})}.\rho''' = \widehat{(\bar j~k+1)}\widehat{(\bar k~k+1)}.\rho''.
        \]
        Now we have $(\bar j~k+1)(\bar k~k+1) = (\bar k~k+1)(\bar j~\bar k)$. So we consider the cases:
        \begin{itemize}
        \item If $CJ\widehat{(\bar j~k+1)}\widehat{(\bar k~k+1)}.\rho'' = T^2\widehat{(\bar k~k+1)}\widehat{(\bar j~\bar k)}.\rho''$, then $\rho''\cup\tau$ has a cycle of the following structure:
        
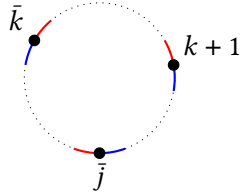
\begin{center}
        \begin{tikzpicture}

    \begin{scope}
    \draw[dotted] (0,0) circle (1cm);


        \draw[thick,blue,opacity=1] (-10:1cm) arc[start angle=-10, end angle=10, radius=1cm];
        \draw[thick,red,opacity=1] (10:1cm) arc[start angle=10, end angle=30, radius=1cm];
        \draw[thick,red,opacity=1] (130:1cm) arc[start angle=130, end angle=150, radius=1cm];
        \draw[thick,blue,opacity=1] (150:1cm) arc[start angle=150, end angle=170, radius=1cm];
        \draw[thick,red,opacity=1] (250:1cm) arc[start angle=250, end angle=270, radius=1cm];
        \draw[thick,blue,opacity=1] (270:1cm) arc[start angle=270, end angle=290, radius=1cm];
        
        \filldraw (10:1cm) circle (2pt); 
        \filldraw (150:1cm) circle (2pt);
        \filldraw (270:1cm) circle (2pt);

        \node[above right] at (10:1cm) {$k+1$}; 
        \node[above left] at (150:1cm) {$\bar k$};
        \node[below] at (270:1cm) {$\bar j$};

    \end{scope}
    \end{tikzpicture}
    \end{center}
    And setting $\rho' = (\bar j~\bar k)(k~k+1)(\bar j~k+1)(\bar k~k+1).\rho''$ we have
    \[
    \widehat{(k~k+1)}\widehat{(\bar j~\bar k)}.\rho' = \widehat{(\bar j~k+1)}\widehat{(\bar k~k+1)}.\rho''
    \]
    The union $\rho'\cup \tau$ should have a cycle of the following structure:
            
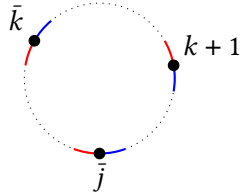
\begin{center}
        \begin{tikzpicture}

    \begin{scope}
    \draw[dotted] (0,0) circle (1cm);


        \draw[thick,blue,opacity=1] (-10:1cm) arc[start angle=-10, end angle=10, radius=1cm];
        \draw[thick,red,opacity=1] (10:1cm) arc[start angle=10, end angle=30, radius=1cm];
        \draw[thick,blue,opacity=1] (130:1cm) arc[start angle=130, end angle=150, radius=1cm];
        \draw[thick,red,opacity=1] (150:1cm) arc[start angle=150, end angle=170, radius=1cm];
        \draw[thick,red,opacity=1] (250:1cm) arc[start angle=250, end angle=270, radius=1cm];
        \draw[thick,blue,opacity=1] (270:1cm) arc[start angle=270, end angle=290, radius=1cm];
        
        \filldraw (10:1cm) circle (2pt); 
        \filldraw (150:1cm) circle (2pt);
        \filldraw (270:1cm) circle (2pt);

        \node[above right] at (10:1cm) {$k+1$}; 
        \node[above left] at (150:1cm) {$\bar k$};
        \node[below] at (270:1cm) {$\bar j$};

    \end{scope}
    \end{tikzpicture}
    \end{center}
    Finally, we notice that as $(\bar j~\bar k)(j~k) \in H_n$ we can set
    \[\rho = (j~k)(\bar j~\bar k).\rho',\]
    an similarly to the formula~\ref{eq:bar} we get
    \[
    \widehat{(k~k+1)}\widehat{(\bar j~\bar k)}.\rho' = \widehat{(k~k+1)}\widehat{( j~ k)}.\rho,
    \]
    where the union $\rho\cup \tau$ has a cycle with the following structure:
    
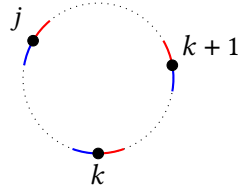
\begin{center}
     \begin{tikzpicture}

    \begin{scope}
    \draw[dotted] (0,0) circle (1cm);


        \draw[thick,blue,opacity=1] (-10:1cm) arc[start angle=-10, end angle=10, radius=1cm];
        \draw[thick,red,opacity=1] (10:1cm) arc[start angle=10, end angle=30, radius=1cm];
        \draw[thick,red,opacity=1] (130:1cm) arc[start angle=130, end angle=150, radius=1cm];
        \draw[thick,blue,opacity=1] (150:1cm) arc[start angle=150, end angle=170, radius=1cm];
        \draw[thick,blue,opacity=1] (250:1cm) arc[start angle=250, end angle=270, radius=1cm];
        \draw[thick,red,opacity=1] (270:1cm) arc[start angle=270, end angle=290, radius=1cm];
        
        \filldraw (10:1cm) circle (2pt); 
        \filldraw (150:1cm) circle (2pt);
        \filldraw (270:1cm) circle (2pt);

        \node[above right] at (10:1cm) {$k+1$}; 
        \node[above left] at (150:1cm) {$ j$};
        \node[below] at (270:1cm) {$ k$};

    \end{scope}
    \end{tikzpicture}
    \end{center}
    Evaluating $\mathbf B\left((k~k+1),(j~k),\rho\right)$ produces the operator positioned in the desired order.
    \item If $T^2\widehat{(\bar j~k+1)}\widehat{(\bar k~k+1)}.\rho'' = CJ\widehat{(\bar k~k+1)}\widehat{(\bar j~\bar k)}.\rho''$, then $\rho''\cup\tau$ has a cycle of the following structure:
   
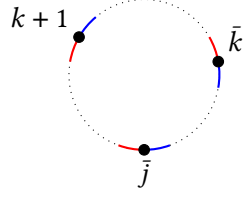
\begin{center}
    \begin{tikzpicture}

    \begin{scope}
    \draw[dotted] (0,0) circle (1cm);


        \draw[thick,blue,opacity=1] (-10:1cm) arc[start angle=-10, end angle=10, radius=1cm];
        \draw[thick,red,opacity=1] (10:1cm) arc[start angle=10, end angle=30, radius=1cm];
        \draw[thick,blue,opacity=1] (130:1cm) arc[start angle=130, end angle=150, radius=1cm];
        \draw[thick,red,opacity=1] (150:1cm) arc[start angle=150, end angle=170, radius=1cm];
        \draw[thick,red,opacity=1] (250:1cm) arc[start angle=250, end angle=270, radius=1cm];
        \draw[thick,blue,opacity=1] (270:1cm) arc[start angle=270, end angle=290, radius=1cm];
        
        \filldraw (10:1cm) circle (2pt); 
        \filldraw (150:1cm) circle (2pt);
        \filldraw (270:1cm) circle (2pt);

        \node[above right] at (10:1cm) {$\bar k$}; 
        \node[above left] at (150:1cm) {$k + 1$};
        \node[below] at (270:1cm) {$\bar j$};

    \end{scope}
    \end{tikzpicture}
    \end{center}
       The fixed point free involution $\rho' = (\bar j~\bar k)( k~k+1)(\bar j~k+1)(\bar k~k+1).\rho''$ differs from $\rho''$ by four twists as $(j~k+1)(k~k+1) = (k~k+1)(j~k)$, and $(j~k)$ commutes with $(\bar k~k+1).$ The union $\rho'\cup \tau$ has a cycle of the following structure:
       
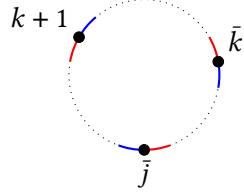
\begin{center}
    \begin{tikzpicture}

    \begin{scope}
    \draw[dotted] (0,0) circle (1cm);


        \draw[thick,blue,opacity=1] (-10:1cm) arc[start angle=-10, end angle=10, radius=1cm];
        \draw[thick,red,opacity=1] (10:1cm) arc[start angle=10, end angle=30, radius=1cm];
        \draw[thick,blue,opacity=1] (130:1cm) arc[start angle=130, end angle=150, radius=1cm];
        \draw[thick,red,opacity=1] (150:1cm) arc[start angle=150, end angle=170, radius=1cm];
        \draw[thick,blue,opacity=1] (250:1cm) arc[start angle=250, end angle=270, radius=1cm];
        \draw[thick,red,opacity=1] (270:1cm) arc[start angle=270, end angle=290, radius=1cm];
        
        \filldraw (10:1cm) circle (2pt); 
        \filldraw (150:1cm) circle (2pt);
        \filldraw (270:1cm) circle (2pt);

        \node[above right] at (10:1cm) {$\bar k$}; 
        \node[above left] at (150:1cm) {$k + 1$};
        \node[below] at (270:1cm) {$\bar j$};

    \end{scope}
    \end{tikzpicture}
    \end{center}
    And
    
    \[
    \widehat{( k~k+1)}\widehat{(\bar j~\bar k)}.\rho' = \widehat{(\bar j~k+1)}\widehat{(\bar k~k+1)}.\rho''.
    \]

    Similarly to the previous case, we set
    \[\rho = (j~k)(\bar j~\bar k).\rho',\]
    and
    \[
    \widehat{(k~k+1)}\widehat{(\bar j~\bar k)}.\rho' = \widehat{(k~k+1)}\widehat{( j~ k)}.\rho,
    \]
    where the union $\rho\cup \tau$ has a cycle with the following structure:
           
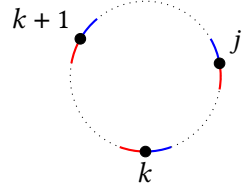
\begin{center}
    \begin{tikzpicture}

    \begin{scope}
    \draw[dotted] (0,0) circle (1cm);


        \draw[thick,red,opacity=1] (-10:1cm) arc[start angle=-10, end angle=10, radius=1cm];
        \draw[thick,blue,opacity=1] (10:1cm) arc[start angle=10, end angle=30, radius=1cm];
        \draw[thick,blue,opacity=1] (130:1cm) arc[start angle=130, end angle=150, radius=1cm];
        \draw[thick,red,opacity=1] (150:1cm) arc[start angle=150, end angle=170, radius=1cm];
        \draw[thick,red,opacity=1] (250:1cm) arc[start angle=250, end angle=270, radius=1cm];
        \draw[thick,blue,opacity=1] (270:1cm) arc[start angle=270, end angle=290, radius=1cm];
        
        \filldraw (10:1cm) circle (2pt); 
        \filldraw (150:1cm) circle (2pt);
        \filldraw (270:1cm) circle (2pt);

        \node[above right] at (10:1cm) {$ j$}; 
        \node[above left] at (150:1cm) {$k + 1$};
        \node[below] at (270:1cm) {$k$};

    \end{scope}
    \end{tikzpicture}
    \end{center}
    
    Evaluating $\mathbf B(( k~k+1),(j~k),\rho')$ produces the operators positioned in the desired order.

    \item Otherwise, we have 
    $\widehat{(\bar j~k+1)}\widehat{(\bar k~k+1)}.\rho'' = \widehat{(\bar k~k+1)}\widehat{(\bar j~\bar k)}.\rho''$. Like in the previous cases set
    \[\rho = (j~k)(\bar j~\bar k).\rho'',\]
    so that 
    \[\widehat{(\bar k~k+1)}\widehat{(\bar j~\bar k)}.\rho'' = \widehat{(\bar k~k+1)}\widehat{( j~k)}.\rho \]
    and evaluate $\mathbf B\left((\bar k~k+1),( j~k), \rho\right). $

        \end{itemize}
        
    \end{itemize}

    \end{itemize}
    Finally, we have to check that the cases do not overlap. Here, we provide only one check: the rest are done similarly. Consider the case
    \[
    CJ \widehat{(j~k+1)}\widehat{(k~k+1)}.\rho = T^2 \widehat{(k~k+1)}\widehat{(j~k)}.\rho
    \]
    then consider the result of the conjugation of the produced element $\rho = (j~k)(\bar j~\bar k).\rho'$  It is straightforward to check that $\widehat{(\bar j~k+1)}\widehat{(\bar k~k+1)}.\rho \ne \widehat{(\bar k~k+1)}\widehat{(\bar j~\bar k)}.\rho$ 
    
    This finishes the proof.   
\end{proof}
Notice, that the results of this section rely on the fact that we work over $\mathbb C(C,J,T),$ not $\mathbb C[C,J,T]$. However, as the coefficients of $\mathbf X_{k_p}\cdots \mathbf X_{k_1}(\mathcal D_\lambda)$ are \emph{polynomials} is $C,J,T$, we can conclude, that even in the case of arbitrary specialization of the variables $C,J,T$, the refined odd Jucys-Murphy elements commute, and define an action of the ring of symmetric functions on a Fock space.

We refer to the constructed action of the ring of symmetric functions on $I(C,J,T)$  as to the \emph{preliminary refined action}. The table in Appendix~\ref{sec:pdata} provides examples of the values of the elementary symmetric functions evaluated on the refined Jucys-Murphy elements applied to the type indicators.

\section{Refined monotone Hurwitz numbers}\label{sc:ref-mon}

Any symmetric function is a polynomial in complete homogeneous symmetric functions $\mathrm h_k$. Given an action of $\Lambda$ on a Fock space $\mathcal F$, is enough to know the action of the complete homogenous functions on the basis elements to compute the action of any symmetric polynomial uniquely in a way similar to the computation of Equation~\ref{eq:strcoef}. 

In this section, we produce a recursion that allows us to determine the structure coefficients of the refined action of $\mathrm h_k$. 

The recursion we construct computes the \emph{connected} numbers. By the usual inclusion-exclusion argument, the generating functions for connected and disconnected numbers differ by exponentiation, and the corresponding numbers are related by the formulas~\ref{eq:dis-to-conn} and~\ref{eq:conn-to-dis}. A  reason for considering the connectedness condition is that for the Schur and zonal action the corresponding numbers are referred to as the \emph{monotone Hurwitz numbers,} that play a crucial role in Weingarten calculus, and admit a nice generalization to the $b-$deformed case (see~\cite{Coulter-Do} and other relevant references in the Introduction).



First, following~\cite{Karev-Do} we define a notion of a \emph{monotone twisted factorisation}.





\begin{definition}\label{def:twfac}

For $k\in \mathbb N\cup \{0\},$
and $\mu,\nu \vdash n$ a  twisted factorisation  of type $(k,\mu,\nu)$ is  a pair  
$(\rho,(\sigma_1,\ldots,\sigma_k))$ consisting of an element  $\rho\in d_\nu$, and $\sigma_1,\ldots,\sigma_k$ --- a sequence of transpositions of $S_{2n}$ that satisfy:
\begin{itemize}
\item For every $i$ the transposition $\sigma_i$ is  the form $\sigma_i = (p_i\, q_i),$ with $q_i \in [n],$ $p_i \in [\widebar{q_i -1} ]$.
\item The type of the element $\sigma_k\cdots \sigma_1.\rho$ is $\mu.$
\end{itemize}
If we also impose the condition
\begin{itemize}
    \item For all $i > 1$ we have $q_{i-1} \le q_i$ (\emph{monotonicity condition}),
\end{itemize}
the corresponding factorisation is called \emph{monotone}.
Moreover, if 
\begin{itemize}
    \item The subgroup $\langle \tau,\rho,\sigma_1,\ldots,\sigma_k\rangle\le S_{2n}$ is transitive (\emph{transitivity condition}),
\end{itemize}
we call the factorisation \emph{transitive}.

Given a  twisted factorisation $\mathfrak f = (\rho,(\sigma_1,\ldots,\sigma_k)),$ with $\sigma_k\cdots\sigma_1.\rho = \varrho,$ we define its \emph{weight} to be $w(\mathfrak f) = [\varrho].\hat \sigma_k\cdots\hat \sigma_1.\rho\in \mathbb C(C,J,T).$ 
\end{definition}


Let us comment on the definition. The notion of the weight is a direct generalization of the CJT-weight of a pair of a fixed point-free involution and a transposition we have used so far. If we do not require the transitivity condition, the sum of weights of the monotone twisted factorisations for $\mu =(1^{m_1}2^{m_2}\ldots)$ equals
\[
\sum_{\mathfrak f\mbox{ is of type }(k,\mu,\nu) } w(\mathfrak f) = \sum_{\varrho\in d_\mu} [\varrho].\mathrm h_k(\mathbf X_2,\ldots, \mathbf X_n)(\mathcal D_\nu) = \frac {2^{|\mu|}|\mu|!}{\prod (2j)^{m_j}m_!} [\mathcal D_\mu].\mathrm h_k(\mathbf X_2,\ldots, \mathbf X_n)(\mathcal D_\nu), 
\]
where $\mathrm h_k(\mathbf X_2,\ldots,X_n)$ is the $k'$th  complete homogeneous symmetric function evaluated on the set of preliminary refined Jucys-Murphy elements, and the rightmost expression follows due to the fact that $\mathrm h_k(\mathbf X_2,\ldots, \mathbf X_n)$ preserves the space of type indicators. The transitivity condition is inspired by the real Hurwitz theory: coverings that correspond to transitive factorisations are irreducible (see~\cite{Guay-Paquet-Markwig-Rau,Lozhkin}).

In the framework of the Hurwitz theory, it is more natural to use  \emph{genus} $g$ instead of the number of transpositions $k$. We say that a transitive factorisation of type $(k,\mu,\nu)$ is of \emph{genus $0$}, if $k$ is the smallest possible number of transpositions needed to transform a fixed point-free involution of type $\nu$ to a fixed point-free involution of type $\mu$ in a transitive way. 

\begin{proposition}
    A genus $0$ factorisation of type $(k,\mu,\nu)$ has $k = l(\mu) + l(\nu) - 2.$
\end{proposition}
\begin{proof}
    The number of transpositions needed for a factorisation does not depend on a particular choice of a representative of type $\nu$. Every applied transposition either makes a cut, or a join, or a twist, but as we want the number of transpositions to be minimal, twists are not present. So the problem reduces to the problem of cutting and joining the cycles in a transitive way and is similar to the well-known problem of minimizing the number of transpositions in the Schur case. 
\end{proof}

The genus of a factorisation measures the deviation of the number of transpositions from the minimal possible one.

\begin{definition}
    The genus of a transitive  factorisation of type $(k,\mu,\nu)$ is the number $g = \frac{k + 2 - l(\mu) - l(\nu)}{2}. $
\end{definition}

Notice, that unlike the complex Hurwitz theory, we allow half-integer values of the genus.

\begin{remark}
    As the reader realises, the theory we are going to develop, can be specialized and used for enumeration of purely real covers. In the purely real case, the genus, we introduce, is related to the Euler characteristic $\chi(\Sigma)$ of the covering real surface $(\Sigma,\mathcal T)$ by the formula $g = 1 - \chi(\Sigma)/4,$ and actually reflects the topological properties of $\Sigma$ quotiented by the action of real structure $\mathcal T$. 
\end{remark}

\begin{proposition}\label{prop:C<->J}
    Let a  transitive factorisation of type $(k,\mu,\nu)$ be of genus $g$ and involve $t$ twists. Then it involves $c = l(\mu) + g - \frac t2 - 1$ cuts and $j = l(\nu) + g - \frac t2 - 1$ joins.
\end{proposition}
\begin{proof}
    Every cut increases the number of parts by 1, and every join decreases the number of parts by 1. So we have $l(\nu) + c - j = l(\mu).$ On the other hand, we have $c + j + t = 2g - 2 + l(\mu) + l(\nu).$ Solving the system of linear equations for $c$ and $j,$ we obtain the result.
\end{proof}

Every monotone factorisation $\mathfrak f = (\rho,(\sigma_1,\ldots,\sigma_k))$ of type $(k,\mu,\nu)$ gives rise to an \emph{opposite factorisation}
\[
\mathfrak f' = (\varrho,(\sigma_k,\ldots,\sigma_1))
\] 
of type $(k,\nu,\mu)$ respectively. The weights of $\mathfrak f$ and $\mathfrak f'$ are related:
\[
w(\mathfrak f) = \left. w(\mathfrak f')\right|_{C \leftrightarrow J},
\]
 where $C\leftrightarrow J$ in the subscript denotes the application of the automorphism of $\mathbb C(C,J,T)$ that interchanges $C$ and $J.$ Notice that a factorisation $\mathfrak f$ is transitive if and only if the opposite factorisation $\mathfrak f'$ is transitive. 
Since the preliminary refined Jucys-Murphy elements restricted to the space of type indicators commute (Proposition \ref{th:comm}), we have the following equality:
\[
\sum_{\varrho \in d_\mu } [\varrho].\mathrm h_k(\mathbf X_2,\ldots,\mathbf X_n)(\mathcal D_\nu) = 
\sum_{\rho \in d_\nu } \left.[\rho].\mathrm h_k(\mathbf X_2,\ldots,\mathbf X_n)(\mathcal D_\mu)\right|_{C \leftrightarrow J},
\]
or
\begin{equation}\label{eq:symmetry}
\frac {2^{|\mu|}|\mu|!}{\prod (2j)^{m_j}m!} [\mathcal D_\mu].\mathrm h_k(\mathbf X_2,\ldots, \mathbf X_n)(\mathcal D_\nu) = \frac {2^{|\nu|}|\nu|!}{\prod (2j)^{n_j}n_j!} \left.[\mathcal D_\nu].\mathrm h_k(\mathbf X_2,\ldots, \mathbf X_n)(\mathcal D_\mu)\right|_{C\leftrightarrow J}, 
\end{equation}
where we assume that  the partition $\nu = (1^{n_1}2^{n_2}\ldots).$ Proposition~\ref{prop:C<->J} implies, that the expression on the left-hand side of the equality is of the form $\sum_{t\ge 0} p_t T^t C^{l(\mu)+g - \frac t2 - 1}J^{l(\nu)+g - \frac t2 - 1}$, where $p_t$ is a collection of rational numbers, while the expression on the right-hand side is $\sum_{t\ge 0} q_t T^t C^{l(\nu)+g - \frac t2 - 1}J^{l(\mu)+g - \frac t2 - 1}$ for a collection of rational numbers $q_t$. Given, that the automorphism $C\leftrightarrow J$ interchanges these expressions, we obtain:
\begin{proposition}\label{prop:symmetry}
For any $k\in \mathbb N,$ $\mu = (1^{m_1}2^{m_2}\ldots), \nu = (1^{n_1}2^{n_2}\ldots)$ the equality 
\[
\frac {1}{\prod (2jC)^{m_j}m!} [\mathcal D_\mu].\mathrm h_k(\mathbf X_2,\ldots, \mathbf X_n)(J^{-l(\nu)}\mathcal D_\nu ) = \frac {1}{\prod (2jC)^{n_j}n_j!} [\mathcal D_\nu].\mathrm h_k(\mathbf X_2,\ldots, \mathbf X_n)(J^{-l(\mu)}\mathcal D_\mu )
\]holds.
\end{proposition}

Now we are in the position to give the principal definition of this section:
\begin{definition}
    For partitions $\mu,\nu\vdash n,$ and a number $g \in \frac 12\mathbb N \cup \{0\}$ the refined monotone Hurwitz numbers $H^{\le}_g\left(\begin{smallmatrix}\nu\\ \mu\end{smallmatrix}\right)$ equals the weighted number of transitive monotone factorisations of type $({2g - 2 +l(\mu) + l(\nu)},\mu,\nu)$ divided by $2^n n! C^{l(\mu)} J^{l(\nu)}.$ 
\end{definition}

If we don't require the transitivity we get the notion of the \emph{disconnected refined monotone Hurwitz numbers} 
\begin{equation}\label{eq:monHurnumdef}
H^{\le,
\bullet}_g\left(\begin{smallmatrix}\nu\\ \mu\end{smallmatrix}\right) = \frac {1}{\prod (2jC)^{m_j}m_j!} [\mathcal D_\mu].\mathrm h_k(\mathbf X_2,\ldots, \mathbf X_n)(J^{-l(\nu)}\mathcal D_\nu ),\end{equation} where $k$ is related to $g$ by $k ={2g - 2 +l(\mu) + l(\nu)}.$ The disconnected numbers can be non-zero even for negative values of the genus. The relation between disconnected and transitive Hurwitz numbers is through the inclusion-exclusion formula (see Equations~\ref{eq:dis-to-conn},\ref{eq:conn-to-dis}).

Using the relation between the refined monotone Hurwitz numbers and the symmetric functions $\mathrm h_k$, we can repeat verbatim the argument from Lemma 7 of~\cite{Karev-Do}, and obtain:

\begin{proposition}\label{prop:rhoindependence}
The weighted number of monotone twisted factorisations $  \sigma_k\cdots  \sigma_1.\rho$, where $\rho$ is a fixed element of $d_\nu$, and $ \sigma_k\cdots \sigma_1.\rho\in d_\mu$, depends on $\nu$ only.  Moreover, the result still holds if we restrict to transitive monotone factorisations.
\end{proposition}


Using this proposition, we can arrange the computation of a refined monotone Hurwitz number $H_g^{\le}\left(\begin{smallmatrix}\nu\\ \mu\end{smallmatrix}\right)$ by the method proposed in~\cite{Karev-Do}: for the given $\nu = (1^{n_1}2^{n_2}\ldots) \vdash n $, choose its convenient representative in $d_\nu$, count the weighted number of transitive monotone factorisations of type $(k,\mu,\nu)$, and multiply the result by $\frac1{C^{l(\mu)}  \prod_j (2Jj)^{n_j}n_j!}$.

To break the symmetry that can interfere the analysis, we are going to use \emph{compositions} instead of partitions. 
 Let $\mathbf n = (n_1,\ldots,n_p)$ be a composition of the number $ n = n_1 + \cdots n_p$ of length $p$, with the partition $\nu$ corresponding to it. The  part $n_1$ of $\mathbf n$ corresponds to the product of transpositions
$(\bar 1~2)(\bar 2~3)\cdots(\bar n_1~ 1)$.
We refer to $n_1$ as \emph{the last used index}.

Now, suppose we have already considered the involution that corresponds to parts $(n_1,n_2,\dots,n_{r-1})$ of $\underline  n$, and the next part is the part of length $n_r$. Let the last used index is  $ l $.
Thus for $n_r$ we add to the involution we construct the product of transpositions
\[
(\widebar{l+1}~l+2)(\widebar{l+2}~l+3)\cdots(\widebar{l+n_r}~l+1,)
\]
 and redefine the last used index to be $l + n_r$. Proceed until all the parts of $\mathbf n$ are used. We denote the obtained fixed point-free involution by $\rho_\mathbf n.$ 
As an example, taking the composition $(1,1,2,3)$, we have that $n_1=1$ correspond to the transposition $(\bar1\ 
 1)$, to $n_2=1$ corresponds $(\bar2\ 2)$, to $n_3=2$ corresponds $(\bar 3\ 4)(\bar4\ 3)$ and
 finally to $n_4=3$ corresponds $(\bar 5\ 6)(\bar 6 \ 7)(\bar 7\ 5)$,
 and the preferred representative is $(\bar1\  1) (\bar2\ 2)(\bar 3\ 4)(\bar4\ 3)(\bar 5\ 6)(\bar 6 \ 7)
(\bar 7\ 5)$, see Figure~\ref{fig:pref}. We denote the type representative we have constructed $\rho_{\underline n}$

\begin{figure}
\begin{tikzpicture}

    \begin{scope}[shift={(6,0)}]
    \draw[thick,blue] (0,0) circle (1cm);

    \foreach \angle in {0, 120, 240} {
        \draw[thick,red,opacity=1] (\angle:1cm) arc[start angle=\angle, end angle=\angle+60, radius=1cm];
        
        \filldraw (\angle:1cm) circle (2pt); 
        \filldraw (\angle+60:1cm) circle (2pt);

    }

        \node[above right] at (0:1cm) {$6$}; 
        \node[above right] at (60:1cm) {$\bar 5$};
        \node[above left] at (120:1cm) {$5$};
        \node[above left] at (180:1cm) {$\bar 7$};
        \node[below left] at (240:1cm) {$ 7$};
        \node[below right] at (300:1cm) {$\bar 6$};

    \end{scope}

        \begin{scope}[shift={(3,0)}]
    \draw[thick,blue] (0,0) circle (1cm);

    \foreach \angle in {0, 180} {
        \draw[thick,red,opacity=1] (\angle:1cm) arc[start angle=\angle, end angle=\angle+90, radius=1cm];
        
        \filldraw (\angle:1cm) circle (2pt); 
        \filldraw (\angle+90:1cm) circle (2pt);
    }

    \node[right] at (0:1cm) {$4$}; 
    \node[above] at (90:1cm) {$\bar 3$};
    \node[left] at (180:1cm) {$3$};
    \node[below] at (240:1cm) {$\bar 4$};
    \end{scope}

            \begin{scope}[shift={(-0.5,0.5)}]
    \draw[thick,blue] (0,0) circle (0.4cm);

        \draw[thick,red,opacity=1] (0:0.4cm) arc[start angle=0, end angle=180, radius=0.4cm];
        
        \filldraw (0:0.4cm) circle (2pt); 
        \filldraw (180:0.4cm) circle (2pt);

    \node[right] at (0:0.4cm) {$1$}; 
    \node[left] at (180:0.4cm) {$\bar 1$};

    \end{scope}

      \begin{scope}[shift={(0.5,-0.5)}]
    \draw[thick,blue] (0,0) circle (0.4cm);

        \draw[thick,red,opacity=1] (0:0.4cm) arc[start angle=0, end angle=180, radius=0.4cm];
        
        \filldraw (0:0.4cm) circle (2pt); 
        \filldraw (180:0.4cm) circle (2pt);

    \node[right] at (0:0.4cm) {$2$}; 
    \node[left] at (180:0.4cm) {$\bar 2$};

    \end{scope}

\end{tikzpicture}

    \caption{The perfect matchings corresponding to $\tau$ (in blue), and the preferred representative $\rho_\mathbf n$ for $\mathbf n = (1,1,2,3)$ (in red).}
    \label{fig:pref}
\end{figure}
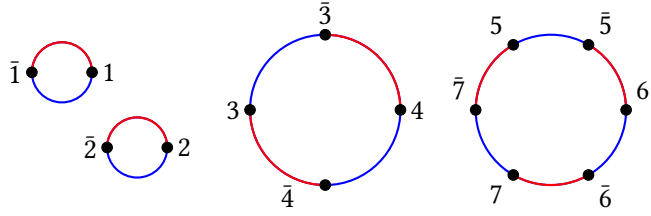


Following the ideas of~\cite{DDM,Karev-Do}, we produce the recursion keeping track of the cycle of $\mu$ that contains the largest element. We introduce the following splitting.
Let $\underline n = (n_1,\ldots,n_p),\underline m = (m_1,\ldots,m_q)$ be two compositions of $n$. The symbol $N_g^a\left( \begin{smallmatrix} n_1,\ldots,n_p \\  m_1,m_2,\ldots,m_q | r \end{smallmatrix}\right)$ for $r \in \{1,\ldots,q\}, a\in\{ 1,\ldots,n_p\}$ stands for $\frac 1{C^qJ^p\prod (2j)^{n_j}}$ multiplied by the weighted number of pairs $((\sigma_1,\ldots,\sigma_k),\preccurlyeq)$, satisfying

\begin{itemize}
\item $(\sigma_1,\ldots,\sigma_k)$ is a transitive monotone string of transpositions;
\item $k =  {2g - 2 + p + q};$
\item the type of $\varrho = \sigma_k\cdots \sigma_1. \rho_\mathbf n$ is a fixed partition $\mu\vdash n$; 
\item $\preccurlyeq$ is a linear order on the parts of $\mu$ turning it into the composition $(m_1,\ldots,m_q)$;
\item the element $n$ belongs to the cycle of $\varrho\tau$ of length $2m_r$ that has order $r$ with respect to $\preccurlyeq$; 
\item the last transposition $\sigma_k$ is either $(j~ n - n_p +a)$ or $(\bar j~ n - n_p + a)$ for  $a \in \{1,\ldots,n_p\}$ with $j < n - n_p + a.$
\end{itemize}
 The last condition arises due to the transitivity and monotonicity requirement. Notice also, that in case all the numbers $n-n_p+a,\ldots,n$ belong to the same cycle of $(\sigma_k\cdots \sigma_1.\rho_\mathbf n)\tau$ of length $2m_r.$

We set:
\[\mathrm{N}_g\left(
    \begin{smallmatrix}
    n_1,\ldots,n_p\\ m_1,\ldots, m_q
\end{smallmatrix}\right) = \sum_{a = 1}^{n_p} \sum_{r = 1}^{q} N^a_g\left(\begin{smallmatrix}
    n_1,\ldots,n_p\\ m_1,\ldots, m_q|r
\end{smallmatrix}\right)\]
to denote the weighted number of strings of transpositions with labeled cycles with no conditions on the position of the largest element and the last used transposition.

The initial conditions for these numbers correspond to the case $k = 0.$ Due to the transitivity requirement, both $p,q$ are equal 1. The genus of the corresponding string of transpositions is $0$. We also set $a = 1$ in this case. Summarizing, the initial conditions are:

\[  N^{0,a}_0\left(\begin{smallmatrix}
    n_1,\ldots,n_p\\ m_1,\ldots,m_q|r
\end{smallmatrix}\right)= \delta_{a,1} \delta_{p,1}\delta_{q,1}\delta_{m_1,n_1} \frac{1}{2CJn_1}. \]


The cut-and-join analysis for the number  $N_g^a\left( \begin{smallmatrix} n_1,\ldots,n_p \\  m_1,m_2,\ldots,m_q | r \end{smallmatrix}\right)$ for $r \in \{1,\ldots,q\}$, breaks into the following cases depending on the action of the permutations $\sigma_k$. The main idea is to establish a weighted bijective correspondence of the numbers before and after the application of $\sigma_k$ keeping track of the order of parts and the part containing the largest element.
\begin{itemize}
\item If $\sigma_k$ performs a cut, then the cycle that was cut due to the transitivity condition must have contained the element $n$ in it. The cut
produces two new cycles:
\begin{itemize}
\item the one that contains $n$; the corresponding part becomes the special one with respect to $\preccurlyeq$;
\item the one that does not contain $n$; its order with respect to $\preccurlyeq$ can be arbitrary;
\end{itemize} The cut is only possible if $n - (n - n_p + a) < m_q:$ otherwise the number of elements in the cycle is not enough to form a cycle of the required length. Moreover, if the lengths of the cycles obtained after the cut are specified, given that the second entry of $\sigma_k$ is fixed, due to Lemma~\ref{lm:cut-twist}, there is a unique option for the first entry of $\sigma_k$. As the cut takes place, the number should be multiplied by $C$, but as the number of part increases, we also need to divide by $C$. Collecting everything together, the cut term takes the following form:
\[
 \Theta(m_r  - n_p + a - 1)  \sum_{j \ne r}\sum_{l = 1}^a N_{g}^l \left( \begin{smallmatrix} n_1,\ldots,n_p \\ m_1,\ldots,\hat m_j,
 \ldots, m_r + m_j,\ldots,m_q| r'\end{smallmatrix}\right),
\]
where \[r' = \begin{cases} r-1,\quad\mbox{if } j < r,\\ r,\quad\mbox{otherwise,}\end{cases}\] and  $\Theta$ is the Heaviside step function, and $\hat m_j$ means that this part is omitted.
\item If $\sigma_m$ performs a twist
in case when the second entry of $\sigma_m$ is fixed, due to Lemma~\ref{lm:cut-twist}, there are $(m_q - n_p + a-1)$ possible options for the first its entry. This condition implies the inequality $(m_q - n_p + a -1) > 0$ similar to the cut case.  Twists preserve the order $\preccurlyeq$ of the parts of $\mu$. So, the term corresponding to the twist is:
\[
(m_r - n_p + a - 1 )T\Theta(m_r  - n_p + a - 1)\sum_{l = 1}^a N^l_{g-\frac 12}\left( \begin{smallmatrix} n_1,\ldots,n_p \\ m_1,\ldots,m_q| r\end{smallmatrix}\right).
\]
\item If $\sigma_k$ performs a redundant join, we assume that the cycles of the lengths $2\alpha$ and $2\beta$ with $\alpha + \beta = m_q$ were joined to form a cycle of length $2m_q$. Due to the transitivity, the cycle of length $2\alpha$ contains the element $n$, and its position with respect to $\preccurlyeq$ is $r$. For bijectivity, we set that the cycles of length  $2\beta$ had position $r+1$ according to  $\preccurlyeq$ before the join takes place.  There are $2\beta$ possible choices of the first components of $\sigma_k$. The fact that we have the join provides the weight $J,$ but there is also factor $C$ arising from the fact that the total number of cycles becomes one less. The overall contribution is: 
\[
2CJ\sum_{\alpha+ \beta = m_r} \sum_{l = 1}^a \beta N_{g-1}^l \left( \begin{smallmatrix} n_1,\ldots,n_p \\ m_1,\ldots,\alpha,\beta,\ldots,m_q| r\end{smallmatrix}\right),
\]
where the factor $2$ arises from the fact that we only record the halves of cycle lengths.
\item Finally, if $\sigma_k$ makes an essential join, the weighted number equals
\[
2CJ\sum_{\alpha + \beta = m_r} \sum_{g_1 + g_2 = g} \sum_{
\begin{smallmatrix} K_1 \sqcup K_2 = \{1,\ldots,p\},\ p\in K_2 \\ I_1 \cup I_2 = \{1,\ldots,q\},\ I_1\cup I_2 = \{r\}
\end{smallmatrix}
} \sum_{l = 1}^a \beta  N_{g_1}(\begin{smallmatrix}n_{K_1}\\ m_{I_1}({m_r \mapsto  \beta)}\end{smallmatrix})N^l_{g_2}(\begin{smallmatrix} n_{K_2}\\ m_{I_2}(m_r \mapsto \alpha)|r\end{smallmatrix}),
\]
Here, for $I_1 \subset \{1,\ldots q\}$ ($K_1\subset \{1,\ldots,p\}$, respectively) we mean a subcomposition of $(m_1,\ldots,m_q)$ ($(n_1,\ldots,n_p)$, respectively) formed by parts whose indices are contained in $I_1$ ($K_1$, respectively). The notation $m_{I_1}(m_r \mapsto \beta)$ mean the following. The index $r$ belongs to $I_1$. We take a subcomposition $m_{I_1}$ and replace its part $m_r$ by $\beta$. 
\end{itemize}

Summarizing, we obtain the following.

\begin{theorem}\label{th:cut-and-join-monotone}
The numbers $N^a_g\left(\begin{smallmatrix}
    n_1,\ldots,n_p\\ m_1,\ldots,m_q|r
\end{smallmatrix}\right)$ satisfy the  initial condition
\[ N^a_g\left(\begin{smallmatrix}
    n_1,\ldots,n_p\\ m_1,\ldots,m_q|r
\end{smallmatrix}\right) = \delta_{a,1} \delta_{g,0} \delta_{p,1}\delta_{q,1} \frac{1}{2CJn_1}, \]
and the recursion
\begin{gather}
N^a_g\left(\begin{smallmatrix}
    n_1,\ldots,n_p\\ m_1,\ldots,m_q|r
\end{smallmatrix}\right) = \Theta(m_r  - n_p + a - 1) \left(  \sum_{j \ne r}\sum_{l = 1}^a N_{g}^l \left( \begin{smallmatrix} n_1,\ldots,n_p \\ m_1,\ldots,\hat m_j,
 \ldots, m_r + m_j,\ldots,m_q| r'\end{smallmatrix}\right) + \right. \\ \left. (m_r - n_p + a - 1 )T\sum_{l = 1}^a N^l_{g-\frac 12}\left( \begin{smallmatrix} n_1,\ldots,n_p \\ m_1,\ldots,m_q| r\end{smallmatrix}\right) \right)
+\sum_{\alpha+ \beta = m_r} \sum_{l = 1}^a 2CJ \beta \Biggl( N_{g-1}^l \left( \begin{smallmatrix} n_1,\ldots,n_p \\ m_1,\ldots,\alpha,\beta,\ldots,m_q| r\end{smallmatrix}\right) + \Biggr. \\ \Biggl.   \sum_{g_1 + g_2 = g} \sum_{
\begin{smallmatrix} K_1 \sqcup K_2 = \{1,\ldots,p\},\ p\in K_2 \\ I_1 \cup I_2 = \{1,\ldots,q\},\ I_1\cup I_2 = \{r\}
\end{smallmatrix}
}   \mathrm{N}_{g_1}(\begin{smallmatrix}n_{K_1}\\ m_{I_1}({m_r \mapsto  \beta)}\end{smallmatrix})N^l_{g_2}(\begin{smallmatrix} n_{K_2}\\ m_{I_2}(m_r \mapsto \alpha)|r\end{smallmatrix})\Biggr).
\end{gather}
\end{theorem}

Notice, that the numbers $N_g\left( \begin{smallmatrix} n_1,\ldots,n_p \\  m_1,\ldots,m_q \end{smallmatrix}\right) = \sum_{a = 1}^{n_p} \sum_{r = 1}^{q} N^a_g\left(\begin{smallmatrix}
    n_1,\ldots,n_p\\ m_1,\ldots, m_q|r
\end{smallmatrix}\right)$ neither depend on the composition $(n_1,\ldots,n_p)$ giving rise to the partition $\nu,$ nor on the order $\preccurlyeq$ we impose on the parts $(m_1,\ldots,m_q)$ of the partition $\mu.$ The refined monotone Hurwitz numbers $H^\le_g\left(\begin{smallmatrix} \nu\\ \mu \end{smallmatrix}\right)$ are related to the numbers $N_g\left( \begin{smallmatrix} n_1,\ldots,n_p \\  m_1,m_2,\ldots,m_q \end{smallmatrix}\right)$ as follows:

\begin{equation}\label{eq:aut}
H^\le_g\left(\begin{smallmatrix} \nu\\ \mu \end{smallmatrix}\right) = \frac 1{|\mathrm{Aut}(\mu)\times\mathrm{Aut}(\nu)| } \mathrm{N}_g\left( \begin{smallmatrix} n_1,\ldots,n_p \\  m_1\ldots,m_q \end{smallmatrix}\right), \end{equation}
where $(n_1,\ldots,n_p)$ ($(m_1,\ldots,m_q),$ respectively) is an arbitrary composition, realizing the partition $\nu$ ($\mu,$ respectively).

Specialization of the obtained recursion to the values $C = J = T = 1$ clarly recovers the zonal case. 

Introduce the following \emph{refined inner product} on $I:$
\[
\left< \mathcal D_\lambda,  \mathcal D_\mu\right> = \delta_{\lambda,\mu}\prod_{j\ge 1} \frac 1{(2jCJ)^{m_j} m_j!}.
\]
Let us also modify the preliminary refined Jucys-Murphy elements by the following rule to obtain the \emph{refined Jucys-Murphy elements} 
\[ {\mathbb X}_k  = \mathbf{X}_k|_{\begin{smallmatrix} J \to 1 \\ C \to CJ \end{smallmatrix}}.
\]
The meaning of these substitutions is the following: now in the case of cut the corresponding weight becomes $1,$ and in the case of join the weight becomes $CJ.$ The weight of a twist is unchanged. 

\begin{lemma}\label{lm:innprod}
We have
\[
H_g^{\le,\bullet}\left( \begin{smallmatrix} \nu \\ \mu \end{smallmatrix}\right) 
= \left< \mathcal D_{\mu}, h_k ({\mathbb X}_2,\ldots,  {\mathbb X}_{|\nu|}) (\mathcal D_\nu)\right>,
\]
where $k$ is related to $g$ by $k = 2g - 2 +\ell(\mu) + \ell(\nu). $
\end{lemma}
\begin{proof}
We have defined the disconnected refined monotone Hurwitz numbers by
\[
H_g^{\le,\bullet}\left( \begin{smallmatrix} \nu \\ \mu \end{smallmatrix}\right)  = \frac 1{(2jC)^{m_j}m_j!} [\mathcal D_\mu].\mathrm h_k ({\mathbf X}_2,\ldots,  {\mathbf X}_{|\lambda|}) (J^{-l(\nu)}\mathcal D_\nu),
\]
and we seek to check that this expression coincides with
\[
\frac 1{(2jCJ)^{m_j}m_j!} [\mathcal D_\mu].\mathrm h_k ({\mathbb X}_2,\ldots,  {\mathbb X}_{|\lambda|}) (\mathcal D_\nu),
\]
where the total degree of $J$ coming from the denominator is ${-\ell(\mu)}.$ Every sequence of transpositions contributing to the count of $H_g^{\le,\bullet}\left( \begin{smallmatrix} \nu \\ \mu \end{smallmatrix}\right)$ has the total $J$-degree of $-\ell(\nu) + \#\mbox{joins}$ in its weight. However, Proposition~\ref{prop:C<->J} indicates that this number coincides with $-\ell(\mu) + \#\mbox{cuts}.$ The assertion follows.
\end{proof}

\begin{theorem}\label{th:sa} Let $\mathrm F\in \Lambda$ be a symmetric function. Then the operator
\[
\mathcal D_\lambda \mapsto \mathrm F({\mathbb X}_2,\ldots,  {\mathbb X}_{|\lambda|}) (\mathcal D_\lambda),
\]
is  self-adjoint  with respect to the inner product $\left<\cdot\right>$.
\end{theorem}

\begin{proof}
As any symmetric function, $\mathrm F$ is a linear combination of products of homogeneous symmetric functions, and a composition of commuting self-adjoint operators is self-adjoint, it is enough to establish the declared property for the homogeneous symmetric functions $\mathrm h_k.$ Notice, that for the homogeneous symmetric functions, the inner product
\[
\left< \mathcal D_\mu, \mathrm h_k({\mathbb X}_2,\ldots,  {\mathbb X}_{|\lambda|}) (\mathcal D_\lambda)\right>
\]
is the corresponding disconnected refined Hurwitz number by  Lemma~\ref{lm:innprod}, and Proposition~\ref{prop:symmetry} implies the self-adjointness.
\end{proof}

The second important feature of the theory is the possibility for the specialization for different values of the parameters. For any $\mathrm F\in \Lambda$ \emph{the refined structure coefficients}
\[
\mathrm f^\mu_{\mathrm F,\nu} =  {\left< \mathcal D_\mu, \mathrm F({\mathbb X}_2,\ldots,  {\mathbb X}_{|\mu|})(\mathcal D_\nu)\right>}
\]
are completely determined by the inner products involving the homogeneous symmetric function as for any $\mathrm F,\mathrm G\in \Lambda$ we have:

\begin{equation}
\mathrm f_{\mathrm F\mathrm G,\nu}^{\mu} = \sum_{\lambda} \frac{\mathrm f^{\mu}_{\mathrm G,\lambda} \mathrm f^{\lambda}_{\mathrm F,\nu}}{\left\langle \mathcal D_\lambda,\mathcal D_\lambda \right\rangle }
\end{equation}
(compare to Equation~\ref{eq:strcoef}). 
Both the structure coefficients and the inverse of the inner square of $\mathcal D_\mu$ are rational functions in $C,J,T$ with possible poles in $C = 0$ and $J = 0$, so it makes sense to  ask about specializations of the structure coefficients.

\begin{theorem}
    \label{thm:structcoeff}
    For $CJ = T = 1$ the structure coefficients $f^\mu_{F,\lambda}$ of the refined action coincide with the structure coefficients of the zonal action. For $CJ = \frac 12, T = 0$ the structure coefficients of the refined action coincide with the structure coefficient of the Schur action.
\end{theorem}
\begin{proof} It follows immediately from the fact that the inner product $\left<,\right>$ specializes to the $\left<,\right>_0$ and $\left<,\right>_1$ for the corresponding values of parameters, and the recursion of Theorem~\ref{th:cut-and-join-monotone} that specializes to the zonal, and Schur case correspondingly. 
\end{proof}
\begin{remark}
It is not surprising that the specialization $CJ = T = 1$ gives the zonal case, as the refined case takes its origin as its deformation. However, the fact, that the structure coefficients of the Schur case appear in this family, given that the Schur case is based on a different combinatorial model (\emph{no $\bar i$!)}, is satisfying.
\end{remark}
\begin{remark}
    \label{rem-eigenvec}
     Similarly to Schur and zonal actions, it is possible to relate the common eigenvectors of the operators $F({\mathbb X}_2,\ldots,  {\mathbb X}_{|\lambda|})$ to a set of symmetric functions orthogonal with respect to the inner product $\left<\cdot,\cdot\right>$. These function are a rescaled version of Jack functions.\footnote{We are grateful to A. Liashyk who pointed it out.}
\end{remark}
\begin{remark}
\label{rem-spec}
We also note here, that the specialization $CJ = \frac{1+b}2, T = b$ conjecturally provides the \emph{Jack} case extensively studied nowadays (see e.g.~\cite{Chapuy-Dolega,Bonzom-Chapuy-Dolega}). Namely, we can notice immediately, that under the indicated specialization, the refined Laplace-Beltrami operator (see Eq.~\ref{eq:Laplace-Beltrami})  specialize to the corresponding operator derived in~\cite{Chapuy-Dolega}.  However, a proof that the refined Hurwitz numbers specialize to double $b-$weighted Hurwitz numbers for all possible weights, should state the coincidence of not only the eigenfunctions, but also of the \emph{spectrum} of the action. Such a proof is a subject of a separate publication. This paper mentions explicit coincidence to the \emph{single b-monotone Hurwitz numbers} case, which refers to the trivial partition $\nu$ in $H_g\left(\begin{smallmatrix} \nu\\ \mu\end{smallmatrix}\right),$ and \emph{double b-Hurwitz numbers} case, which refers to the action of the symmetric function $\mathrm p_1^n$ (see subsequent sections).
\end{remark}




\section{Tropical refined simple Hurwitz numbers }\label{sec:simple}
In this section, we derive tropialisations of monotone and simple $b$--Hurwitz numbers. We note that in principle our constructions generalise to the full CJT--refinement, however due to the importance of the $b$--case and to keep notation light, starting from subsection~\ref{sec:tropb} we focus on the specialization $CJ=\frac{1+b}{2}$ and $T=b$.

\subsection{Recursion for double simple refined Hurwitz numbers}
\label{sec-rec}
In this section, we focus on the refined simple Hurwitz numbers, and their tropical interpretation, that prove to be fruitful for investigating structural properties in the classical Schur case. Let $p_1$ be the first Newton power sum.

\begin{definition}
    Let $k\in \mathbb N\cup \{0\}$, $\mu,\nu$ are partitions. The disconnected refined simple Hurwitz number  $h^{\circ}_k\left(\begin{smallmatrix}
        \nu \\ \mu
    \end{smallmatrix} \right)$ is the inner product
    \[
    h^{\circ}_k\left(\begin{smallmatrix}
        \nu \\ \mu
    \end{smallmatrix} \right) = \left<\mathcal D_\mu, {\mathrm p_1^k({\mathbb X}_2,\ldots,  {\mathbb X}_{n})(\mathcal D_\nu)}\right>.
    \]
\end{definition}

Notice, that by definition we have $h^{\circ}_k\left(\begin{smallmatrix}
        \nu \\ \mu
    \end{smallmatrix} \right) = h^{\circ}_k\left(\begin{smallmatrix}
        \mu \\ \nu
    \end{smallmatrix} \right).$ As usual, we can also define the corresponding connected numbers $h_{k}\left(\begin{smallmatrix}
        \nu \\ \mu
    \end{smallmatrix} \right)$, using the inclusion-exclusion formula. It is instructive for us to define the connected numbers in terms of the weighted count of strings of transpositions:

\begin{definition}
For $k\in \mathbb N\cup\{0\}$, and partitions $\mu,\nu,$ the connected refined simple Hurwitz number $h^k\left(\begin{smallmatrix}
        \nu \\ \mu
    \end{smallmatrix} \right)$ equals the weighted number of transitive twisted factorisations (see Definition~\ref{def:twfac}) of type $(k,\mu,\nu)$ divided by $2^{|\mu|}|\mu|! C^{\ell(\mu)}J^{\ell(\nu)}$.
\end{definition}

Using this definition we can conduct the cut-and-join analysis by the study of possible behavior under the action of the last transposition $\sigma_k.$ Similarly to the monotone case, it is convenient to introduce an order $\preccurlyeq$ on the parts of $\mu$ producing a composition $(m_1,\ldots,m_q)$ instead. The numbers that we compute do not depend on this order. As for the monotone numbers the base case is
\[
h_0\left(\begin{smallmatrix} (n^1) \\ m\end{smallmatrix}\right) = \delta_{n,m}\frac 1{2CJn}
\]
Now, consider the cases:
\begin{itemize}
\item Let the transposition $\sigma_k$ performs a cut. Notice, that if  a transposition $(i~j)$ for $i,j\in \{1,\bar 1,\ldots,n,\bar n\}$ cuts a cycle of length $2(m_r + m_s)$ into cycles of lengths $2m_r$ and $2m_s$ respectively then the transposition $(\bar i~\bar j)$ also cuts it into cycles of the same lengths. However, among $(i~j)$ and $(\bar i~\bar j)$, there is exactly one that can be a part of a twisted factorisation. As the total number of transpositions that make the required cut is
\[
\begin{cases} 2(m_r + m_s),\,\mbox{if $m_r \ne m_s$}\\ 2m_r,\,\mbox{otherwise,}\end{cases}
\]
and any cycle can be cut, taking into account the factor $C^{-\ell(\mu)}$ from the definition of the refined Hurwitz numbers, the cut term is
\[
\frac {m_r + m_s}{2}\sum_{r\ne s}h^{k-1}\left(\begin{smallmatrix}
    \nu \\ m_1,\ldots,\hat m_r,\ldots, \hat m_s,\ldots,m_q,m_r + m_s 
\end{smallmatrix} \right).
\]
\item If $\sigma_k$ performs a twist of a cycle of length $2m_r$ there are exactly $1 + \ldots +m_r-1 = \frac{m_r(m_r-1)}2$ ways to do it using transpositions that can be a part of transitive factorisation. As any cycle can be twisted, the resulting term is 
\[
\frac T2\sum_{r = 1}^q m_r(m_r-1) h_{k-1}\left(\begin{smallmatrix}\nu \\m_1,\ldots,m_q \end{smallmatrix} \right).
\]
\item In the case of a redundant join, cycles of lengths $2\alpha$ and $2\beta$ merge to form a cycle of length $2m_r.$ There are $4\alpha\beta$ possible ways to make this join, however among the transpositions $(i~j),(\bar i~j),(i~\bar j),$ and $(\bar i~\bar j)$ for $i,j = 1,\ldots n,$ only two can be part of a transitive factorisation. Taking in account the factor $C^{-\ell(\mu)},$ the redundant join terms is
\[
CJ\sum_{r = 1}^q \sum_{\alpha + \beta = m_r} \alpha\beta h_{k-1}\left(\begin{smallmatrix}
    \nu \\ m_1,\ldots,\hat  m_r,\ldots ,m_q,\alpha,\beta
\end{smallmatrix}\right)
\]
\item For the essential join we need to choose, which of the permutations $\sigma_1,\ldots,\sigma_{k-1}$ effect two groups of cycles that join, that results in binomial coefficient ${k - 1 \choose k_1}.$ Similarly to the previous case, there are $2\alpha\beta$ possible choices of $\sigma_k.$ Summarizing, the essential join term takes the form

\[
CJ\sum_{\alpha + \beta = m_r} \sum_{k_1 + k_2 = k-1} {k-1 \choose k_1}\sum_{
\begin{smallmatrix} K_1 \sqcup K_2 = \{1,\ldots,p\},\\ I_1 \cup I_2 = \{1,\ldots,q\},\ I_1\cup I_2 = \{r\}
\end{smallmatrix}
} \alpha \beta  h_{k_1}(\begin{smallmatrix}n_{K_1}\\ m_{I_1}({m_r \mapsto  \beta)}\end{smallmatrix})h_{k_2}(\begin{smallmatrix} n_{K_2}\\ m_{I_2}(m_r \mapsto \alpha)\end{smallmatrix}),
\]
where we use the pieces of notation similar to those used for the cut-and-join analysis in monotone case.
\end{itemize}

\subsection{Tropical interpretation of $b$-Hurwitz numbers}\label{sec:tropb}
We now focus on the case of $b$--Hurwitz numbers, i.e. we consider the invariants 

\begin{equation}
    h_k^{(b)}\left(\begin{smallmatrix} \nu \\ \mu\end{smallmatrix}\right) = h_k\left(\begin{smallmatrix} \nu \\ \mu\end{smallmatrix}\right)|_{\substack{CJ\to\frac{1+b}{2}\\T\to b}}.
\end{equation}

Our cut--and--join analysis in the previous subsection produces the following recursion that is a rephrasing of \cite[Theorem 6.5]{Chapuy-Dolega}.

\begin{theorem}\label{thm:recursionbHN}
		For partitions $\mu,\nu$,
		the $b$-Hurwitz numbers satisfy the following recursion:
		\begin{gather} \label {Eq:MonCJ}
			h_k^{(b)}\left(
            \begin{smallmatrix}
                \nu \\ \mu
            \end{smallmatrix}\right) = \frac{1}{2}\cdot \Bigg(
			\sum_{\substack{ i\neq j}} (\mu_i+\mu_j) h_{k-1}^{(b)}\left(
            \begin{smallmatrix} \nu\\ \mu(\hat{i},\hat{j}),\mu_i+\mu_j\end{smallmatrix}\right)\\
			+(1+b)\sum_{i=1}^{\ell(\mu)}\sum_{\alpha+\beta = \mu_i}\alpha\beta \cdot 
			h_{k-1}^{(b)}\left( \begin{smallmatrix} \nu \\
            \mu(\hat{i}),\alpha,\beta\end{smallmatrix}\right)\\
            +(1+b) \sum_{k_1+k_2=k-1} \sum_{I_1\cup I_2 =\{1,\ldots,\ell(\mu)\}\setminus\{i\}} \sum_{ J_1\cup J_2 =\{1,\ldots,\ell(\nu)\}}
			\binom{k-1}{k_1,k_2} \alpha\beta\cdot  h_{k_1}^{(b)}\left(
            \begin{smallmatrix} \nu_{J_1} \\ \mu_{I_1},\alpha\end{smallmatrix}\right)h_{k_2}^{(b)}\left(
            \begin{smallmatrix} \nu_{J_2} \\ \mu_{I_2},\beta\end{smallmatrix}\right) \\
			+b\sum_{i=1}^{\ell(\mu)}\mu_i(\mu_i-1)h_{k-1}^{(b)}\left(
            \begin{smallmatrix} \nu \\ \mu\end{smallmatrix}\right)\Bigg),
		\end{gather}
		where the $\hat{}$ sign means removing the entry, and the initial condition $h_0^{(b)}\left(\begin{smallmatrix}(m^1) \\(m^1) \end{smallmatrix}\right) = \frac{1}{(1+b)\cdot m}$, 
		
	\end{theorem}

We can interpret this recursion in terms of a graph formalism, which we can view as a count of tropical covers. Recall the notion of a twisted tropical cover from Section \ref{sec:tropprelim}.

\begin{definition}[Branch point multiplicities]\label{def:brpointmult}
    Let $\pi:\Gamma\rightarrow \mathbb{R}$ be a twisted tropical cover and $\overline{\pi}:\overline{\Gamma}\rightarrow \mathbb{R}$ its quotient. 
Let $p_j$ be a branch point.

Assume the preimage of $p_j$ is a $4$-valent vertex $V$. Denote the weight of its adjacent edges by $\omega_V$. We define the \emph{branch point multiplicity} to be $m_j:=b\cdot (\omega_V-1)$.

Assume the preimage of $p_j$ consists of a pair of vertices which both join two edges (considered from right to left). Assume that by cutting the image vertex $V$ in the quotient, $\overline{\Gamma}$ stays connected.
Pick one of the join edges $e$ of $\overline{\Gamma}$ (i.e.\ one of the two edges which are mapped to the right of the image of $V$ in $\mathbb{R}$). Remove $e$, and $g-1$ further edges of $\overline{\Gamma}$ (where $g$ is the genus of $\overline{\Gamma}$) such that we are left with a spanning tree. In $\Gamma$, remove the preimage edges accordingly.
Now re-insert $e$ into $\overline{\Gamma}$ and the preimage edges back into $\Gamma$. If the result is connected (i.e.\ the preimages of $e$ connect top with bottom and bottom with top, as described above), we define 
the \emph{branch point multiplicity} to be $m_j:=b$.
Else, i.e.\ if the preimage edges connect top with top and bottom with bottom, we define $m_j:=1$.

For all other branch points, $m_j$ is set to be $1$.

\end{definition}

\begin{definition}[Tropical $b$ double Hurwitz number]
\label{def:twisttrophn}
We define the tropical $b$ double Hurwitz 
number $h^{\trop}_k\left( \begin{smallmatrix}\nu \\ \mu\end{smallmatrix}\right)$ for $k>0$ to be  the weighted enumeration of connected twisted tropical covers with labeled ends of type $(k,\mu,\nu)$, with branch points at fixed points $p_i$, such that each cover $\pi\colon \Gamma\to\mathbb{R}$ is counted with multiplicity
\begin{equation}
    \frac{2}{b+1}\cdot \frac{1}{|\Aut(\pi)|}\cdot\prod_{j=1}^r m_j\prod_e\omega(e),
\end{equation}
where the first product goes over all branch points $p_j$ and $m_j$ denotes the branch point multiplicity from Definition \ref{def:brpointmult}, and the second product is taken over all internal edges of the quotient cover and $\omega(e)$ denotes their weights.
For $r=0$, the tropical twisted cover consisting of two edges of weight $m$ which are switched by the involution is counted with multiplicity $\frac{1}{(1+b)m}$, so $h^{\trop}_k\left(\begin{smallmatrix}(m^1) \\ (m^1)\end{smallmatrix}\right)= \frac{1}{(1+b)m}$.
\end{definition}

\begin{remark}\label{rem:conventiondifferences}
For $b=1$, this definition specializes to the twisted tropical double Hurwitz number from \cite{HMtwisted1} up to three differences in convention which we point out here:
\begin{enumerate}
    \item In \cite{HMtwisted1}, there is an additional global factor of $2^k$, where $k$ is the number of branch points.
    \item In \cite{HMtwisted1}, we do not label the ends. This produces an additional factor of $\frac{1}{|\Aut(\mu)||\Aut(\nu)|}$.
    \item In \cite{HMtwisted1}, we require the source of the tropical cover to be connected rather than its quotient graph. That is, here we allow more twisted covers, namely those which consist of two copies of the quotient graph which are switched by the involution. Here, we consider those to be connected via the action of the involution.
\end{enumerate}
\end{remark}

The following lemma is a direct consequence of Proposition \ref{prop-automquotient}:

\begin{lemma}
\label{lem-quotient}
  The  tropical $b$ double Hurwitz 
number  $h^{\trop}_k\left( \begin{smallmatrix}\nu \\ \mu\end{smallmatrix}\right)$ for $k>0$ can equivalently be determined by enumerating quotient covers $\overline{\pi}$, which we count with multiplicity
\begin{equation}
     (1+b)^{g-1}\cdot \frac{1}{|\Aut(\overline{\pi})|}\cdot\prod_{V \mbox{ $2$-valent}} \frac{b}{2}\cdot (\omega_V-1)   \prod_e\omega(e),
\end{equation}
where $g$ is the genus of the quotient graph.

For $k=0$, the quotient cover consisting of a single edge of weight $m$ is counted with multiplicity $\frac{1}{(1+b)m}$, so $h^{\trop}_0\left(\begin{smallmatrix}(m^1) \\ (m^1)\end{smallmatrix}\right)= \frac{1}{(1+b)m}$ can also be determined in terms of quotient covers.
\end{lemma}

For $b=0$, this count specializes to the standard count of tropical double Hurwitz numbers: quotient covers with $2$-valent vertices do not contribute, and the quotient covers which do not have any $2$-valent vertices contribute with multiplicity $\frac{1}{|\Aut(\overline{\pi})|}\cdot  \prod_e\omega(e) $.

\begin{example}
\label{ex-tropbcomp}
    We compute the tropical $b$ double Hurwitz 
number  $h^{\trop}_2\left( \begin{smallmatrix}(m^1) \\(m^1)\end{smallmatrix}\right)$, where $m\in\mathbb{N}_{>0}$.

Figure \ref{fig:tropbHN} depicts in the top row the twisted tropical covers which contribute to the count, and in the bottom row the quotient tropical covers, both with their multiplicities.

The first two pictures in the top row appear for every choice of $i$. Note that if $i=\frac{m}{2}$, both pictures and their quotient cover have to be counted with another factor of $\frac{1}{2}$ to take the extra automorphism into account which arises from exchanging the internal edges of the same weight. 
Since we have to consider all $i$, it is most natural to overcount by summing over all $i=1,\ldots,m-1$ and multiply by a global factor of $\frac{1}{2}$ which will also take care of the extra automorphisms of the middle summand if $m$ is even.

The first picture has another automorphism, given by the twisting, and the two internal edges of its quotient cover have weight $i$ and $b-i$. The branch point multiplicity is $1$. For the second picture, the multiplicity is almost the same, except for an additional factor of $b$ which arises as branch point multiplicity for the left branch point.
The third picture has $3$ independent automorphism, yielding a factor of $\frac{1}{8}$. The branch point multiplicity of both branch points is $b\cdot (m-1)$. The quotient cover has one internal edge of weight $m$.

Altogether, we obtain
$$ h^{\trop}_2\left( \begin{smallmatrix} (m^1) \\(m^1)\end{smallmatrix}\right) = \frac{1}{2}\cdot \sum_{i=1}^{m-1} i\cdot (m-i) + \frac{1}{4\cdot (b+1)}b^2\cdot (m-1)^2\cdot m.$$

Alternatively, we can count quotient covers directly. The two right pictures of the top row yield the same quotient cover, which contributes $i\cdot (m-i)$. The left picture has a quotient cover of the same multiplicity. We obtain the same sum as before.

Observe that if we insert $b=1$, we obtain the purely real double Hurwitz number 
$$ h^{\trop}_2\left(\begin{smallmatrix} (m^1) \\ (m^1)\end{smallmatrix}\right)|_{b=1} = \frac{1}{2}\cdot \sum_{i=1}^{m-1} i\cdot (m-i) + \frac{1}{8} (m-1)^2\cdot m.$$
If we insert $b=0$, we obtain the complex double Hurwitz number
$$h^{\trop}_2\left( \begin{smallmatrix} (m^1) \\ (m^1)\end{smallmatrix}\right)|_{b=0} = \frac{1}{2}\cdot \sum_{i=1}^{m-1} i\cdot (m-i). $$
\begin{figure}
    \centering

\tikzset{every picture/.style={line width=0.75pt}} 

\begin{tikzpicture}[x=0.75pt,y=0.75pt,yscale=-1,xscale=1]

\draw    (99.42,39.37) -- (135.72,39.37) ;
\draw    (135.72,39.37) .. controls (144.97,31.02) and (172.2,30.23) .. (181.09,39.37) ;
\draw    (135.72,39.37) .. controls (144.97,50.6) and (171.84,48.65) .. (181.09,39.37) ;
\draw    (181.09,39.37) -- (217.39,39.37) ;
\draw    (99.42,65.53) -- (135.72,65.53) ;
\draw    (135.72,65.53) .. controls (144.97,57.17) and (172.2,56.39) .. (181.09,65.53) ;
\draw    (135.72,65.53) .. controls (144.97,76.76) and (171.84,74.8) .. (181.09,65.53) ;
\draw    (181.09,65.53) -- (217.39,65.53) ;
\draw    (186.9,181.63) -- (223.2,181.63) ;
\draw    (223.2,181.63) .. controls (232.45,173.27) and (259.68,172.48) .. (268.57,181.63) ;
\draw    (223.2,181.63) .. controls (232.45,192.86) and (259.32,190.9) .. (268.57,181.63) ;
\draw    (268.57,181.63) -- (304.87,181.63) ;
\draw    (262.76,39.67) -- (299.06,39.67) ;
\draw    (299.06,39.67) .. controls (308.32,31.31) and (335.54,30.53) .. (344.44,39.67) ;
\draw    (299.06,39.67) .. controls (308.32,50.9) and (339.54,58.83) .. (344.44,65.83) ;
\draw    (344.44,39.67) -- (380.73,39.67) ;
\draw    (262.76,65.83) -- (299.06,65.83) ;
\draw    (299.06,65.83) .. controls (308.32,57.47) and (346.07,46.29) .. (344.44,39.37) ;
\draw    (299.06,65.83) .. controls (308.32,77.06) and (335.18,75.1) .. (344.44,65.83) ;
\draw    (344.44,65.83) -- (380.73,65.83) ;
\draw    (407.96,39.37) .. controls (423.39,41.99) and (437.91,43.55) .. (444.26,49.17) ;
\draw    (407.96,58.96) .. controls (428.83,56.87) and (441.9,51.39) .. (444.26,49.17) ;
\draw    (489.63,49.17) .. controls (496.62,41.59) and (517.13,38.85) .. (525.93,39.37) ;
\draw    (444.26,49.17) .. controls (453.51,40.81) and (480.74,40.03) .. (489.63,49.17) ;
\draw    (444.26,49.17) .. controls (453.51,60.4) and (480.37,58.44) .. (489.63,49.17) ;
\draw    (489.63,49.17) .. controls (496.44,57.66) and (516.67,58.51) .. (525.48,59.03) ;
\draw    (407.96,181.72) -- (525.93,181.72) ;
\draw  [fill={rgb, 255:red, 0; green, 0; blue, 0 }  ,fill opacity=1 ] (445.89,181.59) .. controls (445.89,180.83) and (445.32,180.22) .. (444.62,180.22) .. controls (443.92,180.22) and (443.35,180.83) .. (443.35,181.59) .. controls (443.35,182.35) and (443.92,182.96) .. (444.62,182.96) .. controls (445.32,182.96) and (445.89,182.35) .. (445.89,181.59) -- cycle ;
\draw  [fill={rgb, 255:red, 0; green, 0; blue, 0 }  ,fill opacity=1 ] (490.9,181.59) .. controls (490.9,180.83) and (490.33,180.22) .. (489.63,180.22) .. controls (488.93,180.22) and (488.36,180.83) .. (488.36,181.59) .. controls (488.36,182.35) and (488.93,182.96) .. (489.63,182.96) .. controls (490.33,182.96) and (490.9,182.35) .. (490.9,181.59) -- cycle ;

\draw (116.84,29.51) node  [font=\footnotesize] [align=left] {\begin{minipage}[lt]{12.34pt}\setlength\topsep{0pt}
$\displaystyle m$
\end{minipage}};
\draw (115.39,57.2) node  [font=\footnotesize] [align=left] {\begin{minipage}[lt]{12.34pt}\setlength\topsep{0pt}
$\displaystyle m$
\end{minipage}};
\draw (198.15,31.34) node  [font=\footnotesize] [align=left] {\begin{minipage}[lt]{12.34pt}\setlength\topsep{0pt}
$\displaystyle m$
\end{minipage}};
\draw (200.45,56.81) node  [font=\footnotesize] [align=left] {\begin{minipage}[lt]{12.34pt}\setlength\topsep{0pt}
$\displaystyle m$
\end{minipage}};
\draw (203.96,173.69) node  [font=\footnotesize] [align=left] {\begin{minipage}[lt]{12.34pt}\setlength\topsep{0pt}
$\displaystyle m$
\end{minipage}};
\draw (286.36,171.86) node  [font=\footnotesize] [align=left] {\begin{minipage}[lt]{12.34pt}\setlength\topsep{0pt}
$\displaystyle m$
\end{minipage}};
\draw (278.74,30.56) node  [font=\footnotesize] [align=left] {\begin{minipage}[lt]{12.34pt}\setlength\topsep{0pt}
$\displaystyle m$
\end{minipage}};
\draw (277.65,56.81) node  [font=\footnotesize] [align=left] {\begin{minipage}[lt]{12.34pt}\setlength\topsep{0pt}
$\displaystyle m$
\end{minipage}};
\draw (356.78,30.95) node  [font=\footnotesize] [align=left] {\begin{minipage}[lt]{12.34pt}\setlength\topsep{0pt}
$\displaystyle m$
\end{minipage}};
\draw (362.95,57.98) node  [font=\footnotesize] [align=left] {\begin{minipage}[lt]{12.34pt}\setlength\topsep{0pt}
$\displaystyle m$
\end{minipage}};
\draw (422.54,35) node  [font=\footnotesize] [align=left] {\begin{minipage}[lt]{12.34pt}\setlength\topsep{0pt}
$\displaystyle m$
\end{minipage}};
\draw (420.48,65.43) node  [font=\footnotesize] [align=left] {\begin{minipage}[lt]{12.34pt}\setlength\topsep{0pt}
$\displaystyle m$
\end{minipage}};
\draw (467.07,37.09) node  [font=\footnotesize] [align=left] {\begin{minipage}[lt]{12.34pt}\setlength\topsep{0pt}
$\displaystyle m$
\end{minipage}};
\draw (465.49,64.25) node  [font=\footnotesize] [align=left] {\begin{minipage}[lt]{12.34pt}\setlength\topsep{0pt}
$\displaystyle m$
\end{minipage}};
\draw (516.79,36.44) node  [font=\footnotesize] [align=left] {\begin{minipage}[lt]{12.34pt}\setlength\topsep{0pt}
$\displaystyle m$
\end{minipage}};
\draw (516.43,67.12) node  [font=\footnotesize] [align=left] {\begin{minipage}[lt]{12.34pt}\setlength\topsep{0pt}
$\displaystyle m$
\end{minipage}};
\draw (514.62,173.43) node  [font=\footnotesize] [align=left] {\begin{minipage}[lt]{12.34pt}\setlength\topsep{0pt}
$\displaystyle m$
\end{minipage}};
\draw (166.06,21.94) node  [font=\footnotesize] [align=left] {\begin{minipage}[lt]{24.15pt}\setlength\topsep{0pt}
$\displaystyle m-i$
\end{minipage}};
\draw (163.94,81.36) node  [font=\footnotesize] [align=left] {\begin{minipage}[lt]{28.67pt}\setlength\topsep{0pt}
$\displaystyle m-i$
\end{minipage}};
\draw (250.45,197.98) node  [font=\footnotesize] [align=left] {\begin{minipage}[lt]{26.21pt}\setlength\topsep{0pt}
$\displaystyle m-i$
\end{minipage}};
\draw (324.86,80.45) node  [font=\footnotesize] [align=left] {\begin{minipage}[lt]{23.74pt}\setlength\topsep{0pt}
$\displaystyle m-i$
\end{minipage}};
\draw (326.07,28.21) node  [font=\footnotesize] [align=left] {\begin{minipage}[lt]{27.03pt}\setlength\topsep{0pt}
$\displaystyle m-i$
\end{minipage}};
\draw (164.77,41.1) node  [font=\footnotesize] [align=left] {\begin{minipage}[lt]{12.34pt}\setlength\topsep{0pt}
$\displaystyle i$
\end{minipage}};
\draw (164.15,66.6) node  [font=\footnotesize] [align=left] {\begin{minipage}[lt]{12.34pt}\setlength\topsep{0pt}
$\displaystyle i$
\end{minipage}};
\draw (251.45,167.72) node  [font=\footnotesize] [align=left] {\begin{minipage}[lt]{12.34pt}\setlength\topsep{0pt}
$\displaystyle i$
\end{minipage}};
\draw (323.45,43) node  [font=\footnotesize] [align=left] {\begin{minipage}[lt]{12.34pt}\setlength\topsep{0pt}
$\displaystyle i$
\end{minipage}};
\draw (322.75,64.6) node  [font=\footnotesize] [align=left] {\begin{minipage}[lt]{11.31pt}\setlength\topsep{0pt}
$\displaystyle i$
\end{minipage}};
\draw (393.79,97.1) node [anchor=north west][inner sep=0.75pt]  [font=\small] [align=left] {$\displaystyle \frac{2}{b+1} \cdot \frac{1}{8} \cdot b^{2} \cdot ( m-1)^{2} \cdot m$};
\draw (392.58,205.5) node [anchor=north west][inner sep=0.75pt]  [font=\small] [align=left] {$\displaystyle \frac{2}{b+1} \cdot \frac{1}{8} \cdot b^{2} \cdot ( m-1)^{2} \cdot m$};
\draw (89.85,104.93) node [anchor=north west][inner sep=0.75pt]  [font=\small] [align=left] {$\displaystyle \frac{2}{b+1} \cdot \frac{1}{2} \cdot i\cdot ( m-i)$};
\draw (239.67,106.88) node [anchor=north west][inner sep=0.75pt]  [font=\small] [align=left] {$\displaystyle \frac{2}{b+1} \cdot \frac{1}{2} \cdot b\cdot i\cdot ( m-i)$};
\draw (206.91,230.51) node [anchor=north west][inner sep=0.75pt]  [font=\small] [align=left] {$\displaystyle i\cdot ( m-i)$};

\end{tikzpicture}

    \caption{The tropical count of the $b$ double Hurwitz number $h^{\trop}_2\left(\begin{smallmatrix}(m^1) \\ (m^1)\end{smallmatrix}\right)$. The top row shows the twisted covers, the bottom row the quotient covers.}
    \label{fig:tropbHN}
\end{figure}
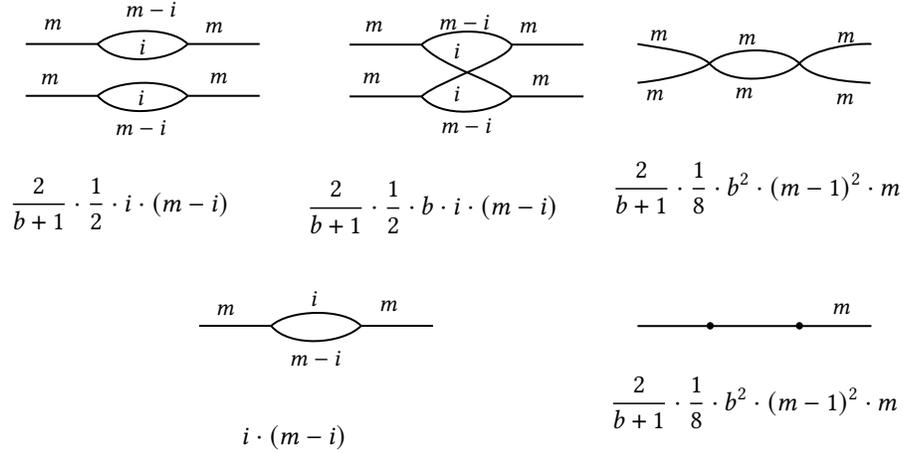
\end{example}

\begin{theorem}\label{thm-corresdoubleHNb}
  The  tropical $b$ double Hurwitz 
number  $h^{\trop}_k\left(\begin{smallmatrix}\nu \\ \mu\end{smallmatrix}\right)$ equals its $b-$counterpart $h^{(b)}_k\left(\begin{smallmatrix}\nu \\ \mu\end{smallmatrix}\right)$. 
\end{theorem}
\begin{proof}
We prove the statement using induction on $k$. The case $k=0$ follows from the definition.
Assume the statement is true for $k-1$.
Consider all tropical quotient covers contributing to the count of $h^{\trop}_k\left(\begin{smallmatrix}\nu \\ \mu\end{smallmatrix}\right)$ and cut off their leftmost vertex (thus reducing the number of branch points by one). We can express $h^{\trop}_k\left(\begin{smallmatrix}\nu \\ \mu\end{smallmatrix}\right)$ in terms of all possibilities how these cut off pictures can look like and what they contribute. We claim that this is exactly equivalent to the recursion for $h^{(b)}_k\left(\begin{smallmatrix}\nu \\ \mu\end{smallmatrix}\right)$ from 
Theorem \ref{thm:recursionbHN}. The statement then follows by induction.

The leftmost vertex can be a vertex which cuts an edge of weight $\mu_i+\mu_j$ into two edges of weight $\mu_i$ and $\mu_j$. If the cut edge is a right end (in case $k=1$), the factor of $\mu_i+\mu_j$ cancels with the initial value for $k=0$. Else, we have cut an internal edge and produced an end from it. Before cutting, the edge yielded a factor to the multiplicity of the quotient graph, namely its weight $\mu_i+\mu_j$. In any case, we have to multiply with $\mu_i+\mu_j$. We also have to take a factor of $\frac{1}{2}$ since we make two choices $(i,j)$ and $(j,i)$ in the sum for the two edges we cut.
This first case thus contributes the first summand of the recursion.

The leftmost vertex can be a vertex which joins two edges, such that the cut off graph is still connected. That means that cutting off the vertex reduces the genus of the quotient graph by one. We thus have to multiply by a factor of $(1+b)$. Furthermore, we cut two internal edges, so we have to multiply with their weights.
The two cut edges are ends of the cut quotient cover, and as such they come with labels. Since these labels are dropped after we join them, we have to multiply with a factor of $\frac{1}{2}$. 
    This second case thus contributes the second summand of the recursion.

    The leftmost vertex can be a vertex which joins two edges, such that the cut off graph has two connected components. One of the connected components must have $k_1$ vertices, and the other $k_2$, such that $k_1+k_2=k-1$. The ends must split into some that belong to one part, and others that belong to the other. 
    Given two quotient covers that contribute to the two counts corresponding to these components, we have to align their vertices before we merge them together at the leftmost vertex, this produces a factor of $\binom{k}{k_1,k_2}$. 
    Also, both components contribute a factor of $(1+b)^{g_1-1}$ resp.\ $(1+b)^{g_2-1}$, where $g_i$ denotes the genus of each component. The genus of the graph after joining equals $g_1+g_2$, so we want a contribution of $(1+b)^{g_1+g_2-1}$. For that reason, we have to multiply with $(1+b)$.
    Furthermore, we we cut two internal edges, so we have to multiply with their weights. Again, we drop labels of the two ends we join, so we have to multiply with a factor of $\frac{1}{2}$. 
    This third case thus contributes the third summand of the recursion.

    Finally, the leftmost vertex can be a $2$-valent vertex. We multiply by the branch point multiplicity which equals $\frac{b}{2}\cdot (\mu_i-1)$ and by the weight of the edge.
    This last case thus contributes the last summand of the recursion.

At last, we have to convince ourselves that we counted each quotient cover with the right multiplicity in this way. From the arguments before, this is already clear for all factors of $(1+b)$, the branch point multiplicities for the $2$-valent vertices and the edge weights of the internal edges. We still have to discuss the factor of $\frac{1}{|\Aut(\overline{\pi})|}$. Since the quotient cover is has only $2$- and $3$-valent vertices and labeled ends, the only automorphisms come from wieners, i.e.\ cycles formed by two edges of the same weight (see Lemma 4.2 in \cite{CJM10}). The factor $\frac{1}{|\Aut(\overline{\pi})|}$ thus equals $\big(\frac{1}{2}\big)^w$, where $w$ is the number of wieners. These factors occur in our recursion however, since we multiplied by $\frac{1}{2}$ when we join two edges. When the two edges have different weight, the factor just accounts for the two possibilities to choose weights for the two edges. When the two edges have the same weight, the factor yields the requested contribution to $\frac{1}{|\Aut(\overline{\pi})|}$.

We conclude that using the recursion from Theorem \ref{thm:recursionbHN}, we indeed count all possibilities for quotient covers (in terms of quotient covers with less branch points by adding one more branch point) with the requested multiplicity, thus obtaining $h^{\trop}_k\left(\begin{smallmatrix}\nu \\ \mu\end{smallmatrix}\right)$. Since $h^{\trop}_k\left(\begin{smallmatrix}\nu \\ \mu\end{smallmatrix}\right)$ satisfy the same recursion and have the same initial values by definition, the equality $h^{\trop}_k\left(\begin{smallmatrix}\nu \\ \mu\end{smallmatrix}\right) = h^{(b)}_k\left(\begin{smallmatrix}\nu \\ \mu\end{smallmatrix}\right)$ follows. 
 \end{proof}
\subsection{Piecewise polynomiality of double $b$-Hurwitz numbers}
\label{subsec-poly}
In this subsection, for the $b$-Hurwitz number  $h^{(b)}_k\left(\begin{smallmatrix}\nu \\ \mu\end{smallmatrix}\right)$ we use the genus $g\in \frac 12\mathbb N \cup \{0\} $ instead of the number of elements in the factorisation $k$. By the abuse of notation, $h^{(b)}_g\left(\begin{smallmatrix}\nu \\ \mu\end{smallmatrix}\right)$ denotes the weighted number of factorisations of the corresponding genus. Recall, that the genus and the number of factorisations are related as $k = 2g - 2 + \ell(\mu) + \ell(\nu)$. The main result of this subsection is the following theorem. 

\begin{theorem}
    \label{thm-piecepoly}
    Let $g$ be a non-negative integer and $n$ a positive integer. Then $(1+b)h_g^{(b)}$ is a polynomial in $b$ whose coefficients are piecewise polynomials in the entries of $\mu$ and $\nu$ of degree $2g-1+\ell(\mu)+\ell(\nu)$ with respect to the resonance arrangement.
\end{theorem}

As mentioned before for $b=0$, this result was proved in \cref{thm-GJV}. Moreover, for $b=1$ a proof was given via tropical geometry methods by the second and fourth author in \cite{HMtwisted1}. Our proof of \cref{thm-piecepoly} follows similar ideas as the one in \textit{loc.cit.}.

We recall from \cite{HMtwisted1} the definition of twisted monodromy graphs, which can be viewed as combinatorial types of twisted tropical covers as discussed in the previous section and Section \ref{sec:tropprelim}.

\begin{definition}
\label{def-mongraph}
For fixed $g$ and two partitions $\mu$ and $\nu$ of $n$, a graph $\Gamma$ is a \emph{twisted monodromy graph of type $(g, \mu,\nu)$}  if:
\begin{enumerate}
\item $\Gamma$ is a connected, genus $g$, directed graph. 
\item $\Gamma$ has $3$-valent and $4$-valent vertices.
\item $\Gamma$ has an involution $\iota$ whose fixed points are the $4$-valent vertices.
\item $\Gamma$ has $2\ell(\mu)$ inward ends of weights $\mu_i$. Ends which are mapped to each other under the involution have the same weight.
\item $\Gamma$ has $2\ell(\nu)$ outward ends of weights $\nu_i$. Ends mapped to each other via the involution have the same weight.
\item $\Gamma$ does not have sinks or sources.
\item The inner vertices are ordered compatibly with the partial ordering induced by the directions of the edges.
\item Every bounded edge $e$ of the graph is equipped with a weight $w(e)\in \mathbb{N}$. These satisfy the \emph{balancing condition} at each inner vertex: the sum of all weights of  incoming  edges equals the sum of the weights of all outgoing edges.
\item The four edges adjacent to a $4$-valent vertex all have the same weight.
\item The involution is compatible with the weights.
\end{enumerate}
We call two twisted monodromy graphs $\Gamma,\Gamma'$ isomorphic, when there is a graph isomorphism $g\colon\Gamma\to \Gamma'$ that respects the orderings, the involutions and edge weights of both graphs.
\end{definition}

Next, we define quotients of twisted monodromy graphs.

\begin{definition}
    Let $\Gamma$ be a twisted monodromy graph of type $(g,\mu,\nu)$. We define its quotient graph $\overline{\Gamma}$ by identifying vertices $v,v'$ with $\iota(v)=v'$ and edges $e,e'$ with $\iota(e)=e'$. Let $[e]$ be the equivalence class of an edge in $\overline{\Gamma}$, then we define its weight $\omega([e])=\omega(e)=\omega(\iota(e))$.
\end{definition}

We note that by the same argument as in the proof of \cite[Corollary 24]{HMtwisted1} there is a natural bijection between twisted monodromy graphs and twisted tropical covers, as well as between their quotients (see also Section \ref{sec:tropprelim}).

Thus, we immediately obtain the following analogue of \cref{prop-automquotient}.

\begin{corollary}
    Let $\overline{\Gamma}$ be the quotient of a twisted monodromy graph. Assume $\overline{\Gamma}$ has $c$ $2$-valent vertices and is of genus $g$. Then, we have
    \begin{equation}
        \sum \frac{1}{|\mathrm{Aut}(\Gamma)|}=\frac{2^g}{2^{c+1}|\mathrm{Aut}(\overline{\Gamma})|},
    \end{equation}
    where the sum is over all twisted monodromy graphs $\Gamma$ whose quotient is $\overline{\Gamma}$.
\end{corollary}

As a next step, we define vertex multiplicities in analogy with \cref{def:brpointmult}.

\begin{definition}
    Let $\Gamma$ a twisted monodromy graph with quotient $\overline{\Gamma}$. Let $v$ a vertex of $\overline{\Gamma}$. We have the following cases.
    \begin{enumerate}
        \item If $v$ is a $2$-valent vertex with both adjacent edges having weight $\omega_v$. Then, we define its vertex multiplicity $m_v=b\cdot(\omega_v-1)$.
        \item If $v$ is a vertex that joins two edges (considered from right to left), then we define its vertex multiplicity $m_v=1+b$.
        \item For all vertices $v$, we define the vertex multiplicity $m_v=1$.
    \end{enumerate}
\end{definition}

As a corollary of \cref{thm-corresdoubleHNb,lem-quotient}, we obtain:

\begin{corollary}
\label{cor-quotmult}
 The  double $b$-Hurwitz numbers $h^{(b)}_g\left(\begin{smallmatrix}\nu \\ \mu\end{smallmatrix}\right)$ is equal to the weighted sum of twisted monodromy graphs of type $(g,\mu,\nu)$ counted with multiplicity
\begin{equation}
        \frac{2}{b+1}\cdot \frac{1}{|\mathrm{Aut}(\pi)|}\cdot\prod_{v\in V(\overline{\Gamma})} m_v\prod_e\omega(e).
    \end{equation}
    
    or equivalently to the weighted sum of twisted quotient monodromy graphs $\overline{\Gamma}$ counted with multiplicity
    \begin{equation}
     (1+b)^{g-1}\cdot \frac{1}{|\Aut(\overline{\pi})|}\cdot\prod_{v \mbox{ $2$-valent}} \frac{m_v}{2}   \prod_e\omega(e).
\end{equation}
\end{corollary}

We are now ready to prove \cref{thm-piecepoly}.

\begin{proof}[{Proof of \cref{thm-piecepoly}}]
Let $\Gamma$ be a twisted monodromy graph of type $(g,\mu,\nu)$, where we forget the vertex ordering. Treating the entries of $\mu$ and $\nu$ as variables, we obtain that the edge weights of $\Gamma$ also depend on these. Moreover, we fix $\mathfrak{o}(\Gamma)$ to be the number of vertex orderings compatible with edge directions of $\Gamma$. 
We observe that from \cref{cor-quotmult}, the polynomiality in $b$ is immediate. Thus, we now focus on the piecewise polynomiality in the entries of $\mu$ and $\nu$ and thus may treat $\frac{2}{b+1}$ and $\frac{1}{|\mathrm{Aut}(\pi)|}$ in \cref{cor-quotmult} as parameters independent of $\mu$ and $\nu$. We therefore need to derive the piecewise polynomiality of the factors $\prod\omega(e)$ and $\prod m_v$.

Assume $\Gamma$ has $c$ $4$-valent vertices. Then by the same arguments as in \cite{HMtwisted1}, we have that $\Gamma$ has
\begin{itemize}
	\item $2(\ell(\mu)+\ell(\nu))+3g-3-c$ bounded edges and
	\item $2(\ell(\mu)+\ell(\nu))+2g-2c-2$ $3$-valent vertices.
\end{itemize}
Equivalently, we obtain that $\overline{\Gamma}$ has
\begin{itemize}
	\item $\ell(\mu)+\ell(\nu)+\frac{3g-3-c}{2}$ bounded edges and
	\item $\ell(\mu)+\ell(\nu)+g-c-1$ $3$-valent vertices.
\end{itemize}
Moreover, $\overline{\Gamma}$ has $c$ $2$-valent vertices. An Euler characteristic computation yields that the genus $g'$ of $\overline{\Gamma}$ is $g'=\frac{1}{2}(g-c+1)$. We may introduce $g'$ parameters $i_1,\dots,i_{g'}$ (one for each independent cycle), such that the balancing condition imposes and expression of all edge weights in the entries of $\mu$ and $\nu$ as well as in $i_1,\dots,i_{g'}$. Thus, the contribution of $\overline{\Gamma}$ to $h_g^b(\mu,\nu)$ is the sum of the multiplicities for all "legal" choices of $i_j$. Indeed, the only restriction is that edge weights are non-negative. Thus, by the same argument as in the proof of \cite[Theorem 27]{HMtwisted1}, we obtain that $h^{(b)}_g\left(\begin{smallmatrix}\nu \\ \mu\end{smallmatrix}\right)$ is piecewise polynomial of desired degree.
\end{proof}

\begin{example}
    We note that we have already seen an example of \cref{thm-piecepoly} in \cref{ex-tropbcomp}. More precisely, we have computed that
    \begin{equation}
        h^{(b)}_1\left(\begin{smallmatrix}(m^1) \\ (m^1)\end{smallmatrix}\right) =\frac{1}{2}\sum_{i=1}^{m-1}i(m-i)+\frac{1}{4(1+b)}b^2(m-1)^2m.
    \end{equation}
    By Faulhaber's formulas, we obtain
    \begin{equation}
        h^{(b)}_1\left(\begin{smallmatrix}(m^1) \\ (m^1)\end{smallmatrix}\right)=\frac{1}{12}(m-1)m(m+1)+\frac{1}{4(1+b)}b^2(m-1)^2m.
    \end{equation}
    We immediately observe the necessity of the factor $(1+b)$ for the polynomiality in $b$. Since there is only one chamber here, there is actual global polynomiality in the coefficients.
    
\end{example}

\subsection{Tropical count of single monotone $b$-Hurwitz numbers}

In this subsection, we derive a tropicalization of single $b$--monotone Hurwitz numbers. Recall the recursion for the double monotone case we produced in Section~\ref{sc:ref-mon}. For the case $\nu = 1^n$, we have $n_p = 1, a = 1$ and the recursion of Theorem~\ref{th:cut-and-join-monotone} simplifies to

\begin{gather}
N_g\left(\begin{smallmatrix}
    1^{|
    \mu|}\\ m_1,\ldots,m_q|r
\end{smallmatrix}\right) = \left(  \sum_{j \ne r} N_{g} \left( \begin{smallmatrix} 1^{|
    \mu|} \\ m_1,\ldots,\hat m_j,
 \ldots, m_r + m_j,\ldots,m_q| r'\end{smallmatrix}\right) + \right. \\ \left. (m_r - 1 )T N_{g-\frac 12}\left( \begin{smallmatrix} 1^{|
    \mu|}\\ m_1,\ldots,m_q| r\end{smallmatrix}\right) \right)
+\sum_{\alpha+ \beta = m_r} 2CJ \beta \Biggl( N_{g-1} \left( \begin{smallmatrix} 1^{|
    \mu|} \\ m_1,\ldots,\alpha,\beta,\ldots,m_q| r\end{smallmatrix}\right) + \Biggr. \\ \Biggl.   \sum_{g_1 + g_2 = g} \sum_{
\begin{smallmatrix}  I_1 \cup I_2 = \{1,\ldots,q\},\ I_1\cup I_2 = \{r\}
\end{smallmatrix}
}  {|\mu| - 1 \choose \beta + \sum_{i\in I_1, i\ne r}m_i} \mathrm{N}_{g_1}(\begin{smallmatrix}1^{\beta + \sum_{i\in I_1, i\ne r}m_i}\\ m_{I_1}({m_r \mapsto  \beta)}\end{smallmatrix})N_{g_2}(\begin{smallmatrix} 1^{\alpha + \sum_{j\in I_2, j\ne r}m_j}\\ m_{I_2}(m_r \mapsto \alpha)|r\end{smallmatrix})\Biggr),
\end{gather}
where we omit $a$ from the symbol for $N^a_g$ to make the notation lighter. To take in considerations the automorphisms of $1^{|\mu|}$ (see formula~\ref{eq:aut}), all the numbers $N_g\left(\begin{smallmatrix}
    1^{|
    \mu|}\\ m_1,\ldots,m_q|r
\end{smallmatrix}\right)$ have to be divided by $|\mu|!$. Notice, that in the corresponding recursion for the numbers $\frac 1{|\mu|!}N_g\left(\begin{smallmatrix}
    1^{|
    \mu|}\\ m_1,\ldots,m_q|r
\end{smallmatrix}\right)$ the binomial coefficient ${|\mu| - 1 \choose \beta + \sum_{i\in I_1, i\ne r}m_i}$ gets replaced by the ratio $\frac{\alpha + \sum_{j\in I_2, j\ne r}m_j}{|\mu|}.$ This recursion admits even further simplification due to the following Lemma:
\begin{lemma}
For any $g \in \frac 12\mathbb N\cup\{0\}$, any  composition $(m_1,\ldots ,m_q)$ and any $r=1,\ldots,q$ we have 
\[
N_g\left(\begin{smallmatrix}
    1^{|
    \mu|}\\ m_1,\ldots,m_q|r
\end{smallmatrix}\right) = \frac {m_r}{|
\mu|} \mathrm N_g\left(\begin{smallmatrix}
    1^{|
    \mu|}\\ m_1,\ldots,m_q
\end{smallmatrix}\right).
\]
\end{lemma}
\begin{proof}
Recall, that $N_g\left(\begin{smallmatrix}
    1^{|
    \mu|}\\ m_1,\ldots,m_q|r
\end{smallmatrix}\right),$ is a weighted number of transitive strings of transpositions $(\sigma_1,\ldots,\sigma_k)$ such that $\sigma_k\cdots \sigma_1.\tau$ is of type $\mu$ that corresponds to the composition $(m_1,\ldots,m_q),$ and the element $n = |\mu|$ is in cycle of $(\sigma_k\cdots \sigma_1.\tau)\tau$ of length $2m_r.$ Due to Proposition~\ref{prop:rhoindependence}, all fixed point-free involutions $\varrho \in d_\mu$ contribute equally to the number  $\mathrm N_g\left(\begin{smallmatrix}
    1^{|
    \mu|}\\ m_1,\ldots,m_q
\end{smallmatrix}\right).$ Thus, $N_g\left(\begin{smallmatrix}
    1^{|
    \mu|}\\ m_1,\ldots,m_q|r
\end{smallmatrix}\right)$ makes a share of $\mathrm N_g\left(\begin{smallmatrix}
    1^{|
    \mu|}\\ m_1,\ldots,m_q
\end{smallmatrix}\right)$ equal to the share of fixed point-free involutions with ordered cycles of type $(m_1,\ldots,m_q)$ with element $n$ belonging to the cycle $r$, among all fixed point-free involutions with ordered cycles of type $(m_1,\ldots,m_q)$. This share is $\frac {m_r}{|
\mu|}$.
\end{proof}

Combining the established facts, we obtain the following recursion for the numbers $\mathcal N_g(m_1,\ldots,m_q) = \frac 1{|\mu|!} \mathrm N_g\left(\begin{smallmatrix}
    1^{|
    \mu|}\\ m_1,\ldots,m_q
\end{smallmatrix}\right)$ (cf. the recursion of~\cite{DDM}):

\begin{theorem}\label{th:rec_single_monotone1}
The numbers $\mathcal N_g(m_1,\ldots,m_q)$ satisfy the recurrence
\begin{gather}
m_r\mathcal N_g(m_1,\ldots,m_q) = \sum_{j \ne r}(m_r +m_j) \mathcal N_g(m_1,\ldots,\hat m_j,\ldots m_r + m_j,\ldots,m_q) + \\  m_r(m_r - 1 )T \mathcal N_{g-\frac 12}(m_1,\ldots,m_q)
+\sum_{\alpha+ \beta = m_r} 2CJ \alpha \beta \Biggl( \mathcal N_{g-1}(m_1,\ldots,\hat m_r, \alpha, \beta, \ldots,m_q) + \Biggr. \\ \Biggl.   \sum_{g_1 + g_2 = g} \sum_{
\begin{smallmatrix}  I_1 \cup I_2 = \{1,\ldots,q\},\ I_1\cup I_2 = \{r\}
\end{smallmatrix}
}  \mathcal N_{g_1}\left(m_{I_1}(m_r \mapsto \beta)\right)\mathcal N_{g_2}\left(m_{I_2}(m_r \mapsto \alpha)\right)\Biggr),
\end{gather}

for any $r = 1,\ldots,q$ and the initial condition $\mathcal N_0(m) = \delta_{m,1}\frac 1{2CJ}.$
\end{theorem}

 We use the following notation: 

\begin{equation}
    h^\le_g(m_1,\dots,m_q)=\mathcal N_g(m_1,\ldots,m_q)|_{\substack{CJ\to \frac{1+b}{2}\\T\to b}}.
\end{equation}

Summing over all $q,$ the recursion takes the following form. Notice that this form can also be derived from  \cite[Theorem 2.4]{Bonzom-Chapuy-Dolega} by expanding the involved generating series and comparing coefficients. We consider the tuple $\underline m = (m_1,\ldots,m_q)$ as a composition of a number $d\in \mathbb N$, and make use of the fact that the numbers we compute do not depend on the orders of its parts.

		\begin{gather} \label {Eq:MonCJ}
			dh_g^\le(m_1,\ldots,m_q) =
			\sum_{\substack{ i\neq j}}^s (m_i + m_j) h_g^\le(\underline m(\hat{i},\hat{j}),m_i+m_j)\\
			+(1+b)\sum_{i=1}^q\sum_{\alpha+\beta = m_i}\alpha\beta\left[ 
			h^\le_{g-1}(\underline m(\hat{i}),\alpha,\beta)+ \sum_{\substack{g_1+g_2=g\\ \underline n_1\cup \underline n_2 = \underline m(\hat{i})}}
			\binom{d}{|\underline n_1|,|\underline n_2|}\cdot \frac{1}{\sigma}\cdot  h^\le_{g_1}(\underline n_1,\alpha)\cdot h
            _{g_2}^\le(\underline n_2,\beta) \right]\\
			+b\sum_{i=1}^q m_i(m_i-1)h^\le_{g-\frac 12}(\underline m)
		\end{gather}
		where in the third summand, $\sigma=2$ if $\nu_1=\nu_2$, $g_1=g_2$ and $\alpha=\beta$ and $\sigma=1$ else;
        with the initial condition $h^0(1) = \frac{1}{1+b}$; 
		where the $\hat{}$ sign means removing the entry.

As in Section \ref{sec:tropb}, we derive a graph formalism from this recursion that can be interpreted as a count of tropical covers. 
Recall the notion of twisted tropical covers and their quotient covers from Section \ref{sec:tropprelim}. 

\begin{definition}
    Consider a tropical quotient cover $\overline{\pi}:\overline{\Gamma}\rightarrow \mathbb{R}$.
    Assume we cut all edges between the branch point $p_{i-1}$ and the branch point $p_i$. 

If the vertex $V$ that maps to $p_i$ is $2$-valent, we define its \emph{vertex multiplicity} $\mathrm m(V):=\frac{1}{d}\cdot b \cdot (\omega_V-1)$, where $d$ is the degree of the connected component containing $V$ after cutting as above and $\omega_V$ is the weight of its adjacent edges. 
    
    If $V$ is $3$-valent, we define $\mathrm m(V):=\frac{2}{d}$.

    Two tropical quotient covers $\overline{\pi}$ and $\overline{\pi'}$ are \emph{equivalent}, if they differ only in the metric, more precisely in the order of the branch points that are the images of vertices coming from two different connected components after cutting all edges between the branch points $p_{i-1}$ and $p_i$.

\end{definition}

\begin{example}

The two tropical quotient covers shown in Figure \ref{fig-exequiv} are equivalent: if we cut all edges between the second and third branch point from the left, we obtain two connected components on the right. The orders of the images of the vertices in these connected components differ. E.g.\ in the left picture, the third branch point is the image of the first vertex of the upper connected component, whereas it is the image of the first vertex of the lower connected component in the second picture.

   \begin{figure}
       \centering

\tikzset{every picture/.style={line width=0.75pt}} 

\begin{tikzpicture}[x=0.75pt,y=0.75pt,yscale=-1,xscale=1]

\draw    (110,330) -- (320,330) ;
\draw    (150,270) .. controls (159.26,261.64) and (231.11,260.86) .. (240,270) ;
\draw    (110,280) -- (160,280) ;
\draw    (110,270) -- (150,270) ;
\draw    (200,280) .. controls (218.4,276.27) and (223.2,277.47) .. (240,270) ;
\draw    (200,280) .. controls (215.6,282.87) and (250.8,284.47) .. (260,290) ;
\draw    (240,270) -- (270,270) ;
\draw    (180,318) -- (180,300) ;
\draw [shift={(180,320)}, rotate = 270] [color={rgb, 255:red, 0; green, 0; blue, 0 }  ][line width=0.75]    (10.93,-3.29) .. controls (6.95,-1.4) and (3.31,-0.3) .. (0,0) .. controls (3.31,0.3) and (6.95,1.4) .. (10.93,3.29)   ;
\draw    (180,280) -- (200,280) ;
\draw  [fill={rgb, 255:red, 0; green, 0; blue, 0 }  ,fill opacity=1 ] (152.54,329.83) .. controls (152.54,329.07) and (151.97,328.46) .. (151.27,328.46) .. controls (150.57,328.46) and (150,329.07) .. (150,329.83) .. controls (150,330.59) and (150.57,331.2) .. (151.27,331.2) .. controls (151.97,331.2) and (152.54,330.59) .. (152.54,329.83) -- cycle ;
\draw    (150,270) .. controls (159.2,279.67) and (232.4,306.27) .. (260,290) ;
\draw    (260,290) -- (280,290) ;
\draw  [fill={rgb, 255:red, 0; green, 0; blue, 0 }  ,fill opacity=1 ] (202.54,329.77) .. controls (202.54,329.01) and (201.97,328.4) .. (201.27,328.4) .. controls (200.57,328.4) and (200,329.01) .. (200,329.77) .. controls (200,330.53) and (200.57,331.14) .. (201.27,331.14) .. controls (201.97,331.14) and (202.54,330.53) .. (202.54,329.77) -- cycle ;
\draw  [fill={rgb, 255:red, 0; green, 0; blue, 0 }  ,fill opacity=1 ] (241.63,330.06) .. controls (241.63,329.3) and (241.06,328.69) .. (240.36,328.69) .. controls (239.65,328.69) and (239.09,329.3) .. (239.09,330.06) .. controls (239.09,330.81) and (239.65,331.43) .. (240.36,331.43) .. controls (241.06,331.43) and (241.63,330.81) .. (241.63,330.06) -- cycle ;
\draw  [fill={rgb, 255:red, 0; green, 0; blue, 0 }  ,fill opacity=1 ] (261.86,330.11) .. controls (261.86,329.36) and (261.29,328.74) .. (260.58,328.74) .. controls (259.88,328.74) and (259.31,329.36) .. (259.31,330.11) .. controls (259.31,330.87) and (259.88,331.49) .. (260.58,331.49) .. controls (261.29,331.49) and (261.86,330.87) .. (261.86,330.11) -- cycle ;
\draw    (270,270) -- (320,260) ;
\draw    (270,270) -- (320,280) ;
\draw    (280,290) -- (300,300) ;
\draw    (280,290) -- (320,290) ;
\draw    (300,300) -- (320,300) ;
\draw    (300,300) -- (320,310) ;
\draw  [fill={rgb, 255:red, 0; green, 0; blue, 0 }  ,fill opacity=1 ] (271.66,330.2) .. controls (271.66,329.44) and (271.09,328.82) .. (270.38,328.82) .. controls (269.68,328.82) and (269.11,329.44) .. (269.11,330.2) .. controls (269.11,330.95) and (269.68,331.57) .. (270.38,331.57) .. controls (271.09,331.57) and (271.66,330.95) .. (271.66,330.2) -- cycle ;
\draw  [fill={rgb, 255:red, 0; green, 0; blue, 0 }  ,fill opacity=1 ] (281.06,330) .. controls (281.06,329.24) and (280.49,328.62) .. (279.78,328.62) .. controls (279.08,328.62) and (278.51,329.24) .. (278.51,330) .. controls (278.51,330.75) and (279.08,331.37) .. (279.78,331.37) .. controls (280.49,331.37) and (281.06,330.75) .. (281.06,330) -- cycle ;
\draw  [fill={rgb, 255:red, 0; green, 0; blue, 0 }  ,fill opacity=1 ] (301.26,330) .. controls (301.26,329.24) and (300.69,328.62) .. (299.98,328.62) .. controls (299.28,328.62) and (298.71,329.24) .. (298.71,330) .. controls (298.71,330.75) and (299.28,331.37) .. (299.98,331.37) .. controls (300.69,331.37) and (301.26,330.75) .. (301.26,330) -- cycle ;
\draw    (260,290) -- (290,290) ;
\draw    (359.5,330.36) -- (569.5,330.36) ;
\draw    (399.5,270.36) .. controls (408.76,262) and (510.61,261.22) .. (519.5,270.36) ;
\draw    (359.5,280.36) -- (409.5,280.36) ;
\draw    (359.5,270.36) -- (399.5,270.36) ;
\draw    (449.5,280.36) .. controls (467.9,276.62) and (502.7,277.82) .. (519.5,270.36) ;
\draw    (449.5,280.36) .. controls (465.1,283.22) and (480.8,284.47) .. (490,290) ;
\draw    (519.5,270.36) -- (550,270) ;
\draw    (429.5,318.36) -- (429.5,300.36) ;
\draw [shift={(429.5,320.36)}, rotate = 270] [color={rgb, 255:red, 0; green, 0; blue, 0 }  ][line width=0.75]    (10.93,-3.29) .. controls (6.95,-1.4) and (3.31,-0.3) .. (0,0) .. controls (3.31,0.3) and (6.95,1.4) .. (10.93,3.29)   ;
\draw    (429.5,280.36) -- (449.5,280.36) ;
\draw  [fill={rgb, 255:red, 0; green, 0; blue, 0 }  ,fill opacity=1 ] (402.04,330.19) .. controls (402.04,329.43) and (401.47,328.81) .. (400.77,328.81) .. controls (400.07,328.81) and (399.5,329.43) .. (399.5,330.19) .. controls (399.5,330.94) and (400.07,331.56) .. (400.77,331.56) .. controls (401.47,331.56) and (402.04,330.94) .. (402.04,330.19) -- cycle ;
\draw    (399.5,270.36) .. controls (408.7,280.02) and (462.4,306.27) .. (490,290) ;
\draw  [fill={rgb, 255:red, 0; green, 0; blue, 0 }  ,fill opacity=1 ] (452.04,330.13) .. controls (452.04,329.37) and (451.47,328.76) .. (450.77,328.76) .. controls (450.07,328.76) and (449.5,329.37) .. (449.5,330.13) .. controls (449.5,330.89) and (450.07,331.5) .. (450.77,331.5) .. controls (451.47,331.5) and (452.04,330.89) .. (452.04,330.13) -- cycle ;
\draw  [fill={rgb, 255:red, 0; green, 0; blue, 0 }  ,fill opacity=1 ] (491.13,330.41) .. controls (491.13,329.66) and (490.56,329.04) .. (489.86,329.04) .. controls (489.15,329.04) and (488.59,329.66) .. (488.59,330.41) .. controls (488.59,331.17) and (489.15,331.79) .. (489.86,331.79) .. controls (490.56,331.79) and (491.13,331.17) .. (491.13,330.41) -- cycle ;
\draw  [fill={rgb, 255:red, 0; green, 0; blue, 0 }  ,fill opacity=1 ] (511.36,330.47) .. controls (511.36,329.71) and (510.79,329.1) .. (510.08,329.1) .. controls (509.38,329.1) and (508.81,329.71) .. (508.81,330.47) .. controls (508.81,331.23) and (509.38,331.84) .. (510.08,331.84) .. controls (510.79,331.84) and (511.36,331.23) .. (511.36,330.47) -- cycle ;
\draw    (550,270) -- (570,260) ;
\draw    (550,270) -- (570,280) ;
\draw    (510,290) -- (530,300) ;
\draw    (510,290) -- (570,290) ;
\draw    (530,300) -- (570,300) ;
\draw    (530,300) -- (570,310) ;
\draw  [fill={rgb, 255:red, 0; green, 0; blue, 0 }  ,fill opacity=1 ] (521.16,330.55) .. controls (521.16,329.8) and (520.59,329.18) .. (519.88,329.18) .. controls (519.18,329.18) and (518.61,329.8) .. (518.61,330.55) .. controls (518.61,331.31) and (519.18,331.92) .. (519.88,331.92) .. controls (520.59,331.92) and (521.16,331.31) .. (521.16,330.55) -- cycle ;
\draw  [fill={rgb, 255:red, 0; green, 0; blue, 0 }  ,fill opacity=1 ] (530.56,330.35) .. controls (530.56,329.6) and (529.99,328.98) .. (529.28,328.98) .. controls (528.58,328.98) and (528.01,329.6) .. (528.01,330.35) .. controls (528.01,331.11) and (528.58,331.72) .. (529.28,331.72) .. controls (529.99,331.72) and (530.56,331.11) .. (530.56,330.35) -- cycle ;
\draw  [fill={rgb, 255:red, 0; green, 0; blue, 0 }  ,fill opacity=1 ] (550.76,330.35) .. controls (550.76,329.6) and (550.19,328.98) .. (549.48,328.98) .. controls (548.78,328.98) and (548.21,329.6) .. (548.21,330.35) .. controls (548.21,331.11) and (548.78,331.72) .. (549.48,331.72) .. controls (550.19,331.72) and (550.76,331.11) .. (550.76,330.35) -- cycle ;
\draw    (490,290) -- (510,290) ;

\draw (264.64,265.14) node  [font=\footnotesize] [align=left] {\begin{minipage}[lt]{12.34pt}\setlength\topsep{0pt}
$\displaystyle 2$
\end{minipage}};
\draw (130.36,266.86) node  [font=\footnotesize] [align=left] {\begin{minipage}[lt]{12.34pt}\setlength\topsep{0pt}
$ $
\end{minipage}};
\draw (274.59,284.72) node  [font=\footnotesize] [align=left] {\begin{minipage}[lt]{12.34pt}\setlength\topsep{0pt}
$\displaystyle 3$
\end{minipage}};
\draw (218.9,303.6) node  [font=\footnotesize] [align=left] {\begin{minipage}[lt]{12.34pt}\setlength\topsep{0pt}
$\displaystyle 2$
\end{minipage}};
\draw (130.5,262.8) node  [font=\footnotesize] [align=left] {\begin{minipage}[lt]{12.34pt}\setlength\topsep{0pt}
$\displaystyle 3$
\end{minipage}};
\draw (131.9,291.2) node  [font=\footnotesize] [align=left] {\begin{minipage}[lt]{12.34pt}\setlength\topsep{0pt}
$\displaystyle 2$
\end{minipage}};
\draw (292.5,302.8) node  [font=\footnotesize] [align=left] {\begin{minipage}[lt]{12.34pt}\setlength\topsep{0pt}
$\displaystyle 2$
\end{minipage}};
\draw (538.5,264) node  [font=\footnotesize] [align=left] {\begin{minipage}[lt]{12.34pt}\setlength\topsep{0pt}
$\displaystyle 2$
\end{minipage}};
\draw (379.86,267.21) node  [font=\footnotesize] [align=left] {\begin{minipage}[lt]{12.34pt}\setlength\topsep{0pt}
$ $
\end{minipage}};
\draw (501.1,284) node  [font=\footnotesize] [align=left] {\begin{minipage}[lt]{12.34pt}\setlength\topsep{0pt}
$\displaystyle 3$
\end{minipage}};
\draw (467.6,304.76) node  [font=\footnotesize] [align=left] {\begin{minipage}[lt]{12.34pt}\setlength\topsep{0pt}
$\displaystyle 2$
\end{minipage}};
\draw (378,265.16) node  [font=\footnotesize] [align=left] {\begin{minipage}[lt]{12.34pt}\setlength\topsep{0pt}
$\displaystyle 3$
\end{minipage}};
\draw (381.8,291.16) node  [font=\footnotesize] [align=left] {\begin{minipage}[lt]{12.34pt}\setlength\topsep{0pt}
$\displaystyle 2$
\end{minipage}};
\draw (523.3,306) node  [font=\footnotesize] [align=left] {\begin{minipage}[lt]{12.34pt}\setlength\topsep{0pt}
$\displaystyle 2$
\end{minipage}};

\end{tikzpicture}

       \caption{Two equivalent tropical quotient covers: they differ only by the order of the images of vertices in the two connected components we obtain when we cut all edges between the second and third branch point.}
       \label{fig-exequiv}
   \end{figure}
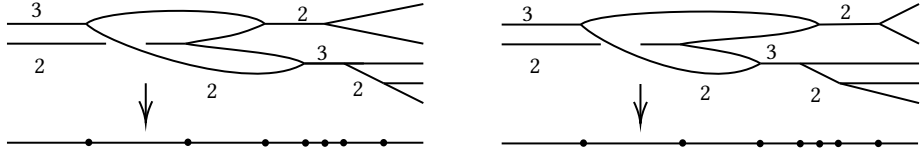 
\end{example}

\begin{definition}
    Let $m=(m_1,\ldots,m_q)$ be a composition of  $d\in \mathbb N$ for $g \in \frac 12 \mathbb N \cup \{0\}$, the integer  $k = 2g - 2 + |\lambda| + d>0$. Then, we define \textbf{tropical monotone $b$-Hurwitz numbers} as
    \begin{equation}
        h^{\le,\textrm{trop}}_g(m)=\sum_{\overline{\pi}}(1+b)^{g(\overline{\Gamma})-1}\Big(\frac{1}{2}\Big)^w\prod_{e}\omega(e)\prod_{V}\mathrm m(V),
    \end{equation}
    where the sum is over all equivalence classes of tropical quotient covers $\overline{\pi}:\overline{\Gamma}\rightarrow \mathbb{R}$ with $k$ inner vertices, left ends of weights given by $m$ and all right ends of weight $1$, $g(\overline{\Gamma})$ denotes the genus of the quotient graph $\Gamma$, $w$ denotes the number of wieners, i.e.\ cycles formed by two edges of the same weight, the first product is over all internal edges $e$ and the second product is over all vertices.

This is well-defined, since the multiplicity with which we count a tropical quotient cover does not depend on the representative of an equivalence class.

    If $k=0$, we let  $h_0^{\le,\textrm{trop}}(1)=\frac{1}{1+b}$.
\end{definition}

\begin{theorem}
    \label{thm-tropmon}
    Let $m=(m_1,\ldots,m_q)$ be a composition of $d\in \mathbb N,$ and $g\in \frac 12\mathbb N \cup\{0\}$. Then, we have
    \begin{equation}
         h^{\le}_{g}(m)= h^{\le,\trop}_g(m).
    \end{equation}
\end{theorem}

\begin{proof}

The proof is very close to the proof of Theorem \ref{thm-corresdoubleHNb}. Here, we stress only the differences: the vertex multiplicities contain factors of $\frac{2}{d}$, where $d$ is the degree of the connected component a vertex lives in after cutting the edges between the branch points to the left. This is because the recursion comes with a factor of $d$ on the left, which we divide by, and has no factor of $\frac{1}{2}$ on the right. Since the tropical quotient cover is connected, $d$ is the degree. If we continue the recursion after cutting of a vertex producing two connected components (corresponding to the third summand in the recursion), the input degree changes, which is in accordance with the definition of vertex multiplicity. Another difference concerns the different binomial factor in the third summand compared to the recursion of Theorem \ref{thm:recursionbHN}. Here, the binomial factor is in charge of distributing labels of the right ends to the two connected components (notice an extra factor of $\frac{1}{2}$ if the two parts are indistinguishable, then it does not matter which part obtains which labels), a contribution which in \ref{thm-corresdoubleHNb} was summed over separately (since there, we had not only ends of weight one). There, instead, there was a binomial factor taking care of vertex orderings of the two connected components. Since here, we count equivalence classes of tropical quotient covers, we do not have a binomial factor for the alignment of the vertices in the two parts.

\end{proof}

\begin{example}
 
Figure \ref{fig:extropmon} depicts the tropical count of $h^{\le,\trop}_1((2))$.

   \begin{figure}
        \centering

\tikzset{every picture/.style={line width=0.75pt}} 

\begin{tikzpicture}[x=0.75pt,y=0.75pt,yscale=-1,xscale=1]

\draw    (140,380) .. controls (149.26,371.64) and (171.11,370.86) .. (180,380) ;
\draw    (110,380) -- (140,380) ;
\draw  [fill={rgb, 255:red, 0; green, 0; blue, 0 }  ,fill opacity=1 ] (301.56,379.74) .. controls (301.56,378.98) and (300.99,378.37) .. (300.29,378.37) .. controls (299.58,378.37) and (299.01,378.98) .. (299.01,379.74) .. controls (299.01,380.5) and (299.58,381.11) .. (300.29,381.11) .. controls (300.99,381.11) and (301.56,380.5) .. (301.56,379.74) -- cycle ;
\draw  [fill={rgb, 255:red, 0; green, 0; blue, 0 }  ,fill opacity=1 ] (341.33,379.96) .. controls (341.33,379.2) and (340.76,378.59) .. (340.06,378.59) .. controls (339.36,378.59) and (338.79,379.2) .. (338.79,379.96) .. controls (338.79,380.72) and (339.36,381.33) .. (340.06,381.33) .. controls (340.76,381.33) and (341.33,380.72) .. (341.33,379.96) -- cycle ;
\draw    (270,380) -- (370,380) ;
\draw    (210,380) -- (230,370) ;
\draw    (210,380) -- (230,390) ;
\draw    (140,380) .. controls (149.8,392.07) and (174.2,389.27) .. (180,380) ;
\draw    (180,380) -- (210,380) ;
\draw    (370,380) -- (390,370) ;
\draw    (370,380) -- (390,390) ;

\draw (131.5,374) node  [font=\footnotesize] [align=left] {\begin{minipage}[lt]{12.34pt}\setlength\topsep{0pt}
$\displaystyle 2$
\end{minipage}};
\draw (199.39,373.78) node  [font=\footnotesize] [align=left] {\begin{minipage}[lt]{12.34pt}\setlength\topsep{0pt}
$\displaystyle 2$
\end{minipage}};
\draw (288.5,374) node  [font=\footnotesize] [align=left] {\begin{minipage}[lt]{12.34pt}\setlength\topsep{0pt}
$\displaystyle 2$
\end{minipage}};
\draw (321.5,374) node  [font=\footnotesize] [align=left] {\begin{minipage}[lt]{12.34pt}\setlength\topsep{0pt}
$\displaystyle 2$
\end{minipage}};
\draw (358.5,374) node  [font=\footnotesize] [align=left] {\begin{minipage}[lt]{12.34pt}\setlength\topsep{0pt}
$\displaystyle 2$
\end{minipage}};
\draw (186.39,413.52) node  [font=\scriptsize] [align=left] {\begin{minipage}[lt]{50.24pt}\setlength\topsep{0pt}
$\displaystyle 1$
\end{minipage}};
\draw (337.72,409.02) node  [font=\scriptsize] [align=left] {\begin{minipage}[lt]{50.24pt}\setlength\topsep{0pt}
$\displaystyle \frac{b^{2}}{( 1+b)}$
\end{minipage}};

\end{tikzpicture}

         \caption{The tropical count of $h^{\le,\trop}_1((2);b)$.}
        \label{fig:extropmon}
    \end{figure}
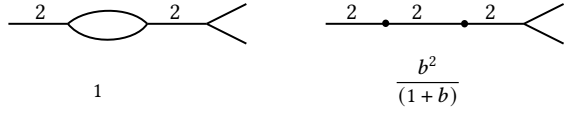
\end{example}

\section{Tropicalizing refined monotone Hurwitz numbers}\label{sc:trop-mono}

In this section, we derive a tropical interpretation of the refined monotone double Hurwitz numbers $N_g^a\binom{n_1,\dots,n_p}{m_1,\dots,m_q|r}$. This is achieved by combining the approach for tropicalising orbifold and double monotone Hurwitz numbers used in \cite{Karev-Do,hahn2019monodromy} with the notion of twisted tropical covers derived in \cite{HMtwisted1}. The main idea is to enrich the twisted tropical covers of \cite{HMtwisted1} by edge colours that keep track of the monotonicity condition and encode the CJT--deformation in vertex multiplicities. We note that while in the previous section, we used the cut--and--join recursion to obtain tropicalizations, here we instead work directly with the factorisation interpretation of refined Hurwitz numbers.

As a first step, we give the following definition.

\begin{definition}
\label{def-montwisttrop}
Let $n$ be a positive integer and $\underline{m}=(m_1,\dots,m_q),\underline{n}=(n_1,\dots,n_p)$ partitions of $n$.
    Let $\pi\colon\Gamma_1\to\Gamma_2$ with involution $\iota\colon\Gamma_1\to\Gamma_1$ be a twisted tropical cover of type $(g,\underline{n},\underline{m})$ in the sense of \cref{def-twistedtropcover}. We assign the following additional data:
    \begin{enumerate}
        \item All edges are either coloured normal, dashed or bold.
        \item The in-ends are labelled $1,\dots,2p$ and the out-ends are labeled $1,\dots,2q$, such that $\iota$ identifies in-ends labeled $i$ and $p+i$ and out-ends labeled $j$ and $q+j$.
        \item The in-ends labelled $1,\dots,p$ are coloured normal and the in-ends labelled $p+1,\dots,2p$ are either labeled dashed or bold.
        \item We call a connected path of bold edges beginning at an in-end a \textbf{chain}.
        \item For a chain $\mathcal{C}$, we denote the first inner vertex by $l_\mathcal{C}$ and the last inner vertex by $r_\mathcal{C}$. We require that for any two chains $\mathcal{C}$ and $\mathcal{C}'$ that $[\pi(l_\mathcal{C}),\pi(r_\mathcal{C})]\cap[\pi(l_{\mathcal{C}'}),\pi(r_{\mathcal{C}'})]=\emptyset$.
        \item The intervals $[\pi(l_\mathcal{C}),\pi(r_\mathcal{C})]$ induce a natural ordering of the chains, namely $\mathcal{C}<\mathcal{C}'$ if and only if $\pi(l_\mathcal{C})<\pi(l_{\mathcal{C}'})$. This ordering is required to be compatible with the labeling. More precisely, we denote by $\mathcal{I}(\mathcal{C})$ the in-end at which $\mathcal{C}$ begins. Then, we have that $\mathcal{C}<\mathcal{C}'$ if and only if $\mathcal{I}(\mathcal{C})<\mathcal{I}(\mathcal{C}')$.
        \item There is a \textbf{counter} $c(e)$ associated to each bold and dashed edge $e$.
        \item The counter of each non-normal in-end is set to $1$.
        \item Let $v$ be an inner vertex of $\Gamma_1$. Then either $v$ or $\iota(v)$ has a unique in-coming bold edge with counter $i(v)$. We require the same vertex and a unique out-going non-normal edge with counter $o(v)$. We require $i(v)\le o(v)$.
        \item Every non-normal edge arises from a unique chain: Every bold edge is part of a chain and every dashed edge is sourced at a bold chain. Let an edge $e$ be part of the chain $C$ starting at the in-end labelled $i$. Then, the counter satisfies $n_{i-p}-\omega(e)<c(e)\le n_{i-p}$.
    \end{enumerate}
    We call the resulting structure a \emph{monotone twisted tropical cover of type $(g,\underline{n},\underline{m})$}.

    Moreover, we call a monotone twisted tropical cover of type $(g,\underline{n},\underline{m})$ to be of type $(g,\underline{n},\underline{m},r,a)$ if for its largest chain $\mathcal{C}$ w.r.t. to the ordering in (6), we have
    \begin{itemize}
        \item final outgoing dashed edge $e$ connects to the out-end labeled $r$ or $q+r$ for $1\le r\le q$
        \item its final counter is $a$.
    \end{itemize}

\end{definition}

We now discuss how monotone twisted tropical covers arise from monotone twisted factorisations. For this, we simply rephrase the proof of \cref{th:cut-and-join-monotone} into terms more compatible with the tropical setting.

To begin with, let $\underline{n}=(n_1,\dots,n_p)$ a composition of $n$. We define a cycle $\sigma_{\underline{n}}^{(i)}$ for $1\le i\le p$ as
\begin{equation}
    \sigma_{\underline{n}}^{(i)}=\left(\overline{\sum_{j=1}^{i-1} n_j+1}\, \overline{\sum_{j=1}^{i-1} n_j+2}\dots\overline{\sum_{j=1}^i n_j}\right)
\end{equation}
and for $p+1\le i\le 2p$ as

\begin{equation}
    \sigma_{\underline{n}}^{(i)}=\left(\tau\sigma_{\underline{n}}^{(i-l)}\tau\right)^{-1}
\end{equation}

and define the permutation

\begin{equation}
    \sigma_{\underline{n}}=\prod_{i=1}^{2p} \sigma_{\underline{n}}^{(i)}.
\end{equation}

We note that $\sigma_{\underline{n}}$ is simply equal to $\rho_{\underline{n}}\tau$, where $\rho_{\underline{n}}$ is the convenient representative constructed in \cref{sc:ref-mon}.

Then, we have that $N_g^a\binom{n_1,\dots,n_p}{m_1,\dots,m_q|r}$ is equal to $\frac{1}{C^qJ^p\prod(2j)^{n_j}}$ the weighted (w.r.t. to the CJT-deformation) count of monotone factorisations $(\sigma_1,\dots,\sigma_k)$, such that
\begin{itemize}
    \item $k=\frac{2g-2+p+q}{2}$,
    \item $\tilde{\sigma}=\sigma_k\cdots\sigma_1\sigma_{\underline{n}}\tau\sigma_1\cdots\sigma_k\tau$ is of cycle type $(\underline{m},\underline{m})$,
    \item the cycles of $\sigma_{\underline{n}}$ and $\tilde{\sigma}$ are labeled by $1,\dots,2p$ and $1,\dots,2q$ respectively, such that the cycle labeled $j$ has length $n_j$ or $m_j$ respectively and such that $\tau\sigma_{\underline{n}}\tau$ and $\tau\tilde{\sigma}\tau$ respectively induce the bijections $i\leftrightarrow \overline{i}$ and $j\leftrightarrow \overline{j}$ for $1\le i\le p$, $1\le j\le q$,
    \item $n$ lies in the cycle of $\tilde{\sigma}$ labeled $r$ or $q+r$,
    \item $\sigma_k=(j\, n-n_p+a)$ or $\sigma_k=(\overline{j}\,n-n_p+a)$.
\end{itemize} 

Introducing notation, we define $B_n^\sim$ to be the set of permutations $\sigma$ in $S_{2n}$, such that $\tau\sigma\tau=\sigma^{-1}$ and without self--symmetric cycles, i.e. without cycles $\eta$ of $\sigma$ satisfying $\tau\eta\tau=\eta^{-1}$ (see \cite[Lemma 2.1]{burman2021ribbon}). Note for any fixed point-free involution $\rho\in S_{2n},$ the product $\rho\tau \in B_n^\sim$.

We now show how to associate a monotone twisted tropical cover to such a factorisation.

\begin{construction}
\label{constr-tropcov}
     Let $(\sigma_1,\dots,\sigma_k)$ a factorisation as above. We let $\Gamma_2$ be $\mathbb{R}$ subdivided by $k$ vertices $p_i$ and construct a monotone twisted tropical cover $\pi\colon\Gamma_1\to\Gamma_2$ associated to this.

\begin{enumerate}
    \item We fix $p_1,\dots,p_k\in V(\Gamma_2)$ with $p_i<p_{i+1}$. Furthermore, we set $p_0=-\infty$ and $p_{k+1}=
    \infty$.
    \item We consider $2p$ ends over the interval $(-\infty,p_1]$, labelled by $\sigma_{\underline{n}}^{(1)},\dots,\sigma_{\underline{n}}^{(2l)}$. The action of $\tau$ on $\sigma_{\underline{n}}$ yields an involution on these ends. We colour all ends labeled $\sigma_{\underline{n}}^{1},\dots,\sigma_{\underline{n}}^{p}$ normal and the others dashed. Moreover, we associate a \textbf{counter} to all dashed ends of value $c(e)=1$.
    \item We first consider the action of the transposition $\sigma_1$ on $\sigma_{\underline{n}}$, i.e. we take 
    \begin{equation}
    \label{equ:firststep}
        \sigma_1\sigma_{\underline{n}}\tau\sigma_1\tau.
    \end{equation} 
    
    We first note that $\sigma_m$ must be of the form $(r_1\, s_1)$ or $(\overline{r_1}\,s_1)$ with $r_1<s_1$. Since $(\sigma_1,\dots,\sigma_k)$ is transitive and by monotonicity $s_i$ -- for $\sigma_i=(p_i\,q_i)$ or $\sigma_i=(\overline{r_i}\,s_i)$ and $r_i<s_i$ -- can only weakly increase, we have 
     \begin{equation}
        r_1\in\left\{1,\dots,n_1\right\}.
    \end{equation}

    By \cite[Theorem 2.10]{burman2021ribbon} and \cite[Construction 4.1]{HMtwisted1}, there are three cases to consider for $\sigma_1\sigma_{\underline{n}}\tau\sigma_1\tau$:
    \begin{enumerate}
        \item In this first case, we have four cycles $\sigma_1^1,\sigma_1^2,\sigma_1^3,\sigma_1^4$ of $\sigma_{\underline{n}}$, such that
        \begin{equation}
            \sigma_1^1\sigma_1^2\sigma_1^3\sigma_1^4\in B^{\sim}_n.
        \end{equation}
        Then, the product in \cref{equ:firststep} joins two pairs of these cycles to two new cycles, e.g. $\sigma_1^1$, $\sigma_1^2$ to a new cycle $\eta_1\sigma_1^1\sigma_1^2\tau\eta_1\tau$; and $\sigma_1^3$, $\sigma_1^4$ to a new cycle $\eta_1\sigma_1^3\sigma_1^4\tau\eta_1\tau$.
        Following the construction in \cite{HMtwisted1}, we create two vertices over $p_1$, each adjacent to two ends. More precisely, we attach the ends that are joined via \cref{equ:firststep} to the same vertex. Moreover, we attach to each vertex two edges projecting to $[p_1,p_2]$ that are temporarily labeled by the corresponding cycles obtained from the join. We colour the incoming edge corresponding to the cycle containing $s_1$ bold and the outgoing edge corresponding to the cycle containing $s_1$ dashed. Moreover, we give the outgoing dashed edge the counter $s_1$. This is illustrated in the following picture for the case that $\sigma_1^1$ is joined with $\sigma_1^2$ and $\sigma_1^3$ is joined with $\sigma_1^4$. Again, the action of $\tau$ on the permutation obtained from the join yields an involution of the tropical cover.

        
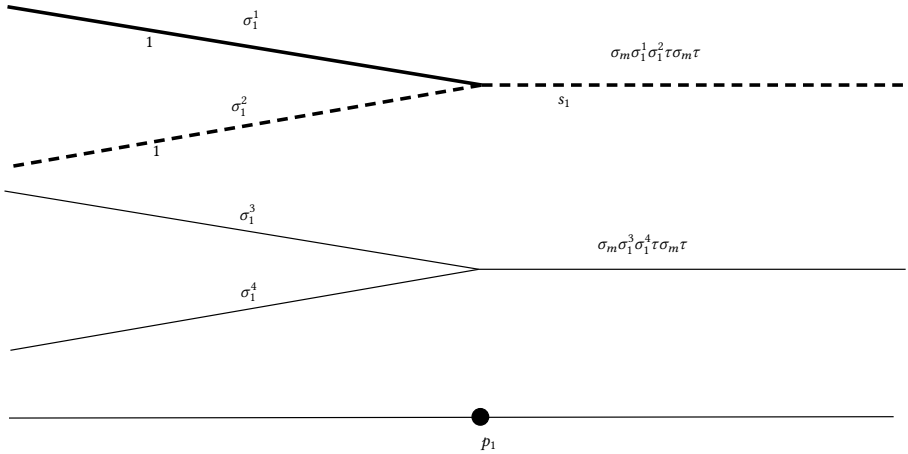
\begin{figure}[H]

        \scalebox{0.6}{

\tikzset{every picture/.style={line width=0.75pt}} 

\begin{tikzpicture}[x=0.75pt,y=0.75pt,yscale=-1,xscale=1]

\draw [line width=2.25]    (40.52,16) -- (439.19,81.86) ;
\draw [line width=2.25]  [dash pattern={on 6.75pt off 4.5pt}]  (45.5,150) -- (439.19,81.86) ;
\draw [line width=2.25]  [dash pattern={on 6.75pt off 4.5pt}]  (439.19,81.86) -- (798,81.86) ;
\draw [line width=0.75]    (38,171) -- (436.67,236.86) ;
\draw [line width=0.75]    (42.98,305) -- (436.67,236.86) ;
\draw [line width=0.75]    (436.67,236.86) -- (795.48,236.86) ;
\draw [line width=0.75]    (41.78,362) -- (785.4,360.86) ;
\draw  [fill={rgb, 255:red, 0; green, 0; blue, 0 }  ,fill opacity=1 ] (431,361) .. controls (431,357.13) and (434.13,354) .. (438,354) .. controls (441.87,354) and (445,357.13) .. (445,361) .. controls (445,364.87) and (441.87,368) .. (438,368) .. controls (434.13,368) and (431,364.87) .. (431,361) -- cycle ;

\draw (437.36,377) node [anchor=north west][inner sep=0.75pt]   [align=left] {$\displaystyle p_{1}$};
\draw (237.06,17) node [anchor=north west][inner sep=0.75pt]   [align=left] {$\displaystyle \sigma _{1}^{1}$};
\draw (227.06,90) node [anchor=north west][inner sep=0.75pt]   [align=left] {$\displaystyle \sigma _{1}^{2}$};
\draw (234.06,181) node [anchor=north west][inner sep=0.75pt]   [align=left] {$\displaystyle \sigma _{1}^{3}$};
\draw (235.06,246) node [anchor=north west][inner sep=0.75pt]   [align=left] {$\displaystyle \sigma _{1}^{4}$};
\draw (546.06,44) node [anchor=north west][inner sep=0.75pt]   [align=left] {$\displaystyle \sigma _{m} \sigma _{1}^{1} \sigma _{1}^{2} \tau \sigma _{m} \tau $};
\draw (535.06,206) node [anchor=north west][inner sep=0.75pt]   [align=left] {$\displaystyle \sigma _{m} \sigma _{1}^{3} \sigma _{1}^{4} \tau \sigma _{m} \tau $};
\draw (155,39) node [anchor=north west][inner sep=0.75pt]   [align=left] {$\displaystyle 1$};
\draw (161,131) node [anchor=north west][inner sep=0.75pt]   [align=left] {$\displaystyle 1$};
\draw (502,90) node [anchor=north west][inner sep=0.75pt]   [align=left] {$\displaystyle {s_{1}}$};

\end{tikzpicture}
}
        
        \end{figure}
           
        \item In the second case, we have two cycles $\sigma_1^1,\sigma_1^2$ of $\sigma_{\underline{n}}$, such that
        \begin{equation}
            \sigma_1^1\sigma_1^2\in B^{\sim}_n.
        \end{equation}
        Then the product in \cref{equ:firststep} splits two cycles $\sigma_1^1$ and $\sigma_1^2$ each into two cycles, that we denote by $\sigma_1^{1,a},\sigma_1^{1,b}$ and $\sigma_1^{2,a},\sigma_1^{2,b}$ respectively. In this case, we create two vertices over $p_1$, each adjacent to one end labeled by $\sigma_1^1$ and $\sigma_1^2$ respectively. Moreover, we attach to each vertex two edges projecting to $[p_1,p_2]$ that are temporarily labeled by the corresponding cycles obtained from the split, i.e. the new edges attached to the vertex adjacent to $\sigma_1^i$ are labeled by $\sigma_1^{i,a}$ and $\sigma_1^{i,b}$. Moreover, colour the incoming edge corresponding to the cycle containing $s_1$ bold and the respective outgoing edge dashed. The outgoing edge is assigned the counter $s_1$. We illustrate this construction in the following picture.
         
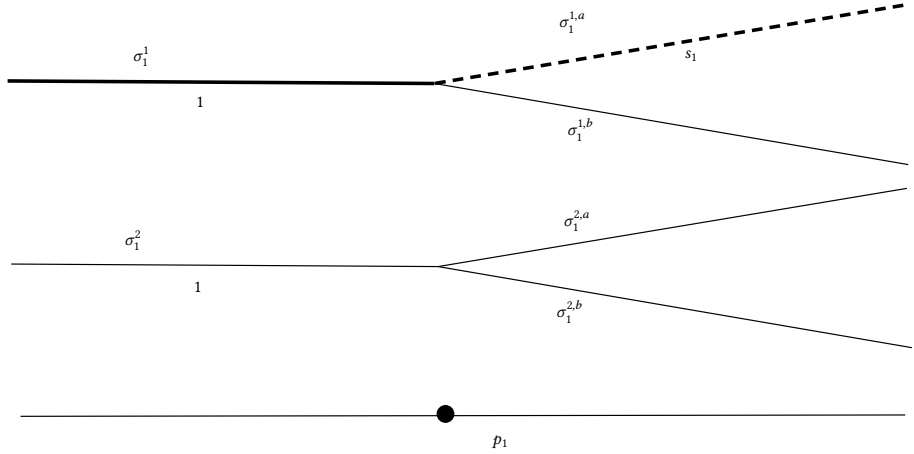
\begin{figure}[H]
        \scalebox{0.6}{

\tikzset{every picture/.style={line width=0.75pt}} 

\begin{tikzpicture}[x=0.75pt,y=0.75pt,yscale=-1,xscale=1]

\draw [line width=0.75]    (791.06,304.26) -- (392.78,236.03) ;
\draw [line width=0.75]    (786.87,170.23) -- (392.78,236.03) ;
\draw [line width=0.75]    (392.78,236.03) -- (33.98,233.9) ;
\draw [line width=0.75]    (41.78,362) -- (785.4,360.86) ;
\draw  [fill={rgb, 255:red, 0; green, 0; blue, 0 }  ,fill opacity=1 ] (392,360) .. controls (392,356.13) and (395.13,353) .. (399,353) .. controls (402.87,353) and (406,356.13) .. (406,360) .. controls (406,363.87) and (402.87,367) .. (399,367) .. controls (395.13,367) and (392,363.87) .. (392,360) -- cycle ;
\draw [line width=0.75]    (788,150.24) -- (389.72,82.01) ;
\draw [line width=2.25]  [dash pattern={on 6.75pt off 4.5pt}]  (783.81,16.21) -- (389.72,82.01) ;
\draw [line width=2.25]    (389.72,82.01) -- (30.92,79.88) ;

\draw (437.36,377) node [anchor=north west][inner sep=0.75pt]   [align=left] {$\displaystyle p_{1}$};
\draw (135,49) node [anchor=north west][inner sep=0.75pt]   [align=left] {$\displaystyle \sigma _{1}^{1}$};
\draw (128,204) node [anchor=north west][inner sep=0.75pt]   [align=left] {$\displaystyle \sigma _{1}^{2}$};
\draw (493,20) node [anchor=north west][inner sep=0.75pt]   [align=left] {$\displaystyle \sigma _{1}^{1,a}$};
\draw (500,110) node [anchor=north west][inner sep=0.75pt]   [align=left] {$\displaystyle \sigma _{1}^{1,b}$};
\draw (497,187) node [anchor=north west][inner sep=0.75pt]   [align=left] {$\displaystyle \sigma _{1}^{2,a}$};
\draw (491,263) node [anchor=north west][inner sep=0.75pt]   [align=left] {$\displaystyle \sigma _{1}^{2,b}$};
\draw (188,91) node [anchor=north west][inner sep=0.75pt]   [align=left] {$\displaystyle 1$};
\draw (186,247) node [anchor=north west][inner sep=0.75pt]   [align=left] {$\displaystyle 1$};
\draw (599,53) node [anchor=north west][inner sep=0.75pt]   [align=left] {$\displaystyle {s_1}$};

\end{tikzpicture}        }
        \end{figure}

        \item In the third case, we have two cycles $\sigma_1^1$ and $\sigma_1^2$ of $\sigma_{\underline{n}}$ of the same length. Here, the product in \cref{equ:firststep} rearranges $\sigma_1^1$ and $\sigma_1^2$ into two new cycles $\tilde{\sigma}_1^1,\tilde{\sigma}_1^2$ of the same length. In this case, we create one vertex over $p_1$ that joins the two ends labeled by $\sigma_1^1,\sigma_1^2$. Moreover, we attach two new edges to this vertex that map to $[p_1,p_2]$ and that are labeled by $\tilde{\sigma}_1^1,\tilde{\sigma}_1^2$. Again, we recolour the incoming edge containing $s_1$ bold and the respective outgoing edge dashed. The dashed outgoing edge is assigned the counter $s_1$.
        We illustrate this construction in the following picture.
    
        
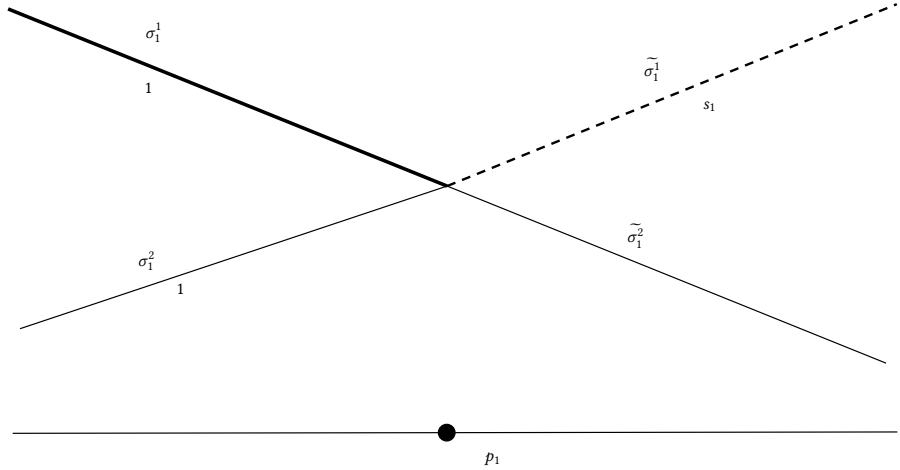
\begin{figure}[H]
        
        \tikzset{every picture/.style={line width=0.75pt}} 
        
        \scalebox{0.6}{

\tikzset{every picture/.style={line width=0.75pt}} 

\begin{tikzpicture}[x=0.75pt,y=0.75pt,yscale=-1,xscale=1]

\draw [line width=0.75]    (34.86,379) -- (778.47,377.86) ;
\draw  [fill={rgb, 255:red, 0; green, 0; blue, 0 }  ,fill opacity=1 ] (392.66,378.43) .. controls (392.66,374.57) and (395.8,371.43) .. (399.66,371.43) .. controls (403.53,371.43) and (406.66,374.57) .. (406.66,378.43) .. controls (406.66,382.3) and (403.53,385.43) .. (399.66,385.43) .. controls (395.8,385.43) and (392.66,382.3) .. (392.66,378.43) -- cycle ;
\draw [line width=2.25]    (31,22) -- (400,171) ;
\draw [line width=0.75]    (400,171) -- (769,320) ;
\draw    (41,291) -- (400,171) ;
\draw [line width=1.5]  [dash pattern={on 5.63pt off 4.5pt}]  (400,171) -- (778,18) ;

\draw (430.43,394) node [anchor=north west][inner sep=0.75pt]   [align=left] {$\displaystyle p_{1}$};
\draw (145,32) node [anchor=north west][inner sep=0.75pt]   [align=left] {$\displaystyle \sigma _{1}^{1}$};
\draw (139,224) node [anchor=north west][inner sep=0.75pt]   [align=left] {$\displaystyle \sigma _{1}^{2}$};
\draw (564,61) node [anchor=north west][inner sep=0.75pt]   [align=left] {$\displaystyle \widetilde{\sigma _{1}^{1}}$};
\draw (550,201) node [anchor=north west][inner sep=0.75pt]   [align=left] {$\displaystyle \widetilde{\sigma _{1}^{2}}$};
\draw (144,82) node [anchor=north west][inner sep=0.75pt]   [align=left] {$\displaystyle 1$};
\draw (171,251) node [anchor=north west][inner sep=0.75pt]   [align=left] {$\displaystyle 1$};
\draw (614,98) node [anchor=north west][inner sep=0.75pt]   [align=left] {$\displaystyle {s_1}$};

\end{tikzpicture}
}

        \end{figure}

    \end{enumerate}
    In each case, we further extend all ends not attached to a vertex over $p_1$ to $[p_1,p_2]$.
    \item We now take the permutation $\sigma_1\sigma_{\underline{n}}\tau\sigma_1\tau$ and consider the product
    \begin{equation}
        \sigma_{2}(\sigma_1\sigma_{\underline{n}}\tau\sigma_1\tau)\tau\sigma_{2}\tau.
    \end{equation}
    We proceed as in step (3) and create the corresponding vertices over $p_2$. The in-coming edge carrying the larger element of $\sigma_{2}$ becomes bold and the corresponding outgoing edge dashed. Let $s_{2}$ the larger element of $\sigma_{2}$. There exist unique $l$ and $c_{2}$ with
    \begin{equation}
        s_2=\sum_{i=1}^ln_i+c_{2}
    \end{equation}
    and $1\le c_{2}< n_{l+1}$. We assign the outgoing dashed edge counter $c_{2}$. 
    
\item     We proceed inductively as in step (4) for $\sigma_{i+1}(\sigma_{i}\cdots\sigma_1\sigma_{\underline{n}}\tau\sigma_1\cdots\sigma_i\tau)\sigma_{i+1}$ until $i=k$. For $i=k$, we obtain ends that are labeled by the cycles of the resulting permutation of cycle type $(\underline{m},\underline{m})$ and that project to $[p_k,\infty)$. The counter $c_i$ at step $i$ is obtained by observing that there are unique $l$ and $c_i$ with
\begin{equation}
    s_i=\sum_{j=1}^lc_i
\end{equation}
for $1\le c_2< n_{l+1}$. Note that we take $n_{k+1}=n$.

This gives a map between graphs $\tilde{\pi}\colon\tilde{\Gamma_1}\to\Gamma_2$.
    \item The conjugation by $\tau$ induces a natural involution $
    \tilde{\iota}\colon\tilde{\Gamma}_1\to\tilde{\Gamma}_1$ that respects the map $\tilde{\pi}$ (but not the colours or the counters).
    \item Next, for each edge $e$ of $\tilde{\Gamma}_1$, we replace the cycle label by the length of this cycle. We consider this cycle length as the weight $\omega(e)$ of $e$. 
\end{enumerate}

\end{construction}

\begin{example}
    In \cref{fig:montwistcover}, we illustrate the monotone twisted tropical cover associated to the factorisation $((\overline{1}\,2),(\overline{1}\,3),(1\,4))$ for $\underline{n}=(2,2)$ and $\underline{m}=(4)$.

    \begin{figure}
        \centering

\tikzset{every picture/.style={line width=0.75pt}} 

\begin{tikzpicture}[x=0.75pt,y=0.75pt,yscale=-1,xscale=1]

\draw [line width=3]    (31,114) -- (110.56,136.35) ;
\draw [line width=1.5]    (28.41,165.91) -- (118.79,140.06) ;
\draw [line width=1.5]    (118.79,140.06) -- (211.18,109.21) ;
\draw [line width=3]  [dash pattern={on 7.88pt off 4.5pt}]  (110.56,136.35) -- (207.77,159.62) ;
\draw [line width=3]    (113.97,85.95) -- (211.18,109.21) ;
\draw [line width=3]    (29,67) -- (113.97,85.95) ;
\draw [line width=1.5]    (117.38,185.47) -- (207.77,159.62) ;
\draw [line width=1.5]    (34,210) -- (117.38,185.47) ;
\draw [line width=3]    (211.18,109.21) -- (305,142) ;
\draw [line width=1.5]    (207.77,159.62) -- (305,142) ;
\draw [line width=3]  [dash pattern={on 7.88pt off 4.5pt}]  (305,142) -- (409,105.33) ;
\draw [line width=1.5]    (305,142) -- (408,171) ;
\draw  [dash pattern={on 0.84pt off 2.51pt}]  (29,140) -- (34,140) -- (500,140) ;

\draw (66.75,43.93) node [anchor=north west][inner sep=0.75pt]    {$2$};
\draw (59.93,95.63) node [anchor=north west][inner sep=0.75pt]    {$2$};
\draw (56.52,158.5) node [anchor=north west][inner sep=0.75pt]    {$2$};
\draw (71.87,208.07) node [anchor=north west][inner sep=0.75pt]    {$2$};
\draw (139.52,105.5) node [anchor=north west][inner sep=0.75pt]    {$2$};
\draw (127.93,146.63) node [anchor=north west][inner sep=0.75pt]    {$2$};
\draw (255.93,105.63) node [anchor=north west][inner sep=0.75pt]    {$4$};
\draw (257.93,152.63) node [anchor=north west][inner sep=0.75pt]    {$4$};
\draw (357.93,88.63) node [anchor=north west][inner sep=0.75pt]    {$4$};
\draw (356.5,160.5) node [anchor=north west][inner sep=0.75pt]    {$4$};
\draw (140,147) node [anchor=north west][inner sep=0.75pt]    {$( 2)$};
\draw (72,96) node [anchor=north west][inner sep=0.75pt]    {$( 1)$};
\draw (78,42) node [anchor=north west][inner sep=0.75pt]    {$( 1)$};
\draw (270,105.63) node [anchor=north west][inner sep=0.75pt]    {$( 1)$};
\draw (372,88) node [anchor=north west][inner sep=0.75pt]    {$( 2)$};

\end{tikzpicture}

        \caption{A twisted monotone cover of type $(1,(2,2),(4))$. The numbers in parentheses denote the counters, those without the edge weights.}
        \label{fig:montwistcover}
    \end{figure}
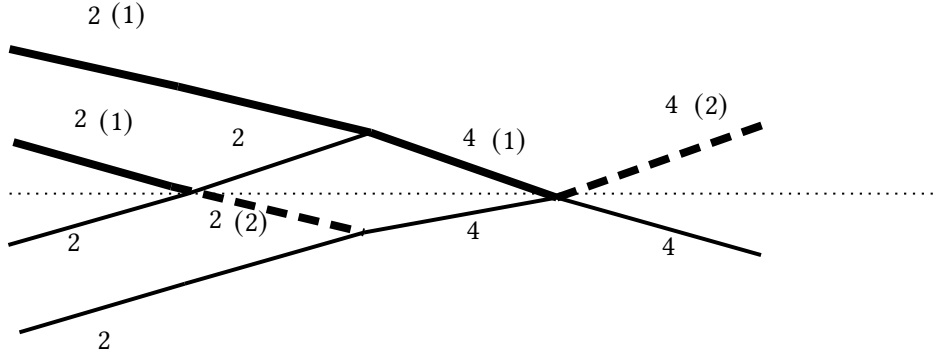
\end{example}

We observe the following lemma which is proved analogously to \cite[Lemma 3.8]{hahn2019monodromy}.

\begin{lemma}
    The covers obtained in \cref{constr-tropcov} are monotone twisted tropical covers in the sense of \cref{def-montwisttrop}.
\end{lemma}

Our next step is to define the multiplicity of monotone twisted tropical covers. For this, we need the following definition.

\begin{definition}
    Let $\pi\colon\Gamma_2\to\Gamma_1$ with involution $\iota$ be a monotone twisted tropical cover in the sense of \cref{def-montwisttrop}.

    We call an edge $e$ of $\Gamma_2$ a contributing edge if neither $e$ nor $\iota(e)$ are bold. Moreover, we denote by $\overline{E}(\Gamma_2)$ the set of contributing edges.
    Moreover, we call a vertex $v$ of $\Gamma_2$ a contributing vertex if $v$ has an incoming bold edge.

    Moreover, for any $4$-valent vertex $v$, we define its weight $\omega_v$ as the weight of its adjacent edges.

    Let $v$ be a contributing vertex of $\Gamma_2$. We define its vertex multiplicity as follows:
    \begin{enumerate}
        \item If $v$ is a $3$-valent vertex that performs a cut, we define its vertex multiplicity as $m(v)=C$.
        \item If $v$ is a $3$-valent vertex that performs a joint, we define its vertex multiplicity as $m(v)=J$.
        \item If $v$ is a $4$-valent vertex, i.e. $v$ performs a twist and assume that the $v$ is part of a chain of bold edges starting and $n_i$ and with outgoing non-bold edge $e$, then we define $m_v=T(\omega_v-c(e)+1)$.
    \end{enumerate}
\end{definition}

We are now ready to define the multiplicity of a monotone twisted tropical cover.

\begin{definition}
    Let $\pi\colon\Gamma_2\to\Gamma_1$ be a monotone twisted tropical cover. We then define its multiplicity to be
    \begin{equation}
        m(\pi)=\sqrt{\prod\omega(e)}\prod m(v),
    \end{equation}
   the first product runs over all contributing edges except out-ends and the second product over all contributing vertices.
\end{definition}

\begin{remark}
    We note that there are no automorphisms of monotone twisted tropical covers, since any such notion would need to respect the colourings. The colourings however prohibit the existence of any automorphism.
\end{remark}

\begin{proposition}
    Let $\pi\colon\Gamma_2\to\Gamma_1$ be a monotone twisted tropical cover. Then, $m(\pi)$ is the weighted count of all monotone factorisations producing $\pi$ via \cref{constr-tropcov}.
\end{proposition}

\begin{proof}
    The proposition follows immediately from the same cut-and-join analysis used to derive \cref{th:cut-and-join-monotone}.
\end{proof}

In total, we have proved the following result

\begin{theorem}
\label{thm-refinedtropcorr}
    Let $g$ be a non-negative integer $\mu$ and $\lambda$ partitions of the same positive integer $d$. Then, we have
    \begin{equation}
       N_g^a\binom{n_1,\dots,n_p}{m_1,\dots,m_q|r}=\frac{1}{C^qJ^p\prod(2j)^{n_j}}\sum m(\pi)
    \end{equation}
    where the sum runs over all monotone twisted tropical covers $\pi$ of type $(g,\underline{n},\underline{m},r,a)$.
\end{theorem}

\begin{example}
    The multiplicity of the monotone twisted tropical cover of type $(1,(2,2),(4),1,2)$ in \cref{fig:montwistcover} is $6T^2J$. Indeed, the factorisations giving this cover are
    \begin{enumerate}
        \item $((\overline{1}\,2),(\overline{1}\,3),(1\,4))$
        \item $((\overline{1}\,2),(\overline{1}\,3),(2\,4))$
        \item $((\overline{1}\,2),(\overline{1}\,3),(\overline{3}\,4))$
        \item $((\overline{1}\,2),(\overline{2}\,3),(1\,4))$
        \item $((\overline{1}\,2),(\overline{2}\,3),(2\,4))$
        \item $((\overline{1}\,2),(\overline{2}\,3),(\overline{3}\,4))$
    \end{enumerate}
    In total, we have 6 monotone factorisations, each providing two twists and one join.
\end{example}

\subsection{Towards the piecewise polynomiality of refined monotone Hurwitz numbers}
Recalling the piecewise polynomiality result for classical double Hurwitz numbers, we stated in \cref{thm-GJV}, the same statement (w.r.t. the same chamber structure) was proved to be true for monotone double Hurwitz numbers in \cite{goulden2016toda}. Employing similar techniques as in \cref{subsec-poly}, one may derive a weaker polynomiality statement for refined monotone Hurwitz numbers based on the tropical correspondence theorem in \cref{thm-refinedtropcorr}. A similar tropical approach was used in \cite{hahn2019monodromy} for the piecewise polynomiality of (non--deformed) monotone double Hurwitz numbers. We briefly summarise the strategy in the following two steps and refer to \cite{hahn2019monodromy} for details:
\begin{enumerate}
    \item The idea is the same as in \cref{subsec-poly} in that one parametrises the possible weights of a given monotone twisted tropical cover $\pi$ as lattice points of a polytope. However, we now also have to parametrise the counters. In total, we obtain the inequalities $\omega(e)\ge0$ and $c(e)\ge n_{i-p}$ (as in \cref{def-montwisttrop} (10)) and $c(e)\ge0$. This yields a polytope $P$.
    \item Using the typical Ehrhart theory, one obtains that summing $m(\pi)$ over all lattice points in $P$ yields piecewise quasipolynomiality with respect to a finer hyperplane arrangement than the resonance arrangement. In the non--deformed case, the additional hyperplanes are not necessary, however this is not obvious from the tropical picture. A similar observation was made in the context of tropicalising so-called pruned Hurwitz numbers in \cite{fitzgerald2023combinatorics}.
\end{enumerate}

We state without proof the following result.

\begin{proposition}
    \label{prop-refinpoly}
    The refined monotone Hurwitz numbers $N_g\binom{n_1,\dots,n_p}{m_1,\dots,m_q}$ are piecewise quasipolynomial in the variables $n_1,\dots,n_p$ and $m_1,\dots,m_q$.
\end{proposition}

We note that a different tropical approach used in \cite{hahn2022tropical,hahn2020wall} that employed the Fock space formalism did recover the full polynomiality for monotone double Hurwitz numbers as in \cite{goulden2016toda}. It would be interesting to see whether similar techniques work for the refinement. We leave this as an open problem for future work.

%

\printbibliography
\newpage
\appendix

\section{Elementary symmetric functions in preliminary refined Jucys-Murphy elements applied to the type indicators} \label{sec:pdata}

The table below contains the results of application of elementary symmetric functions $\mathrm e_k$ evaluated on preliminary refined odd Jucyc-Murphy $\mathbf X_i$ elements to the type indicators $\mathcal D_\lambda$ with $\lambda\vdash n$ for small values of $n.$ For the corresponding action of the elementary symmetric functions evaluated on the refined Jucys-Murphy elements $\mathbb X_i$, it is enough to set $J = 1,$ and $C = CJ.$

\begin{center}
\begin{tabular}{cccl} \toprule
$n$ & function  & type  & result \\ \midrule
$2$ &  $\mathrm e_1$ & $\mathcal D_{11}$ & $J\mathcal D_2$\\
  &  & $\mathcal D_2$ & $T\mathcal D_2 + 2C\mathcal D_{11}$\\ \midrule
$3$ & $\mathrm e_1$ & $\mathcal D_{111}$ & $J\mathcal D_{21}$ \\
 &  & $\mathcal D_{21}$ & $3J\mathcal D_{3} + T\mathcal D_{21} + 6C\mathcal D_{111}$\\
 &  & $\mathcal D_{3}$ & $3T\mathcal D_{3} + 4C\mathcal D_{21}$\\
 & $\mathrm e_2$ & $\mathcal D_{111}$ & $J^2\mathcal D_{3}$\\
& & $\mathcal D_{21}$ & $3JT\mathcal D_{3} + 4CJ\mathcal D_{21}$\\
 &  & $\mathcal D_{3}$ & $(2CJ + 2T^2)\mathcal D_{3} + 4CT\mathcal D_{21} + 8C^2\mathcal D_{111}$\\ \midrule
$4$ & $\mathrm e_1$ & $\mathcal D_{1111}$ & $J\mathcal D_{211}$\\
 &  & $\mathcal D_{211}$ & $3J\mathcal D_{31} + 2J\mathcal D_{22} + T\mathcal D_{211} + 12C\mathcal D_{1111} $\\
 &  & $\mathcal D_{22}$ & $2J\mathcal D_{4} + 2T\mathcal D_{22} + 2C\mathcal D_{211} $\\
&  & $\mathcal D_{31}$ & $4J\mathcal D_{4} + 3T\mathcal D_{31} + 8C\mathcal D_{211} $\\
 &  & $\mathcal D_{4}$ & $6T\mathcal D_{4} + 6C\mathcal D_{31} + 8C\mathcal D_{22}$\\
 & $\mathrm e_2$ & $\mathcal D_{1111}$ & $J^2\mathcal D_{22} + J^2\mathcal D_{31}$\\
& & $\mathcal D_{211}$ & $6J^2\mathcal D_{4} + 3JT\mathcal D_{31} + 2JT\mathcal D_{22} + 10CJ\mathcal D_{211}$\\
& & $\mathcal D_{22}$ & $5JT\mathcal D_{4} + 6CJ\mathcal D_{31} + (4CJ + T^2)\mathcal D_{22} + 2CT\mathcal D_{211} + 12C^2\mathcal D_{1111}$\\
 &  & $\mathcal D_{31}$ & $12JT\mathcal D_{4} + (14CJ + 2T^2)\mathcal D_{31} + 16CJ\mathcal D_{22} + 8CT\mathcal D_{211} + 32C^2\mathcal D_{1111}$\\
 &  & $\mathcal D_{4}$ & $(10CJ + 11T^2)\mathcal D_{4} + 18CT\mathcal D_{31} + 20CT\mathcal D_{22} + 24C^2\mathcal D_{211}$\\
 & $\mathrm e_3$ & $\mathcal D_{1111}$ & $J^3\mathcal D_{4}$\\
&  & $\mathcal D_{211}$ & $6J^2T\mathcal D_{4} + 6CJ^2\mathcal D_{31} + 8CJ^2\mathcal D_{22}$\\
 & & $\mathcal D_{22}$ & $(2CJ^2 + 3JT^2)\mathcal D_{4} + 6CJT\mathcal D_{31} + 4CJT\mathcal D_{22} + 8C^2J\mathcal D_{211}$\\
 &  & $\mathcal D_{31}$ & $(8CJ^2 + 8JT^2)\mathcal D_{4} + 12CJT\mathcal D_{31} + 16CJT\mathcal D_{22} + 16C^2J\mathcal D_{211}$\\
 &  & $\mathcal D_{4}$ & $(14CJT + 6T^3)\mathcal D_{4} + (12C^2J + 12CT^2)\mathcal D_{31} + (8C^2J + 12CT^2)\mathcal D_{22}$ \\
& & & $+ 24C^2T\mathcal D_{211} + 48C^3\mathcal D_{1111}$\\
\bottomrule
\end{tabular}
\end{center}

\section{Simple refined monotone Hurwitz numbers}\label{sc:simple-mon}

In this section, we indicate a single example that the cut-and-join equations for the refined Hurwitz numbers might be solved by the \emph{refined topological recursion} (see~\cite{KO22,Ch-Dolega-Osuga2}). Namely, we use the cut-and-join equation of Theorem~\ref{th:cut-and-join-monotone} specialized to the case $\nu = 1^n$ to produce the Virasoro constraints for the corresponding partition function, allowing us to compute it recursively, constructing an algebraic recursion, that imply the refined recursion. The method we use is close to~\cite{Ch-Dolega-Osuga2}. 

Recall, that by Theorem~\ref{th:rec_single_monotone1} the single monotone refined Hurwitz numbers 
 $\mathcal N_g(m_1,\ldots,m_q)$ satisfy the recurrence
\begin{gather}
m_r\mathcal N_g(m_1,\ldots,m_q) = \sum_{j \ne r}(m_r +m_j) \mathcal N_g(m_1,\ldots,\hat m_j,\ldots m_r + m_j,\ldots,m_q) + \\  m_r(m_r - 1 )T \mathcal N_{g-\frac 12}(m_1,\ldots,m_q)
+\sum_{\alpha+ \beta = m_r} 2CJ \alpha \beta \Biggl( \mathcal N_{g-1}(m_1,\ldots,\hat m_r, \alpha, \beta, \ldots,m_q) + \Biggr. \\ \Biggl.   \sum_{g_1 + g_2 = g} \sum_{
\begin{smallmatrix}  I_1 \cup I_2 = \{1,\ldots,q\},\ I_1\cup I_2 = \{r\}
\end{smallmatrix}
}  \mathcal N_{g_1}\left(m_{I_1}(m_r \mapsto \beta)\right)\mathcal N_{g_2}\left(m_{I_2}(m_r \mapsto \alpha)\right)\Biggr),
\end{gather}

for any $r = 1,\ldots,q$ and the initial condition $\mathcal N_0(m) = \delta_{m,1}\frac 1{2CJ}.$

The remaining part of the section reads in line with~\cite{Ch-Dolega-Osuga2}. We gather the numbers $\mathcal N_g(m_1,\ldots,m_q)$ into the generating function:
\[
\mathcal H_{g,q} = \sum_{m_1,\ldots,m_q \ge 1} u^{2g - 2 + \sum m_j + q}t^{\sum{m_j}}\mathcal N_g(m_1,\ldots,m_q)\frac{ p_{m_1}\cdots p_{m_q}}{q!} \in \mathbb C(C,J,T)[[u,t;p_1,p_2,\ldots]],
\]
where $u$ is a bookkeeping parameter for the number of transpositions used in the monotone factorisation, and $t$ is a bookkeeping parameter for the sum of $m_j.$ We denote \[\mathcal H = \sum_{\begin{smallmatrix} g \in \frac 12 \mathbb N\cup \{0\} \\ q \in \mathbb N \end{smallmatrix}} \mathcal H_{g,q}.\]
Following~\cite{KZ} we notice, that the recursion of Theorem~\ref{th:rec_single_monotone1} is equivalent to the following set of constraints for the function $\mathcal H^\circ = \exp{\mathcal  H}$: for all $i \in \mathbb N$ we have $L_i \mathcal H^\circ = 0,$ where $L_i$ is the differential operator
\[
L_i = -\frac iu \frac{\partial}{\partial p_i} + 2CJ\sum_{\alpha+\beta = i} \alpha\beta\frac{\partial^2}{\partial p_\alpha \partial p_\beta} + \sum_{\alpha\ge 1} (i+\alpha)p_\alpha \frac{\partial}{\partial p_{i+\alpha}} + i(i-1)T\frac{\partial}{\partial p_i} + \frac{t\delta_{i,1}}{2uCJ},\quad i \ge 1,
\]
where $\delta_{i,j}$ is the Kronecker's delta. A similar set of constraints appears in~\cite{Bonzom-Chapuy-Dolega} in the study of single $b$-monotone Hurwitz numbers, which indicates that the specialization $T = b, CJ = \frac {1+b}2$ recovers the monotone $b$-Hurwitz numbers.

As it is mentioned in~\cite{Bonzom-Chapuy-Dolega}, one can use these constraints to produce the evolution equation. Multiply the  constraint $L_i$ by $p_i$ and sum over all $i$ to get:

\begin{eqnarray*}
\sum_{i = 1}^\infty p_i L_i = -\frac 1u \sum_{i = 1}^\infty ip_i \frac \partial{\partial p_i} + 2\mathcal{BL}^\circ + \frac{tp_1}{2uCJ}.
\end{eqnarray*}
This differential operator annihilates the function $\mathcal H^\circ$. Notice, that due to the homogeneousness of $\mathcal H^\circ,$ the expression $-\frac 1u \sum_{i\ge 1} ip_i \frac \partial{\partial p_i} \mathcal H^\circ$ equals $-\frac tu\frac \partial{\partial t} \mathcal H^\circ$. The operator $\mathcal{BL}^\circ_b$ is explicitly given by

\begin{equation}\label{eq:Laplace-Beltrami}
\mathcal{BL}^\circ = \frac{1}{2}\left( \sum_{i,j \ge 1}(i+j)p_i p_j\frac{\partial}{\partial p_{i+j}}
		+ 2CJ\sum_{i,j \ge 1} ij p_{i+j} \frac{\partial^2}{\partial p_i\partial p_j} + 
		T\sum_{i \ge 1} i(i-1)p_{i}\frac{\partial}{\partial p_{i}}\right).
\end{equation}
  We call it the \emph{refined Laplace-Beltrami operator} (the coefficient $\frac 12$ is due to historical reasons).

  Summing everything up, we obtain the following evolution equation for the function $\mathcal H^\circ$ (cf. the evolution equation of~\cite{Bonzom-Chapuy-Dolega}):
  $$\frac 1{2u} \left(t \frac \partial{\partial t} - \frac{tp_1}{2CJ}\right) \mathcal H^\circ = \mathcal{BL}^\circ(\mathcal H^\circ).$$


Following Kazarian and Zograf~\cite{KZ} we introduce a formal operator
\( \delta_x = \sum_{i \ge 1} i x^{i-1} \frac{\partial}{\partial p_i}\) and sum up the results of the action of the constraints with the appropriate factors on the partition function to get the \emph{master equation}:
\begin{eqnarray*}
- \frac{x}{u}\delta_x \mathcal{H} +2CJx^2  \big( \delta_x^2 \mathcal{H} + (\delta_x \mathcal{H})^2 \big) + \delta_y^{-1} d_y \Big( \frac{xy}{x-y} (\delta_x \mathcal{H} - \delta_y \mathcal{H} \big)\Big)\\
+ Tx^2 \frac \partial{\partial x} \delta_x \mathcal{H} + \frac{tx}{2uCJ} = 0.
\end{eqnarray*}

The most non-trivial part of turning the  constraints into the master equation is
\begin{eqnarray*}
    \sum_{i\ge 1} x^i \sum_{j \ge 1} p_j (i + j)\frac {\partial \mathcal H}{\partial p_{i+j}} = \delta_y^{-1} d_y \sum_{j \ge 1} \sum_{i \ge 1} x^i y^j (i + j) \frac {\partial \mathcal H}{\partial p_{i+j}}
    = \delta_y^{-1} d_y \sum_{k \ge 2} \sum_{j = 1 }^{k-1} x^{k-j} y^j k \frac{\partial \mathcal H}{\partial p_k} \\= \delta_y^{-1} d_y \sum_{k\ge 2} xy \frac{x^{k-1} - y^{k-1}}{x-y} k \frac {\partial \mathcal H}{\partial p_k} = \delta_y^{-1} d_y \Big( \frac{xy}{x-y}\big(\delta_x \mathcal H - \delta_y \mathcal H \big) \Big).
\end{eqnarray*}

The master equation can be solved recursively, based on the value of the difference of the degrees of the variables $u$ and $t$, and the total degree in $p$ variables of the terms of the generating function $\mathcal H$. 

The base case corresponds to the total degree in $p$ variables equal 1, and the difference of the degrees of the variables $u$ and $t$ equals $-1$ which is the case that corresponds to  $(g,q) = (0,1)$. The terms of the master equation that contribute to this case are:
\[
-\frac xu \delta_x \mathcal H_{0,1} + 2CJx^2 (\delta_x \mathcal H_{0,1})^2 + \frac{tx}{2uCJ} = 0.
\]
Solving it we get
\[
\delta_x \mathcal H_{0,1} = \frac{1 - \sqrt{1 - 4\,utx}}{uCJx}.
\]
Next two terms emerge for the difference between the degrees of $u$ and $t$ equals 0. In this case, we have the term $\mathcal H_{\frac 12, 1}$, that do not have an oriented counterpart, and the term $\mathcal H_{0,2}$. For the former, we have the equation:
\[
-\frac xu \delta_x \mathcal H_{\frac 12,1} + 4CJ x^2 \delta_x \mathcal H_{\frac 12, 1}\, \delta_x \mathcal H_{0,1} + Tx^2 \frac \partial{\partial x} \delta_x \mathcal H_{0,1} = 0
\]
that produces \[
\delta_x \mathcal H_{\frac 12, 1} = 
\frac{{\left( 1- 2 \, t u x - \sqrt{1 -4 \, t u x }  \right)} T}{4 \, {\left( 1- 4 \, t u x \right)} {CJ} x}.
\]
For the latter, after the application of $\delta_y$, the equation reads
\[
-\frac xu \delta_x \delta_y \mathcal H_{0,2} + 4CJ\delta_x \mathcal H_{0,1}\,  \delta_x \delta_y \mathcal H_{0,2} + d_y \Big( \frac{xy}{x-y}\left(\delta_x \mathcal H_{0,1} - \delta_y \mathcal H_{0,1}\right)  
\Big) = 0,
\]
or
\[
\delta_y \delta_x \mathcal H_{0,2} =  \frac{1 - 2 \, t u x - 2 \, t u y - \sqrt{1  -4 \, t u x } \sqrt{1 -4 \, t u y } }{4 \, \sqrt{1 -4 \, t u x } \sqrt{1 -4 \, t u y } {CJ} {\left(x - y\right)}^{2}}.
\]
Notice, that the variable change
\[
z = \frac 1{\sqrt{1 - 4\, t u x}},\quad w = \frac 1{\sqrt{1 - 4\, t u y}}
\]
turns the computed values into rational expressions:
\begin{eqnarray*}
\delta_x \mathcal H_{0,1} = \frac {tz}{CJ(z + 1)},\\
\delta_x \mathcal H_{\frac 12, 1} = \frac{ t u {\left(z - 1\right)} z^{2}T}{2{CJ} {\left(z + 1\right)}},\\
\delta_x \delta_y \mathcal H_{0,2} = \frac{2 \, t^{2} u^{2} w^{3} z^{3}}{{CJ} {\left(w + z\right)}^{2}}.
\end{eqnarray*}

The general equation for $\mathcal H_{g,q}$  with $2g - 2 + q \ge 1$ reads:

\begin{eqnarray*}
    \Big( \frac xu - 4CJx^2\delta_x \mathcal H_{0,1}\Big) \delta_x \mathcal H_{g,q} = 2CJx^2\left( \delta_x^2 \mathcal H_{g-1,q+1} + \sum_{\begin{smallmatrix}g_1 + g_2 = g\\ q_1 + q_2 = q+1  \end{smallmatrix}}^{'} \delta_x \mathcal H_{g_1,q_1} \delta_x \mathcal H_{g_2,q_2}\right)\\
    + \delta_y^{-1} d_y \left(\frac{xy}{x-y}\left( \delta_x \mathcal H_{g,q-1} - \delta_y \mathcal H_{g,q-1} \right) \right) + Tx^2 \frac \partial{\partial x} \delta_x \mathcal H_{g-\frac 12, q},
\end{eqnarray*}
where the $'$ on top of the sum means that the terms with $(g_1,q_1) = (0,1)$ or $(g_2,q_2) = (0,1)$ are not involved in the summation.

For the study of the structural properties of the solutions, we will rewrite this formula in $z$ variables, isolating the contribution of the terms that have already been investigated.

We introduce the following objects:
\begin{itemize}
\item The form $\eta = \left(\frac 1{ux} - 4CJ \delta_x \mathcal H_{0,1} \right) dx = \frac{2\,dz}{uz^2(z^2 - 1)}$, and its dual vector field $ \frac 1\eta = \frac {uz^2(z^2 - 1)}2 \frac \partial{\partial z}$. Compared to the standard topological recursion, the vector field $\frac 1\eta$ plays the role of the recursion kernel.
\item Studying the expression 
\[
\frac {dx}{\eta}\left( 4CJ \delta_x \mathcal H_{0,2} \, \delta_x \mathcal H_{g,n-1} + \delta_y^{-1} d_y \left( \frac{y}{x(x-y)} \left(\delta_x \mathcal H_{g,n-1} - \delta_y \mathcal H_{g,n-1} \right) \right) \right)
\]
we can rewrite it as follows. Notice, that
\[
\delta_x \delta_y \mathcal H_{0,2} = -d_y \frac{\sqrt{1 -4 \, t u x } - \sqrt{1 -4 \, t u y }}{4 \, \sqrt{1 -4 \, t u x } {CJ} {\left(x - y\right)}},
\]
so that
\[
\left( 4CJ \delta_x \delta_y  \mathcal H_{0,2}  +  d_y \frac{y}{x(x-y)} \right)\,dx\,dy =  \frac{2 \, {\left(w^{2} + z^{2}\right)}\, dz\,dw}{ {\left(w^2 - z^2\right)}^{2}} = d_w \frac{2w\,dz}{z^2 - w^2}.
\]
And
\[
\frac {(dx)^2}{\eta(z)}  d_y\left( \frac{y}{x(x-y)}\delta_y \mathcal H_{g,q-1} \right) =   d_w \left(\frac{2w\,dz}{\eta(w)(z^2 - w^2)}\, \left(dy\, \delta_y \mathcal H_{g,q-1}\right)\right),
\]
where $\eta(z)$ and $\eta(w)$ mean the evaluation of the form $\eta$ in the coordinate $z$ or $w$, respectively. The expression
\[
d_w\left( \frac{2w\,dz}{z^2 - w^2}\left(\frac {\delta_x \mathcal H_{g,q-1}}{\eta(z)} - \frac{\delta_y \mathcal H_{g,q-1}}{\eta(w)}\right) \right)
\]
is the counterpart of the result of the pairing of  $(g,q-1)$-term with the Bergman kernel (that plays the role of $(0,2)$-term) in the standard topological recursion.
\item Lastly, we treat the contribution from $(g-\frac 12, q)$-term as follows. We have
\[
T\frac {dx}{\eta}\left(\frac{\partial}{\partial x} \delta_x \mathcal H_{g - \frac 12,n} + \frac{4CJ}{T}\delta_x \mathcal H_{\frac 12,1}\delta_x \mathcal H_{g - \frac 12,q}\right), 
\]
which can be treated as a multiple of the result of a covariant derivation of $\delta_x \mathcal H_{g-\frac 12,q}$ in the direction of the vector field $\frac 1\eta$. We rewrite it in terms of an application of the derivative to the form $\delta_x \mathcal H_{g - \frac 12,q}\, dx$ expressed in the variable $z$:
\[
\mathcal U(z)\, dz \mapsto T\nabla_{\frac 1\eta} \mathcal U(z)\,dz =  T \left(\frac{dz}{\eta(z)} \frac {\partial U(z)}{\partial z} + \frac{2(2z+1)\,dz}{\eta(z) z(z+1)} \mathcal U(z)\right)\,dz.
\]

\item The initial data for the recursion can be computed from the master equation. Introducing the variables $x_1,x_2,x_3$ and the corresponding $z_1,z_2,z_3,$ we have:
\begin{eqnarray*}
    dx_1\,dx_2\,dx_3 \delta_{x_1} \delta_{x_2} \delta_{x_3} \mathcal H_{0,3} = \frac{u}{2CJ}\,dz_1\,dz_2\,dz_3 = \mathcal U_{0,3},\\
    dx_1\,dx_2 \delta_{x_1} \delta_{x_2} \mathcal H_{\frac 12,2} =\frac{{\left(z_2^{4} + 3 \, z_2^{3} z_1 + 3 \, z_2^{2} z_1^{2} + 3 \, z_2 z_1^{3} + z_1^{4} - z_2^{3} - 3 \, z_2^{2} z_1 - 3 \, z_2 z_1^{2} - z_1^{3} - z_2 z_1\right)} T u}{{2CJ} {\left(z_2 + z_1\right)}^{3}}\,dz_1\,dz_2 \\
    = \frac{Tu}{2CJ}\left(z_1 + \frac{2 \,  z_2^{3} - z_2}{ {\left(z_2 + z_1\right)}^{2}} - \frac{ z_2^{4} - z_2^{2}}{ {\left(z_2 + z_1\right)}^{3}}\right)\,dz_1\,dz_2= \mathcal U_{\frac 12, 2}\\
    dx_1 \delta_{x_1} \mathcal H_{1,1} 
    = \frac{u(z_1 - 1)}{16CJ}\left(5T^2z_1 - 3T^2 + 2CJz_1 + 2CJ\right)\,dz_1 = \mathcal U_{1,1}.
\end{eqnarray*}

\end{itemize}
Notice, that contrary to the standard topological recursion, the computed initial data differentials has poles not only in $z = \infty$, but also along the diagonals $z_1 = -z_2$. Also, these forms are not necessarily odd with respect to the Galois involution $z\mapsto -z$. Nevertheless, we have the following rational recursion for the correlation differential $\mathcal U_{g,q} = dx\,\delta_x \mathcal H_{g,q}$ with $2g - 2 + q \ge 2$:
\begin{theorem}\label{th:toprec} The correlation differentials $\mathcal U_{g,q}$ for $2g - 2 + q \ge 2$ satisfy the recursion
\begin{eqnarray*}
    \mathcal U_{g,q} = \frac{2CJ}{\eta}\left(dx\, \delta_x \mathcal U_{g-1, q + 1} + \sum_{\begin{smallmatrix}g_1 + g_2 = g\\ q_1 + q_2 = q+1 \\ 2g_1 -2 + q_1 \ge 1 \\ 2g_2 -2 + q_2 \ge 1 \end{smallmatrix}}\mathcal U_{g_1,q_1} \mathcal U_{g_2,q_2}  \right) + T\nabla_{\frac 1\eta} \mathcal U_{g-\frac 12,q}\\ + dz\,\delta_y^{-1} d_w \left(\frac {2w}{z^2 - w^2}\left(\frac{\mathcal U_{g,q-1}(z)}{\eta(z)} - \frac{\mathcal U_{g,q-1}(w)}{\eta(w)} \right)\right).
\end{eqnarray*}
\end{theorem}
This form of the recursion indicates that all correlation differentials are residue-free, can have poles at $z_j = \infty$ and along the diagonals $z_i = -z_j$ only, by the properties of the refined topological recursion outlined by  K. Osuga~\cite{Osu23}. Given this, showing that the correlation differentials for simple monotone refined Hurwitz numbers satisfy the refined topological recursion of~\cite{KO22} can be performed using the methods of~\cite{Ch-Dolega-Osuga2}.





\end{document}

%% file: packages.tex
\usepackage[margin=3cm]{geometry}
\usepackage{color}               
\usepackage{faktor}

\usepackage{spverbatim}
\usepackage{float}
\usepackage[pdftex,hyperfootnotes=false,colorlinks=false,hypertexnames=false]{hyperref}
\usepackage[normalem]{ulem}

\usepackage{amsmath, amssymb, anyfontsize,enumerate,stmaryrd,mathtools, thmtools,tikz,txfonts,nccmath}

\usepackage[oldstyle]{libertine}%
\usepackage[T1]{fontenc}
\usepackage[final]{microtype}
\usepackage[utf8]{inputenc}
\usepackage{csquotes}
\makeatletter
\renewcommand*\libertine@figurestyle{LF}
\makeatother
\usepackage[libertine,libaltvw,liby]{newtxmath}
\makeatletter
\renewcommand*\libertine@figurestyle{OsF}
\makeatother
\usepackage[noerroretextools,backend=biber,maxalphanames=5,giveninits,sorting=nyt,%
maxbibnames=99]{biblatex}
\usepackage{tikz-cd}
\usepackage[noabbrev,capitalise]{cleveref}
\usepackage{autonum}
\usepackage{braket}
\usepackage{todonotes}

%% file: environments.tex
\theoremstyle{plain}
    \newtheorem{theorem}{Theorem}[section]
    \newtheorem*{theorem*}{Theorem}
    \newtheorem{construction/theorem}[theorem]{Construction/Theorem}
    \newtheorem{corollary}[theorem]{Corollary}
    \newtheorem{lemma}[theorem]{Lemma}
    \newtheorem{proposition}[theorem]{Proposition}
    
    \newtheorem{conjecture}[theorem]{Conjecture}
\theoremstyle{definition}
    \newtheorem{remark}[theorem]{Remark}
    
    \newtheorem{example}[theorem]{Example}
    \newtheorem*{definition*}{Definition}
    \newtheorem{definition}[theorem]{Definition}
    \newtheorem{construction}[theorem]{Construction}

\newcommand{\comment}[1]{}